\documentclass[11pt,reqno]{amsart} 
\usepackage{amsthm,mathrsfs,fancyhdr,subfigure,amssymb,amsmath}
\usepackage{graphicx} 
\usepackage[noadjust,sort]{cite}
\usepackage[normalem]{ulem}
\usepackage{xpunctuate}
\usepackage{url}
\usepackage{comment}
\usepackage[foot]{amsaddr}
\usepackage{longtable,booktabs,array}
\usepackage[normalem]{ulem}
\usepackage{amssymb,amsmath,bbm}
\usepackage{enumitem}
\usepackage{esint}
\usepackage{soul}
\usepackage[table,xcdraw]{xcolor}

\usepackage{pifont}
\usepackage{float}
\usepackage{url}
\usepackage{varwidth}
\usepackage{mathrsfs}
\usepackage{tikz}
\usepackage{bbm}

\usepackage{rotate}
\usepackage{needspace}
\usepackage[margin=1.05in]{geometry}
\usepackage{rotate}
\usepackage{doi}
\hypersetup{hidelinks}

\makeatletter
\def\paragraph{\@startsection{paragraph}{4}%
  \z@\z@{-\fontdimen2\font}%
  {\normalfont\bfseries}}
\makeatother

\newcommand{\vol}{\operatorname{vol}}
\renewcommand\le{\leqslant}
\renewcommand\ge{\geqslant}
\renewcommand\leq{\leqslant}
\renewcommand\geq{\geqslant}

\allowdisplaybreaks

\numberwithin{equation}{section}
\numberwithin{figure}{section}

\theoremstyle{plain}
\newtheorem{theorem}{Theorem}[section]
\newtheorem*{theorem*}{Theorem}
\newtheorem{lemma}[theorem]{Lemma}
\newtheorem{proposition}[theorem]{Proposition}
\newtheorem{corollary}[theorem]{Corollary}

\newtheorem{definition}[theorem]{Definition}

\theoremstyle{definition}
\newtheorem*{definition*}{Definition}

\newcommand\xqed[1]{%
    \leavevmode\unskip\penalty9999 \hbox{}\nobreak\hfill
    \quad\hbox{#1}}

\newtheorem{exampleqqq}[theorem]{Example}
\newenvironment{example}{\begin{exampleqqq}}
  {\xqed{$\triangle$}
  \end{exampleqqq}}

\newtheorem{remarkqqq}[theorem]{Remark}
\newenvironment{remark}{\begin{remarkqqq}}
  {\xqed{$\triangle$}\end{remarkqqq}}

\theoremstyle{remark}

\newcounter{dummy}
\makeatletter
\newcommand\myitem[1][]{\item[#1]\refstepcounter{dummy}\def\@currentlabel{#1}}
\makeatother

\newenvironment{romenumerate}[1][-10pt]{%
\addtolength{\leftmargini}{#1}\begin{enumerate}
 \renewcommand{\labelenumi}{\textup{(\roman{enumi})}}%
 \renewcommand{\theenumi}{\labelenumi}%
 }{\end{enumerate}}

\newcounter{oldenumi}

\newcounter{thmenumerate}
\newenvironment{thmenumerate}
{\setcounter{thmenumerate}{0}%
 \def\item{\par
 \refstepcounter{thmenumerate}\textup{(\roman{thmenumerate})\enspace}}
}
{}

\newcounter{xenumerate}

\newcommand{\refT}[1]{Theorem~\ref{#1}}
\newcommand{\refTs}[1]{Theorems~\ref{#1}}
\newcommand{\refC}[1]{Corollary~\ref{#1}}

\newcommand{\refL}[1]{Lemma~\ref{#1}}
\newcommand{\refLs}[1]{Lemmas~\ref{#1}}
\newcommand{\refR}[1]{Remark~\ref{#1}}

\newcommand{\refS}[1]{Section~\ref{#1}}

\newcommand{\refSS}[1]{Section~\ref{#1}}

\newcommand{\refProp}[1]{Proposition~\ref{#1}}
\newcommand{\refD}[1]{Definition~\ref{#1}}
\newcommand{\refE}[1]{Example~\ref{#1}}

\newcommand{\refApp}[1]{Appendix~\ref{#1}}

\newcommand\MC{\operatorname{MC}}
\newcommand\C{\operatorname{C}}
\newcommand\MS{\operatorname{MS}}
\newcommand\wpure{_{\mathrm{pure}}}
\newcommand\pure{{\mathrm{pure}}}
\newcommand\mix{{\mathrm{mix}}}
\newcommand\nupure{{\nu,\pure}}
\newcommand\kkpure{{\kk,\pure}}
\newcommand\kkmix{{\kk,\mix}}

\newcommand\fra{{\mathrm{frac}}}  
\newcommand\kkfra{{\kk,\fra}}
\newcommand\comp{^{\textsf c}}

\newcommand\pfitemx[1]{\par#1:}
\newcommand\pfitemref[1]{\pfitemx{\ref{#1}}}

\newcommand{\sumi}{\sum_{i=1}^\infty}

\newcommand{\sumikk}{\sum_{i=1}^\kk}
\newcommand{\sumjkk}{\sum_{j=1}^\kk}
\newcommand{\sumlkk}{\sum_{\ell=1}^\kk}
\newcommand{\sumkkk}{\sum_{k=1}^\kk}
\newcommand{\sumglkk}{\sum_{\glm=1}^\kk}
\newcommand{\sumgllkk}{\sum_{\gl=1}^\kk}
\newcommand{\sumlgl}{\sum_{\ell=1}^\glm}
\newcommand{\sumlnu}{\sum_{\ell=1}^\nu}

\newcommand{\glm}{\gl} 
\newcommand{\prodin}{\prod_{i=1}^n}
\newcommand{\prodinu}{\prod_{i=1}^\nu}

\newcommand\set[1]{\ensuremath{\{#1\}}}
\newcommand\bigset[1]{\ensuremath{\bigl\{#1\bigr\}}}

\newcommand\xpar[1]{(#1)}
\newcommand\bigpar[1]{\bigl(#1\bigr)}
\newcommand\Bigpar[1]{\Bigl(#1\Bigr)}
\newcommand\biggpar[1]{\biggl(#1\biggr)}

\newcommand\bigsqpar[1]{\bigl[#1\bigr]}
\newcommand\sqpar[1]{[#1]}

\newcommand\biggsqpar[1]{\biggl[#1\biggr]}

\newcommand\cpar[1]{\{#1\}}
\newcommand\bigcpar[1]{\bigl\{#1\bigr\}}

\newcommand\bigabs[1]{\bigl\lvert#1\bigr\rvert}
\newcommand\Bigabs[1]{\Bigl\lvert#1\Bigr\rvert}
\newcommand\biggabs[1]{\biggl\lvert#1\biggr\rvert}

\newcommand\norm[1]{\lVert#1\rVert}

\def\rompar(#1){\textup(#1\textup)}

\newcommand\innprod[1]{\langle#1\rangle}

\def\xexp(#1){e^{#1}}

\newcommand\frax[1]{\{#1\}}

\newcommand\ntoo{\ensuremath{{n\to\infty}}}
\newcommand\Ntoo{\ensuremath{{N\to\infty}}}

\newcommand\toinf{\to\infty}

\newcommand{\tend}{\longrightarrow}

\newcommand\asto{\overset{\mathrm{a.s.}}{\tend}}

\newcommand\eqd{\overset{\mathrm{d}}{=}}

\newcommand\bbR{\mathbb R}

\newcommand\bbN{\mathbb N}

\newcommand\bbNx{\bbN^*}

\newcommand{\IR}{\mathbbm{R}}

\newcommand\E{\operatorname{\mathbb E}{}}
\renewcommand\P{\operatorname{\mathbb P{}}}
\renewcommand\Pr{\P} 

\newcommand{\pr}{\mathbb P}
\newcommand{\IP}{\mathbb P}
\newcommand{\Prob}{\mathbb P}
\newcommand{\IE }{\mathbbm E}

\newcommand\gD{\Delta}
\newcommand\gf{\varphi}

\newcommand\kk{\kappa}
\newcommand\gl{\lambda}
\newcommand\gL{\Lambda}

\newcommand\gO{\Omega}
\newcommand\gs{\sigma}

\newcommand\gU{\Upsilon}
\newcommand\eps{\varepsilon}

\newcommand\cA{{\mathcal A}}
\newcommand\cB{\mathcal B}

\newcommand\cF{\mathcal F}
\newcommand\cG{\mathcal G}

\newcommand\cL{{\mathcal L}}

\newcommand\cU{{\mathcal U}}

\newcommand\cW{\mathcal W}

\newcommand\tW{\widetilde W}
\newcommand\tX{{\widetilde X}}

\newcommand\tWW{\widetilde\WW}
\newcommand\tPsi{\tilde\Psi}

\newcommand\tchi{\tilde\chi}
\newcommand\tgf{\tilde\gf}
\newcommand\ttau{\tilde\tau}
\newcommand\tg{\tilde g}

\newcommand\bWW{\overline\WW}
\newcommand\bPsi{\overline\Psi}
\newcommand\bpsi{\overline\psi}

\newcommand\bh{\overline h}

\newcommand\WW{\mathbf{W}}
\newcommand\ww{\mathbf{w}}
\newcommand\bx{\mathbf{x}}
\newcommand\FF{\mathbf{F}}
\newcommand\GG{\mathbf{G}}
\newcommand\GGx{\GG}
\newcommand\hGG{\mathbf{\widehat G}}
\newcommand\KK{\mathbf{K}}
\newcommand\ee{\mathbf{e}}

\newcommand{\tone}{\mathbf{1}}

\newcommand\sindic[1]{\boldsymbol1_{#1}}
\newcommand\indic[1]{\boldsymbol1\cpar{#1}}
\newcommand\bigindic[1]{\boldsymbol1\bigcpar{#1}}

\newcommand\etta{\boldsymbol1}

\newcommand\pairs{^{(2)}}
\newcommand\npairs{[n]^{(2)}}

\newcommand{\tinj}{t_{\mathrm{inj}}}
\newcommand{\tind}{t_{\mathrm{ind}}}

\newcommand\cLoo{\cL_\infty}

\newcommand\hW{\widehat{W}}
\newcommand\hWW{\widehat{\WW}}
\newcommand\cWW{\boldsymbol{\cW}}
\newcommand\cWWX{\cWW^*}
\newcommand\hcWW{\widehat{\cWW}}
\newcommand\cWo{\cW}
\newcommand\hcWo{\widehat{\cW}}
\newcommand\cWc{\cW_c}

\newcommand\tauu{\tau^+}
\newcommand\tgO{{\widetilde \gO}}
\newcommand\tcF{{\widetilde \cF}}
\newcommand\tmu{{\widetilde \mu}}

\newcommand\cnorm[1]{\|#1\|_{\square}}
\newcommand\linorm[1]{\|#1\|_{L^1}}
\newcommand\dcut{{\delta_{\square}}}
\newcommand{\dels}{\delta_\square}

\newcommand\wx{weak-$*$}
\newcommand\wxt{weak-$*$ topology}

\newcommand\mpp{measure-preserving}

\newcommand\xioo{\xi_\emptyset}
\newcommand\PPsi{\gL}
\newcommand\Psix{\Psi^*}
\newcommand\psix{\psi^*}
\newcommand\OFmu{(\gO,\cF,\mu)}
\newcommand\Omu{(\gO,\mu)}
\newcommand\bL{\overline L}

\newcommand\thx{\widetilde h}

\newcommand\hq{\hat q}
\newcommand\hqx{\hat q^*}
\newcommand\qx{q^*}
\newcommand\hQ{\widehat{Q}}
\newcommand\hQx{\widehat{Q}^*}

\newcommand{\qhatstarmix}{\hat{q}^{\,*}_{\mathrm{mix}}}
\newcommand{\qhatstarpur}{\hat{q}^{\,*}_{\mathrm{pure}}}
\newcommand{\qhatstarpurtwo}{\hat{q}^{\,*}_{2,\mathrm{pure}}}
\newcommand{\qhat}{\hat{q}}

\newcommand{\q}{q^*}

\newcommand{\cc}{{\rm cc}}

\newcommand{\Gnp}{G_{n,p}}

\newcommand\intoi{\int_0^1}

\newcommand\oi{\ensuremath{[0,1]}}

\newcommand\ooo{[0,\infty)}

\newcommand\setoi{\set{0,1}}

\newcommand\dd{\,\mathrm{d}}

\newcommand\Uoi{\ensuremath{\mathsf U[0,1]}}
\newcommand\cLeb{\cB^\ast}

\newcommand\qw{^{-1}}

\newcommand{\gsf}{$\gs$-field}

\newcommand\rhs{right-hand side}

\newcommand\xoo{_1^\infty}

\newcommand\CHI{\chi^*}
\newcommand\essinf{\operatorname{ess\,inf}}

\newcommand\hr{\hat{r}}
\newcommand\dnormmin{\widetilde{d}_{\min}}
\newcommand\Ev{\operatorname{E_{\mathit v}}}

\def\given{\typeout{Command 'given' should only be used within bracket command}}
\newcounter{@bracketlevel}

\def\@bracketfactory#1#2#3#4#5#6{
\expandafter\def\csname#1\endcsname##1{%
\addtocounter{@bracketlevel}{1}%
\global\expandafter\let\csname @middummy\alph{@bracketlevel}\endcsname\given%
\global\def\given{\mskip#5\csname#4\endcsname\vert\mskip#6}%
\csname#4l\endcsname#2##1\csname#4r\endcsname#3%
\global\expandafter\let\expandafter\given\csname @middummy\alph{@bracketlevel}\endcsname
\addtocounter{@bracketlevel}{-1}}%
}

\def\bracketfactory#1#2#3{%
\@bracketfactory{#1}{#2}{#3}{relax}{1mu plus 0.25mu minus 0.25mu}{0.6mu plus 0.15mu minus 0.15mu}
\@bracketfactory{b#1}{#2}{#3}{big}{1mu plus 0.25mu minus 0.25mu}{0.6mu plus 0.15mu minus 0.15mu}
\@bracketfactory{bb#1}{#2}{#3}{Big}{2.4mu plus 0.8mu minus 0.8mu}{1.8mu plus 0.6mu minus 0.6mu}
\@bracketfactory{bbb#1}{#2}{#3}{bigg}{3.2mu plus 1mu minus 1mu}{2.4mu plus 0.75mu minus 0.75mu}
\@bracketfactory{bbbb#1}{#2}{#3}{Bigg}{4mu plus 1mu minus 1mu}{3mu plus 0.75mu minus 0.75mu}
}

\bracketfactory{clc}{\lbrace}{\rbrace}
\bracketfactory{clr}{(}{)}
\bracketfactory{cls}{[}{]}
\bracketfactory{abs}{\lvert}{\rvert}
\bracketfactory{norm}{\Vert}{\Vert}
\bracketfactory{floor}{\lfloor}{\rfloor}
\bracketfactory{ceil}{\lceil}{\rceil}
\bracketfactory{angle}{\langle}{\rangle}


\title{Graphon Colourings and Limits of Graph Partitions}

\newcommand{\trian}{
  \begin{tikzpicture}[baseline=-0.3ex,scale=0.25]
  \tikzstyle{vertex}=[circle,fill=black, minimum size=2pt,inner sep=1pt]
  \node[vertex] (v1) at (-0.5, 0){};
  \node[vertex] (v2) at (0.5,0){};
  \node[vertex] (v3) at (0,0.8){};
  \draw (v1)--(v2)--(v3)--(v1);
  \end{tikzpicture}}

\newcommand{\cherry}{
  \begin{tikzpicture}[baseline=-0.3ex,scale=0.25]
  \tikzstyle{vertex}=[circle,fill=black, minimum size=2pt,inner sep=1pt]
  \node[vertex] (v1) at (-0.5, 0){};
  \node[vertex] (v2) at (0.5,0){};
  \node[vertex] (v3) at (0,0.8){};
  \draw (v1)--(v3)--(v2);
  \end{tikzpicture}}

\author[S. Janson]{Svante~Janson$^\ast$}
\address{$^\ast$Department of Mathematics, Uppsala University, \upshape \texttt{\{svante.janson, anda.skeja, fiona.skerman\}@math.uu.se}}

\author[T. Lo]{Tiffany~Y.~Y.~Lo$^\circ$}
\address{$^{\circ}$Department of Mathematics, Stockholm University, \upshape \texttt{tiffany.y.y.lo@math.su.se}}

\author[C. McDiarmid]{Colin~McDiarmid$^\diamond$}
\address{$^\diamond$ Department of Statistics, Oxford University, \upshape \texttt{cmcd@stats.ox.ac.uk}}

\author[A. Skeja]{\\ Anda~Skeja$^\ast$}
\author[F. Skerman]{Fiona~Skerman$^\ast$}

\begin{document}

\begin{abstract}
We develop a framework for pure and mixed graphon colourings and use it to study the limiting behaviour of graph partition parameters under cut convergence.
Pure colourings correspond to measurable partitions, while mixed colourings assign a probability
distribution on the colour set to each point. We develop the representation, equivalence and convergence theory, and show that every mixed graphon colouring has an equivalent pure representative,
possibly on a different probability space. For each fixed finite colour set, we establish compactness and
describe the topology of coloured graph limits using the cut norm
on graphons and the weak-$*$ topology on colourings. The usual coloured cut distance characterises convergence for pure colourings, but not for general mixed colourings.

As applications,  we prove that the Cheeger constant and modularity converge to their mixed graphon counterparts along cut-convergent graph sequences whose orders tend to infinity, assuming uniformly linear minimum degree and positive limiting edge density, respectively. Both mixed functionals are invariant under graphon
equivalence. We also establish their continuity in cut distance on the corresponding classes of graphons, with a stability bound for the Cheeger constant. For modularity, we further 
characterise coloured graphons attaining extremal modularity values and prove attainment of the mixed optimum for each fixed number of colours. We also exhibit a graphon whose modularity supremum is not attained by any finite colouring.
\end{abstract}

\maketitle 

{
\footnotesize\tableofcontents}

\section{Introduction}
Graph partition problems optimise an objective over partitions of the vertex set. Examples include maximum cut, minimum $k$-section, modularity and the Cheeger constant. Studying the asymptotic behaviour of such parameters along convergent graph sequences leads to the following question: how should the partitions over which the optimisation is performed be represented in the limit? More precisely, given a sequence of graphs $(G_n)_{n \ge 1}$ with $|G_n| \to \infty$, whose associated step graphons $W_{G_n}$ converge in cut distance to a graphon $W$, and partitions $\mathcal P_n$ of $V(G_n)$, what is the appropriate limiting representation of the sequence $(G_n, \mathcal P_n)$?

Viewing a vertex partition as a colouring, an observation also made by \cite{Diao}, 
provides the bridge to graph limits: the graph and its partition can then be represented jointly by a coloured graphon. In this paper, we develop a coloured-graphon framework that incorporates vertex partitions into dense graph limits and provide a continuum setting for graph partition optimisation. Within this framework, partitions may be represented either by \emph{pure} colourings, corresponding to measurable partitions of the underlying probability space, or by \emph{mixed} colourings, corresponding to probability-valued assignments. Although finite vertex partitions are naturally pure colourings, mixed colourings arise as the continuum counterpart of fractional partitions and play an important role in the limiting theory. Figure \ref{fig:pure-mixed-coloured-graphons} illustrates pure and mixed colourings in the finite and graphon settings. 
\begin{figure}[t]
   \centering

    \includegraphics[width=\textwidth]{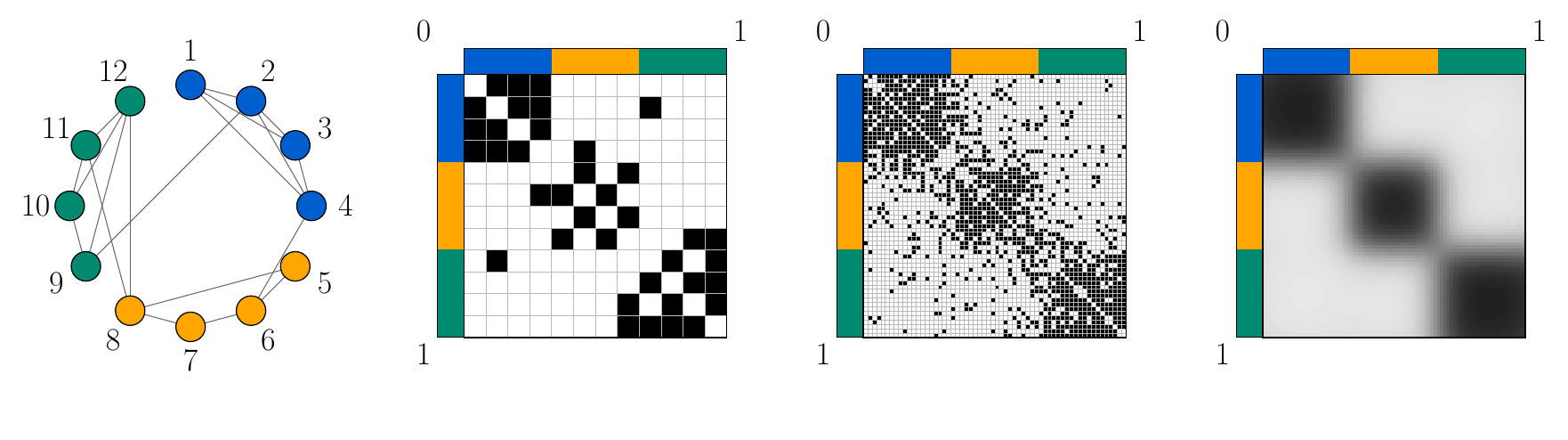}

    \smallskip

    \medskip

    \includegraphics[width=\textwidth]{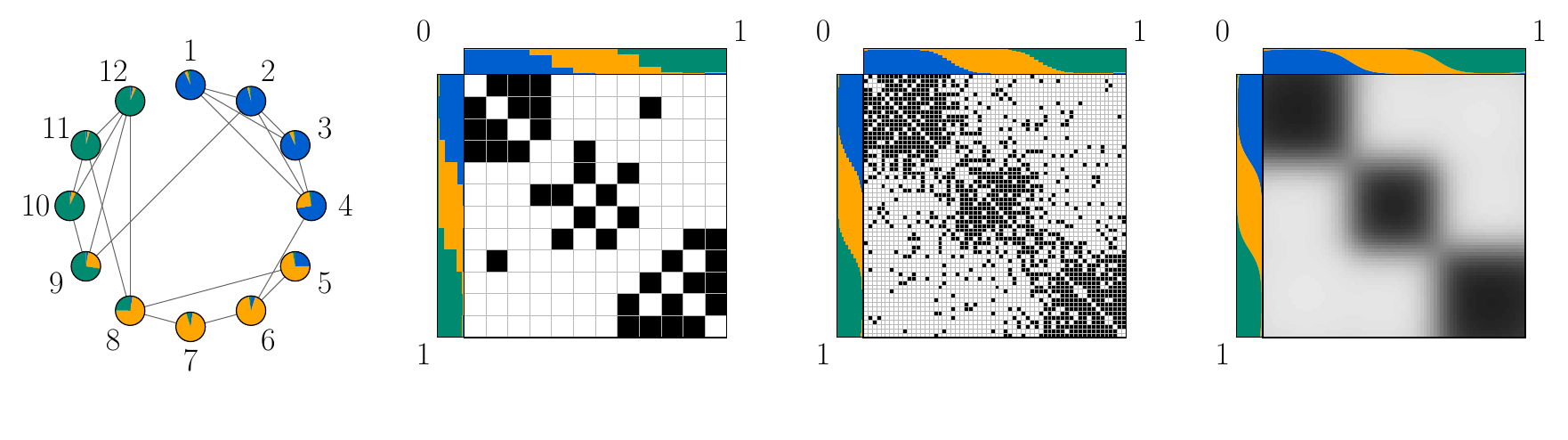}

    \smallskip
    \caption{{\small Finite coloured graphs and their associated coloured graphons. The coloured strips represent the vertex assignments, with a pure colouring \((W,\chi)\) in the top panel and a mixed colouring \((W,\Psi)\) in the bottom panel. In each row, the first two panels show a finite graph with a colouring and its associated step graphon. The third panel shows a step graphon associated with a larger graph, and the final panel illustrates a possible graphon limit. }}
    \label{fig:pure-mixed-coloured-graphons}
\end{figure}

Once partitions are represented on the underlying probability space of the graphon, finite partition optimisation becomes a continuum variational problem over measurable colour classes. Maximum cut illustrates this passage as an optimisation over measurable $2$-colourings of the underlying probability space. Constraints on the sizes of the finite parts translate into constraints on the masses of the corresponding colour classes.  
 Minimum bisection, and more generally minimum $k$-section illustrate this correspondence, with finite balance constraints becoming colour-mass constraints. This perspective complements the graph-partition energy framework of \cite{BCLSV2}, see Remark~\ref{rem:conn-bclsv}.

The Cheeger constant, defined as the minimum ratio of the number of crossing edges to the smaller sum of degrees on either side
of a cut, presents a difficulty because both the numerator
and the denominator must be controlled. A graphon analogue of the Cheeger constant, corresponding to our
formulation in terms of pure colourings, was introduced by \cite{Khetan}. They also showed that cut convergence alone does not imply convergence of Cheeger constants. We prove cut-metric stability for graphons whose degree functions
share a positive lower bound
(Proposition~\ref{prop:cheeger-stability}).
Together with a comparison between finite graphs and their
associated step graphons, this yields convergence under a uniform
linear minimum-degree condition on the finite graphs
(Theorem~\ref{thm:cheeger-conv}).

Modularity, introduced by \cite{NG04}, is a widely used partition parameter for community detection \cite{POM09},  measuring within-community connectivity relative to a degree-based null model. To recall its definition, given a graph $G$, we denote the degree of a vertex $v\in V(G)$ by $d_v$.
Furthermore, for a set $A\subseteq V(G)$ of vertices, we 
let $e(A)$ denote the number of edges inside $A$, and $\vol(A):=\sum_{v\in A} d_v =2e(A)+e(A,A^c)
$.

\begin{definition}[Modularity]\label{D1}
    Let $G$ be a graph with $m\geq 1$ edges. For a partition $\mathcal{A}$
    of the vertices of $G$, the modularity score of $\cA$ on $G$ is
    \begin{align}\label{D1a}
        q(G, \cA) :=
\frac{1}{2m}\sum_{A\in \mathcal{A}} \sum_{u,v\in A}
\Bigpar{\tone_{uv\in G}-\frac{d_ud_v}{2m} } 
= \sum_{A\in \cA}
\Bigpar{\frac{e(A)
}{m} - \frac{\vol(A)^2}{4m^2}},
    \end{align} 
    and the modularity of $G$ is 
    \begin{align}\label{D1b}
q^*(G):=\max_{\cA}q(G, \cA),      
    \end{align}
where the maximum is over all possible vertex partitions of $G$. If $G$ has no edges, then for any vertex partition $\cA$, we let $q(G, \cA):=0$
and so $\qx(G)=0$.
\end{definition}
The within-class degree correction and the unrestricted number of communities make the passage to graphon limits less direct for modularity. A graphon analogue of modularity was introduced by \cite{Klimm} by optimising over measurable partitions, corresponding in our terminology to pure colourings. For a graphon $W$ on $(\Omega,\mu)$ with positive edge density, $\rho_W:=\iint_{\Omega^2}W(x,y) \mathrm d \mu(x) \mathrm d \mu(y)$ and degree function $k_W(x):=\int_\Omega W(x,y) \mathrm d \mu(y)$, we extend their definition by introducing modularity score of a mixed $\kappa$-colouring $\Psi=(\psi_\ell)_{\ell=1}^\kappa$:
\begin{equation}
\hat q(W,\Psi):=\frac{1}{\rho_W}\iint_{\Omega^2}\left(W(x,y)-\frac{k_W(x)k_W(y)}{\rho_W}\right)\sum_{\ell=1}^\kappa \psi_\ell (x) \psi_\ell(y) \mathrm d \mu(x) \mathrm d \mu (y).
\end{equation}
We define the mixed graphon modularity by 
\begin{equation}
\hat{q}_{\mathrm{mix}}^*(W):=\sup_{\kappa \ge 1} \sup_{\Psi} \hat{q}(W,\Psi),
\end{equation}
and use
$\hat q^*(W):=\hat q^*_{\mathrm{mix}}(W)$ as our definition
on general probability spaces. We set $\hat q(W,\Psi):=0$ and consequently $\hqx(W):=0$ when
$\rho_W=0$.

In addition to introducing the graphon analogue of modularity, \cite{Klimm} explore the problem of finding the modularity-maximising partition for several graphon models, but do not establish
convergence of the optimal modularity values along general
cut-convergent graph sequences with positive limiting edge density.
The following theorem establishes this convergence and is one
of the main results of the paper.
\begin{theorem*}[Graphon limit of modularity]\label{thm:intro-main}
Let $(G_n)_{n\geq 1}$ be a sequence of finite simple graphs such that $|G_n|\to\infty$. Suppose that $W_{G_n}\to W$ in cut distance for a graphon $W:[0,1]^2 \to [0,1]$ with $\rho_W>0$. Then
\begin{align}
q^*(G_n)\to \hat q^*(W).
\end{align}
\end{theorem*}
 
On atomless Borel probability spaces, the suprema over pure and mixed colourings agree; see \refT{LX1} and
Remark~\ref{R01+}. The proof combines continuity of graphon modularity at positive density with a comparison between ordinary and fractional modularity. The unrestricted number of communities is handled by a uniform approximation using a bounded number of colours (Lemma~\ref{LMk2}). The assumption $\rho_W>0$ is essential, since graph sequences with different limiting modularity values may all converge to the zero graphon as was noted in \cite{sampling}. 

\subsection{Our contributions}
The paper develops a framework for graphon colourings and applies it to maximum cut, minimum $k$-section, modularity, and the Cheeger constant, with particular emphasis on the latter two.

\vspace{0.5em}

First, we treat pure and mixed colourings simultaneously for any fixed finite colour set. Coloured homomorphism densities describe the limiting sampling laws, and every mixed coloured graphon has an equivalent pure representative. We also introduce simply coloured representatives and establish common-pullback representations for equivalent pure or simply coloured graphons. The product of the cut-norm topology on graphons and the weak-$*$ topology on colourings induces the topology of coloured graph limits.  The usual coloured cut distance describes this topology for pure colourings, but does not characterise equivalence for general mixed colourings. Maximum cut and minimum $k$-section illustrate the use of the framework.

\vspace{0.5em}
Second, we formulate the Cheeger constant using mixed two-colourings, prove equality with the pure infimum on atomless Borel probability spaces, and establish invariance under measure-preserving pullbacks. A finite-graph comparison and a cut-distance stability under a positive degree lower bound yield convergence for growing, cut-convergent
graph sequences with uniformly linear minimum degree. The degree assumption is imposed on the approximating graphs; a lower bound on the limiting graphon's degrees alone does not suffice.

\vspace{0.5em}
Third, we extend the graphon modularity formulation of \cite{Klimm} to mixed colourings on general probability spaces. The mixed optimum is invariant under graphon equivalence, agrees with the pure supremum on atomless Borel probability spaces, and is attained when the number of colours is fixed. We prove continuity at positive density and identify graphon modularity as the limit of finite modularity along cut-convergent graph sequences.  We relate fractional modularity to blow-ups of finite graphs, characterise equality in the bounds for coloured modularity scores, and exhibit a graphon whose modularity is not attained by any finite colouring.

\vspace{0.5em}
Finally, we introduce the notion of restricted estimability, where estimation from a small induced subgraph is required only for graphs satisfying a prescribed condition. Under suitable assumptions on the restriction, we characterise
restricted estimability through convergence and continuity conditions (Proposition~\ref{prop:res-estimability}). Applications include the Cheeger constant under a positive lower
bound on normalised minimum degree, and modularity and the clustering coefficient under a positive lower bound on edge density.

\subsection{Related work}\label{sec:related-work}
\paragraph{Coloured graphons.}
Several variants of coloured graphons have appeared in the graph-limit literature. A related graph-limit theory for pure vertex-coloured graphons with a finite colour set was developed by \cite{Diao}. In common with our framework, they use coloured homomorphism densities to describe convergence of finite coloured graphs and their graphon limits. They additionally introduce a coloured cut metric and establish the corresponding counting lemma, compactness, and equivalence between homomorphism-density and cut-metric convergence. A complementary framework was developed by \cite{AHR26}, who consider two-coloured graphons with soft vertex colourings in the study of co-evolving vertex and edge dynamics on dense graphs. Their framework treats $\kappa=2$ colours and allows mixed colourings $c\colon[0,1]\to[0,1]$. They use coloured subgraph densities to describe convergence. The characterisation of this convergence by coloured cut distance requires restriction to pure colourings. In another line of research,
\cite[Section 7]{GrebikP} use coloured graphons to study large deviation
properties for graphons. Their coloured graphons are more restrictive than
ours; they
define coloured graphons using a partition
$\bigcup_1^\kk A_\ell$ 
of the  probability space (in their case $\oi$) into measurable parts, 
which is equivalent to our pure coloured graphons.

\vspace{0.5em}

\paragraph{Partition functionals.} 
Graph partition problems in dense graphs have been studied in
\cite{BCLSV2} using the so-called fractional $q$-quotients, which are weighted graphs on $q$ vertices with edge and vertex weights given by fractional $q$-partitions, here known as mixed colourings with $q$ colours.
They show that two graphons that are close in cut distance also have sets of fractional $q$-quotients that are close in Hausdorff distance. For a family of partition problems
they establish equivalences between left-convergence of a graph sequence and convergences of these $q$-colour partition problems -- see~\refR{rem:conn-bclsv}.

In \cite{Diao} the authors consider a partition problem that is based on node-level statistics such as normalised vertex degrees and local triangle densities. Informally, a vertex is assigned to a certain part if its corresponding statistics fall within a designated range. By developing a theory of \emph{pure} coloured graphons, they show that, under suitable continuity and relabelling assumptions on the statistics
and boundary assumptions on the prescribed regions, a cut-convergent
dense graph sequence equipped with the resulting partitions converges
in coloured cut distance to the limiting graphon with its induced
pure colouring \cite{Diao}.

The partition problems considered in \cite{BCLSV2} and \cite{Diao} do not include modularity and the Cheeger constant; see Remark~\ref{rem:conn-bclsv} for more details. 
Modularity is moreover 
optimised over partitions with an unbounded number of colours, while these results  
consider optimisation over partitions with a fixed number of parts.

 Analogues of the Cheeger constant were defined for graphons 
 in~\cite{Khetan}, where they also establish corresponding Cheeger and Buser inequalities. For finite weighted graphs, they introduce a fractional Cheeger constant that agrees with the Cheeger constant of the associated step-graphon, and use probabilistic rounding to compare it with the ordinary finite-graph Cheeger constant; they also show that Cheeger constants are not, in general, continuous under cut-norm convergence. We establish a stability bound for graphon Cheeger constants and show convergence from graphs to graphons under a uniform linear minimum-degree bound (Proposition \ref{prop:cheeger-stability} and Theorem \ref{thm:cheeger-conv}).

\vspace{0.5em}
\paragraph{Parameter estimation.} A simple graph parameter (or graph parameter) is a function from the set of simple graphs to reals that is invariant under vertex relabelling. The authors of \cite{BCLSV1} 
show that a graph parameter $f$ is estimable if and only if $f(G_n)$ converges for every convergent graph sequence $(G_n)_{n\ge 1}$ as $n\toinf$. Apart from the trivial examples of subgraph densities, other examples of estimable parameters include maximum cut \cite{AKK99
}, free energies \cite{Lovasz} and edit distance to a hereditary property \cite{BCLSSV}.  The authors of \cite{BKK12} extend the equivalence notions of estimability to weighted bounded parameters and establish that certain balanced minimum multi-way cut densities are estimable. A related 
notion is property testing, pioneered by \cite{GGR98}; see  also \cite{AS05} and references therein. In particular, \cite{GGR98} establishes estimability for an abstract graph partition problem that encompasses $k$-colourability and Max-Cut. This result does not cover  modularity, clustering coefficient or the Cheeger constant -- all of which we show to be $r$-estimable in Section~\ref{sec.estim}.

\subsection{Organisation} The rest of the paper is organised as
follows. Section~\ref{sec:coloured-graphons} develops the representation,
equivalence and topology of coloured
graphons. Section~\ref{sec:graphon-partition} treats maximum cut, minimum
$k$-section and the Cheeger constant. Section~\ref{sec3:modularity} develops
graphon modularity, its relation to fractional partitions and graph
blow-ups, and its continuity and convergence properties. Section \ref{sec.estim} introduces and characterises restricted
estimability, with applications to the Cheeger constant, modularity and the clustering coefficient.

\section{A Framework for Graphon Colourings}\label{sec:coloured-graphons}

This section develops a framework for coloured graphons, which allows both pure colourings, corresponding to measurable partitions of the
underlying probability space, and mixed colourings, corresponding to
probability-valued vertex colours. 

This coloured graphon framework allows the representation and
equivalence theory needed for graph-partition problems. In particular, we
establish the relation between pure and mixed colourings, characterise the
associated coloured graph limits through coloured homomorphism densities, and
develop the pullback and equivalence theory needed to make the framework
independent of the chosen graphon representation. These results will later
allow us to pass between finite, measurable, and fractional partitions.

As noted in the introduction, some parts of the theory developed below have
appeared in~\cite{Diao, GrebikP, AHR26}; see Remark~\ref{rem:earlier-coloured} for a precise account of earlier contributions.

\subsection{Pure and mixed colourings of graphons}
We may without loss of generality (see below) 
consider graphons defined on $\oi$, but it is sometimes useful to define
them on other probability spaces $\OFmu$, so we make the definitions more
general. (We sometimes omit $\cF$ from the notation for brevity and write
$\Omu$.)

\begin{definition}[Pure and mixed-coloured graphons]\label{def:pure-mixed}
For $\kk\in\bbN$, let $\gD_\kk$ be the set of probability measures on
$[\kk]$; this
can be identified with the simplex 
$\bigset{\bx:=(x_\ell)_1^\kk\in\ooo^\kk:\sum_{\ell=1}^\kk x_\ell=1}$.

Let $\kappa\ge1$ be a fixed integer.
A \emph{pure $\kk$-coloured graphon} on a probability space $(\gO,\cF,\mu)$ is a pair
$\WW=(W,\chi)$, where $W$ is a graphon, i.e., a measurable symmetric
function $\gO^2\to\oi$, and $\chi$ is a measurable mapping $\gO\to[\kk]$.

A \emph{mixed $\kk$-coloured graphon} on a probability space $(\gO,\cF,\mu)$ is a pair
$\WW=(W,\Psi)$, where $W$ is a graphon on $\gO$ and $\Psi$ is a measurable map $\gO\to\gD_\kk$;
equivalently, $\Psi=(\psi_\ell)_1^\kk$ where each $\psi_\ell$ is a measurable map
$\gO\to\oi$ and $\sum_\ell\psi_\ell=1$.
(We will use the notation $\Psi=(\psi_\ell)_1^\kk$ frequently, 
often without comment.)
\end{definition}

\noindent The informal interpretation is that for a pure coloured graphon, each $x\in\gO$ is painted
with colour $\chi(x)$, while for a mixed coloured graphon, each $x$ is painted randomly, with
probability $\psi_\ell(x)$ for colour $\ell$ (independently for different
$x\in\gO$).

We will also refer to $\chi$ and $\Psi$ as pure and mixed
\emph{$\kk$-colourings} of $\gO$, or of $W$.

The number of colours $\kk$ is usually fixed, and will often be omitted from
the names.
Note that the case $\kk=1$ is trivial, in the sense that $\chi$ and $\Psi$
then are fixed; hence 1-coloured graphons can be identified with standard
uncoloured graphons.

\begin{remark}\label{R01}
When we consider graphons on $\gO=\oi$,   
we tacitly regard $\oi$ as a probability space by equipping it with
Lebesgue measure $\gl$.
Just as for uncoloured graphons, it does not matter 
whether we then use the Borel or Lebesgue \gsf{} on $\oi$,
since any Lebesgue measurable coloured graphon $(W,\Psi)$
is a.e.\ equal to a Borel measurable one $(W',\Psi')$
(meaning that $W=W'$ a.e.\ on $\oi^2$ and $\Psi=\Psi'$ a.e.\ on $\oi$),
and graphons that are a.e.\ equal will be identified in the sequel.
\end{remark}

\begin{remark}\label{R01+}
We usually consider graphons defined on an arbitrary probability space
$\OFmu$, but we sometimes consider only the traditional case $(\oi,\gl)$.
(Sometimes because it really matters, and sometimes just for
convenience.)
Note that if $\OFmu$ is an atomless Borel probability space, then $\OFmu$ is
isomorphic to $(\oi,\gl)$ by a \mpp{} bijection, see e.g.\ 
\cite[Theorem A.7]{SJ249}. 
Hence, all results below for graphons on $\oi$ hold for any atomless Borel
probability space $\OFmu$.
\end{remark}

\begin{remark}\label{rem:earlier-coloured}
\cite{Diao} define coloured graphons as our pure coloured graphons. For a fixed finite colour set, they establish
cut-metric compactness, characterise convergence by coloured
homomorphism densities, and prove density of finite coloured graphs
\cite[Theorems~3.7--3.9]{Diao}.

\cite{GrebikP} define coloured graphons using a partition 
$\bigcup_1^\kk A_\ell$
of the  probability space (in their case $\oi$) into measurable parts. This
is also 
clearly equivalent to our pure coloured graphons. 
They establish compactness for a coloured version of the cut metric
\cite[Theorem~7.1]{GrebikP}, see \refSS{SSmetrics}.

\cite{AHR26}
consider only the case $\kk=2$ (and $\gO=\oi$). They define a colouring of a
graphon as a function $c:\oi\to\oi$, which equals our $\psi_1$; since
$\psi_2=1-\psi_1$ when $\kk=2$, their definition is equivalent to our definition of mixed colourings
(in the case $\kk=2$);
they also
establish compactness for a coloured version of the cut metric
\cite[Section~2.1]{AHR26};
see again \refSS{SSmetrics}.
\end{remark}

We may regard $[\kk]$ as a subset of $\gD_\kk$, identifying $\ell\in[\kk]$ with the
point mass at~$\ell$, which in the simplex is represented by the unit vector 
$(\etta_{j=\ell})_{j=1}^\kk$. Thus the pure coloured graphons form a subset of the mixed coloured graphons;
more precisely, a mixed coloured graphon $(\WW,(\psi_\ell)_1^\kk)$ is pure if and only if
each $\psi_\ell:\gO\to\setoi$, i.e., if and only if each
$\psi_\ell:=\etta_{\gO_\ell}$ for a measurable partition $\gO=\bigcup_1^\kk\gO_\ell$.

We will see that every mixed coloured graphon is equivalent to a pure one
in the sense defined below; 
hence we may restrict ourselves to pure coloured graphons without real loss of generality.
Nevertheless, it will be convenient to consider mixed coloured graphons,
and we will 
usually do so; 
\emph{by ``coloured graphon" we will mean a mixed coloured graphon}. 

\begin{remark}
  Our definitions generalise standard concepts for graphs.
A \emph{(pure) coloured graph} is a graph $G$ equipped with a mapping
$\chi:V(G)\to [\kk]$ that assigns a colour to each vertex.
Similarly, a \emph{fractionally coloured graph} is a graph $G$
equipped with a mapping $\Psi:V(G)\to \gD_\kk$ that assigns a colour
distribution 
to each vertex. We will usually consider pure coloured graphs only; hence
(in contrast to graphons),
\emph{by ``coloured graph" we will mean a pure coloured graph}. 
\end{remark}

\begin{example}\label{EG1}
  An uncoloured graph $G$ defines a graphon  by a well-known standard
  construction; there are two closely related versions 
which are equivalent for many but not all purposes.
(See \refE{ex:C5} and \eqref{c511} there 
for a situation where the difference matters.)
In one version, which we denote by $W'_G$,
we take $\gO=V(G)$, with the uniform probability measure,
and let $W'_G$\label{def.WdashG} be the graphon defined by the adjacency matrix;
in another version $W_G$\label{def.WG}
we take $\gO=\oi$, partition it into $|G|$ intervals $I_v$ ($v\in V(G)$) of
the same lengths, and define $W_G$ by blowing up each point $v$ in the first
version to the interval $I_v$.
(This is a pullback of $W'_G$ by the map $\gf$ that maps each $I_v$ to
$v$; see \eqref{Wpull}.)
We may call $W'_G$ the \emph{atomic}
and $W_G$ the \emph{continuous}
graphon corresponding to $G$.

For a (pure) coloured graph $\GG$ 
we similarly
define pure coloured graphons $\WW'_\GG$ and
$\WW_\GG$ by adding to $W'_G$ and $W_G$
the map $\chi:\gO\to[\kk]$ which gives each point $v$, or
each interval $I_v$, the colour in $G$ of $v$.
Similarly, a mixed coloured graph $\GG$ defines
(mixed) coloured graphons $\WW'_\GG$ and $\WW_{\GG}$.
\end{example}

\subsubsection{Random graphs defined by a coloured graphon}\label{SSGG}

For $n\in\bbNx:=\mathbbm{N}\cup\{\infty\}$, 
let $\cL_n=\cL_{n;\kk}$ be the set of all $\kk$-coloured graphs with vertex set
$[n]$, where we define $[\infty]:=\bbN$.
$\cL_n$ can be identified with $[\kk]^n\times\setoi^{\npairs}$;
thus $\cL_n$ is a finite set if $n$ is finite, and $\cLoo$ is a compact
metric space. 
Given a mixed coloured graphon $\WW=(W,\Psi)$,
and $n\in\bbN$, let $\GGx(n,\WW)\label{def.WW-random-graph}\in\cL_n$ be the random pure-coloured graph with
vertex set 
$[n]$ defined as follows (generalising the well-known uncoloured case):
first sample i.i.d.\  points $X_i\in\gO$ (according to the
measure $\mu$); then, conditioned on $(X_i)_1^n$,
for each pair $(i,j)\in\npairs$ 
add an edge $ij$
with probability $W(X_i,X_j)$, and for each $i\in[n]$, give 
vertex $i$ colour $\ell$ with probability $\psi_\ell(X_i)$
(with all random choices independent, conditioned on $(X_i)_1^n$).
Note that, as a special case, for a pure coloured graphon $\WW=(W,\chi)$, we
simply give vertex $i$ colour $\chi(X_i)$; see Figure \ref{fig:graphon-sampling}. 

Similarly, $\GGx(\infty,\WW)\in\cL_\infty$ is the countably infinite random graph with
vertex set $[\infty]=\bbN$ defined in the same way.

Note that $\GGx(n,\WW)$, for any $n\in\bbNx$, is an exchangeable random graph,
i.e., its distribution is invariant under any permutation of the vertices.
Note also that $\GGx(n,\WW)$ is the induced subgraph of $\GGx(\infty,\WW)$ defined by
the first $n$ vertices. 

Let $\cL_n^{\fra}=\cL_{n,\kappa}^{\fra}$ be the set of fractionally
$\kappa$-coloured graphs with vertex set $[n]$. 
Given a mixed coloured graphon $\WW=(W,\Psi)$ and $n\in\bbNx$,
we define also a fractionally coloured random graph 
$\GG^{\fra}(n,\WW)\in \cL_n^{\fra}$\label{def.WW-random-graph-frac} as the 
fractionally 
coloured graph $(G(n,W),\eta)$  obtained as follows. First, sample
i.i.d.\  points $X_i\in \Omega$ and define the graph
$G(n,W)$ as above. Then, for each
$i\in[n]$, set $\eta(i)=(\psi_\ell(X_i))_{\ell=1}^\kk$. 
Note that for a pure-coloured graphon $\WW=(W,\chi)$, $\GG^{\fra}(n, \WW)=\GGx(n, \WW)$.

\begin{figure}[t]
    \centering
\includegraphics[width=\textwidth]{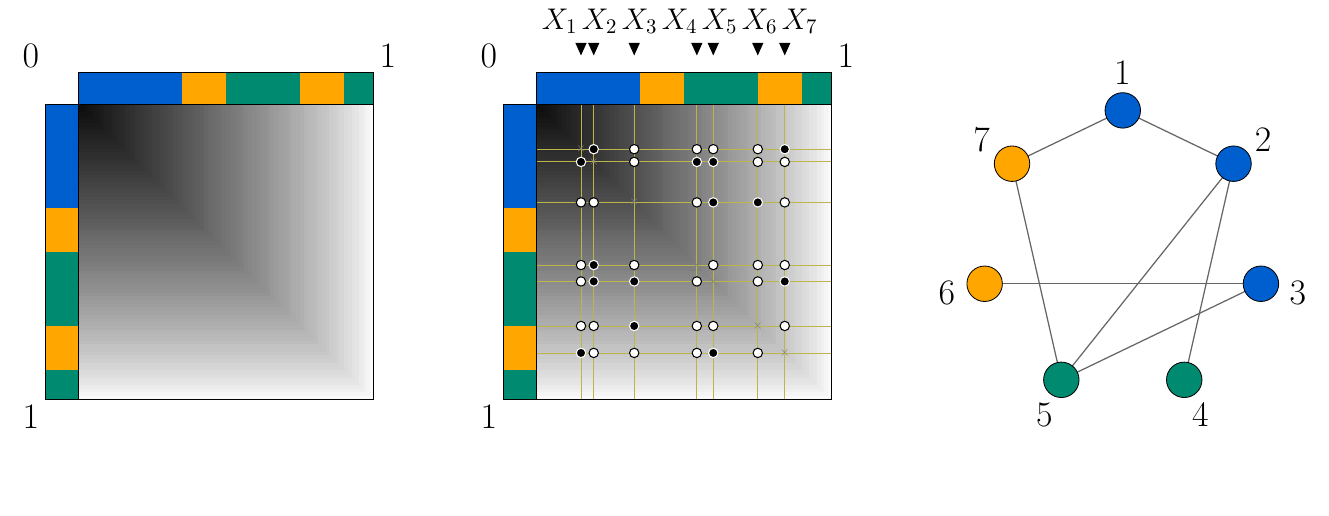}
    \caption{{\small Sampling a pure-coloured random graph $\GGx(n,\WW)$ from a pure-coloured graphon $\WW=(W,\chi)$, illustrated for $n=7$. \emph{Left:} grey levels represent $W$ and the coloured strips represent $\chi$. \emph{Middle:} latent positions $X_1,\ldots,X_7$ are sampled independently, and edges are drawn conditionally independently with probabilities $W(X_i,X_j)$. \emph{Right:} the resulting graph, with vertex $i$ coloured $\chi(X_i)$.}}\label{fig:graphon-sampling}
\end{figure}

\subsection{Homomorphism densities of coloured graphons and equivalence} 

We next define homomorphism densities for coloured graphons.
For a coloured graphon $\WW=\bigpar{W,(\psi_\ell)_\ell}$ and a 
(finite) pure coloured graph $\FF\in\cL_n$
we define, letting $\ell_i$ be the colour of vertex $i\in\FF$, 
\begin{align}\label{pa1a}
  t(\FF,\WW)&:=\P\bigpar{ \GGx(n,\WW)\supseteq\FF}
\\&\phantom:
=\int_{\gO^n} \prodin \psi_{\ell_i}(x_i) \prod_{ij\in  E(\FF)} W(x_i,x_j)
\dd\mu(x_1)\dotsm \dd \mu(x_n)
\label{pa1b}
,\end{align}
where the last equality follows from the definition of $\GGx(n,\WW)$. 
Since the distribution of $\GGx(n,\WW)$ is exchangeable, $t(\FF,\WW)$ is
invariant under relabellings of the vertices, and is thus also well defined
for $\FF\in\cU$, where $\cU$ is the countably infinite set of all
finite unlabelled $\kk$-coloured graphs, i.e.\ $\cU:=\bigcup_1^\infty\cU_n$,
where $\cU_n:=\cU_{n;\kk}$ is the (finite)
set of all unlabelled $\kk$-coloured graphs with 
$n$ vertices. 

\begin{lemma}\label{L=}
  If\/ $\WW_1$ and $\WW_2$ are two coloured graphons, then the following are equivalent.
  \begin{romenumerate}
  \item \label{L=t}
$t(\FF,\WW_1)=t(\FF,\WW_2)$ for every coloured graph $\FF\in\cU$.
  \item\label{L=n} 
$\GGx(n,\WW_1)\eqd \GGx(n,\WW_2)$ for every $n\in\bbN$.
  \item\label{L=oo} 
$\GGx(\infty,\WW_1)\eqd \GGx(\infty,\WW_2)$.
  \end{romenumerate}
\end{lemma}

\begin{proof}
\ref{L=t}$\iff$\ref{L=n}:
  It follows from the definition \eqref{pa1a} that for every $n\in\bbN$,
the distribution of $\GGx(n,\WW)\in\cL_n$ is determined by the vector
$(t(\FF,\WW))_{\FF\in\cL_n}$, 
or equivalently
by $(t(\FF,\WW))_{\FF\in\cU_n}$, 
and conversely.

\ref{L=n}$\iff$\ref{L=oo}:
Follows immediately from the
definitions and the fact that 
$\cL_\infty\cong [\kk]^\infty\times\setoi^{[\infty]\pairs}$ has a product topology.
(Cf.\ the proof of \refL{L1}\ref{L1k}$\implies$\ref{L1hG} below.)
\end{proof}
\begin{figure}[htbp]
    \centering
\includegraphics[width=0.7\textwidth]{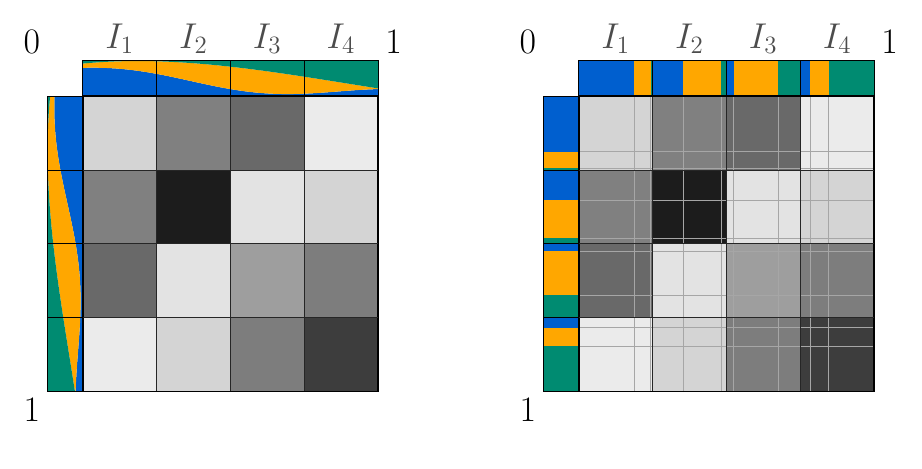}
    \caption{{\small 
    \emph{Left:} The grey blocks represent $W$ on the partition $I_1,\ldots,I_4$, while the coloured strips encode the mixed colouring $\Psi$. \emph{Right:} An equivalent pure-coloured representative, obtained by subdividing each $I_i$ into deterministic colour classes and retaining the value of $W$ from the corresponding parent blocks.}}\label{fig.mix_to_pure}
\end{figure}
\begin{definition}\label{D=}
  Two coloured graphons $\WW_1$ and $\WW_2$ are \emph{equivalent} 
(written $\WW_1\cong\WW_2)$
if they satisfy
one, and thus all, of
the equivalent conditions in \refL{L=}.
(This is, in the uncoloured case, often called 
\emph{weak equivalence}, but we will
not consider any other type of equivalence.)
In particular, note that two coloured graphons that are a.e.\ equal are equivalent.
\end{definition}

 In Figure~\ref{fig.mix_to_pure} we give an illustration of how to obtain an equivalent pure-colouring from a mixed-colouring on a step graphon. 
We show that this holds more generally and that
for every graphon there is an equivalent pure
graphon. 
\begin{lemma}\label{Lmix}
  Every mixed coloured graphon is equivalent to a pure coloured graphon.
\end{lemma}
\begin{proof}
  Let $\WW=\bigpar{W,(\psi_\ell)_1^\kk}$ be a mixed coloured graphon defined on a probability
  space $(\gO,\cF,\mu)$.
Consider the definition of $\GGx(n,\WW)$ above, let $X_i$ be as there, and let
$L_i$ be the colour assigned to vertex $i$.
Then the pairs $\tX_i:=(X_i,L_i)$ are i.i.d.\  random variables in the space 
$\tgO:=\gO\times[\kk]$.
Let $\tmu$ be the distribution of $\tX_i$. Then $\tmu$ is a probability
measure on $\tgO$ (equipped with the product \gsf{} $\tcF$) given by
\begin{align}\label{tmu}
\dd\tmu(x,\ell) = \psi_\ell(x)\dd\mu(x) 
\quad\text{on}\quad 
\gO\times\set \ell,
\;\ell\in[\kk].
\end{align}
Let $\pi_1:\tgO\to\gO$ and $\pi_2:\tgO\to[\kk]$ be the projections of
$\tgO=\gO\times[\kk]$ onto the first and second coordinate, respectively,
and define $\tW(x,y):=W(\pi_1(x),\pi_1(y))$ for $x,y\in\tgO$.
Then  $\WW':=(\tW,\pi_2)$ is a pure coloured graphon defined on
$(\tgO,\tcF,\tmu)$ and it is easy to see that 
$\GGx(n,\WW')\eqd \GGx(n,\WW)$ for every $n\le\infty$, and thus $\WW'\cong\WW$.
\end{proof}

Similarly, we have the following lemma.
\begin{lemma}\label{Loi}
  Every coloured graphon defined on a probability space $(\gO,\cF,\mu)$
is equivalent to a pure coloured graphon defined on $\oi$.
\end{lemma}
\begin{proof}
  This follows from the proof of \refT{TW} below, but can also (more easily)
be seen directly as a special case of \cite[Theorem 7]{SJ210}, combined with
\refL{Lmix} which shows that we may assume the graphon to be pure.
\end{proof}

\subsection{Coloured graph limits and graphon representations}\label{Slim}

In this section we consider limits of sequences of coloured graphs
(with a fixed number of colours); we generalise definitions and results that
are well-known for the uncoloured case, see e.g.\ \cite{Lovasz}.
We first define convergence of sequences of coloured graphs without
specifying the limits; we then use this to define the set of coloured graph
limits in a rather abstract way; finally we identify these limits with
equivalence classes of
coloured graphons. (This is partly inspired by the treatment of uncoloured
graph limits in \cite{SJ209}.)

\subsubsection{Homomorphism densities and convergence of coloured graphs} 
Let $\FF$ and $\GG$ be (pure) coloured graphs.
We define, 
again generalising standard notation,
$t(\FF,\GG)$ as the probability that a random map $V(\FF)\to V(\GG)$
is a homomorphism $\FF\to\GG$, 
which for coloured graphs means that it maps each edge
to an edge, and each vertex to a vertex with the same colour.
We define similarly
$\tinj(\FF,\GG)$ when $|\FF|\le|\GG|$
as the probability that an random injective map $V(\FF)\to V(\GG)$ 
(obtained by sampling $|\FF|$ vertices from $|\GG|$ by drawing
without replacement) is a homomorphism;
we define also
$\tind(\FF,\GG)$ when $|\FF|\le|\GG|$
as the probability that an random injective map $V(\FF)\to V(\GG)$ 
is an isomorphism;
when $|\FF|>|\GG|$ we define $\tinj(\FF,\GG):=\tind(\FF,\GG):=0$.
Note that $\tinj(\FF,\GG)$, and $\tind(\FF,\GG)$ 
are invariant under relabelling
of the vertices of $\FF$ or $\GG$; thus they are well defined for both
labelled and unlabelled coloured graphs.

For a (possibly unlabelled) coloured graph $\GG$ and $k\le|\GG|$, let $\GG[k]$ be
the random induced 
subgraph of $\GG$ obtained by selecting $k$ distinct vertices
$v_1,\dots,v_k$ of $\GG$ uniformly at random; we regard $\GG[k]$ as a labelled
coloured graph in $\cL_k$. Note that for any $\FF\in\cL_k$,
\begin{align}\label{tind}
\tind(\FF,\GG)&=  \P(\GG[k]=\FF),
\\\label{tinj}
\tinj(\FF,\GG)&=  \P(\GG[k]\supseteq\FF).
\end{align}
In particular, if $n=|\GG|$, then $\GG[n]\in\cL_n$ is the random labelled coloured graph
obtained by labelling the vertices of $\GG$ by $1,\dots,n$ at random. 
(If $\GG$ already is labelled we ignore the labels and relabel.)
We extend $\GG[n]$ to a random infinite coloured graph $\hGG\in\cLoo$ by adding new
vertices $n+1,\dots$ but no new edges; the new vertices are given arbitrary
colours, say colour 1 for definiteness.
Note that for $k\le n$, $\GG[k]$ is the induced subgraph of $\hGG$ defined
by the first $k$ vertices.

We extend the notion of graph convergence to coloured graphs. 
We use the definitions above to give a number of different equivalent
conditions.

\begin{definition}\label{Dconv}
 Let\/ $(\GG_n)\xoo$ be a sequence of $\kk$-coloured graphs such that  $|\GG_n|\to\infty$.
 We say that the sequence converges if it satisfies one, and thus all, of
 the equivalent  conditions in \refL{L1} below.
\end{definition}

\begin{lemma}\label{L1}
Let\/ $(\GG_n)\xoo$ be a sequence of $\kk$-coloured graphs such that $|\GG_n|\to\infty$.
Then the following are equivalent.
\begin{romenumerate}  
\item \label{L1t}
$t(\FF,\GG_n)$ converges as \ntoo{} for every $\FF\in\cU$.
\item \label{L1tinj}
$\tinj(\FF,\GG_n)$ converges as \ntoo{} for every $\FF\in\cU$.
\item \label{L1tind}
$\tind(\FF,\GG_n)$ converges as \ntoo{} for every $\FF\in\cU$.
\item \label{L1k}
The random coloured graphs $\GG_n[k]\in\cL_k$ converge in distribution as
\ntoo{} for every fixed $k\in\bbN$. 
\item \label{L1hG}
The random infinite coloured graphs $\hGG_n\in\cLoo$ converge in distribution as \ntoo.
\end{romenumerate}
\end{lemma}
In \ref{L1k} and \ref{L1hG}, the limit is a random graph in $\cL_k$ and
$\cLoo$, respectively.
\begin{proof}
The proof is as for uncoloured graphs (see e.g.\ \cite{SJ209}), 
so we will omit some details.

  \ref{L1t}$\iff$\ref{L1tinj}: 
A random map $V(\FF)\to V(\GG_n)$ is injective
  with probability $1-O(|\GG_n|\qw)=1-o(1)$, and thus
  \begin{align}
\label{tid}
t(\FF,\GG_n)=\tinj(\FF,\GG_n)+o(1). 
\end{align}

  \ref{L1tinj}$\iff$\ref{L1tind}: 
For each fixed $k$ and all $\GG_n$ such that $|\GG_n|\ge k$, 
the vectors $(\tinj(\FF,\GG_n))_{\FF\in\cL_k}$
and $(\tind(\FF,\GG_n))_{\FF\in\cL_k}$ are related by an invertible 
linear map depending only on $k$. 
(See \eqref{tind}--\eqref{tinj} and use inclusion-exclusion.)

  \ref{L1tind}$\iff$\ref{L1k}: 
We see from \eqref{tind} that
for each $k\le|\GG_n|$, 
the distribution of $\GG_n[k]\in\cL_k$ is given by the vector
$(\tind(\FF,\GG_n))_{\FF\in\cL_k}$.

  \ref{L1k}$\iff$\ref{L1hG}: 
Since $\cLoo$ has a product topology, it follows that convergence in
distribution in $\cLoo$ is equivalent to convergence of the restrictions of
the graphs to $[k]$ for every fixed $k<\infty$. As remarked above, this
restriction of $\hGG_n$ (i.e., the induced subgraph with vertex set $[k]$)
is $\GG_n[k]$ as soon as $n$ is so large that $|\GG_n|\ge k$, and the
equivalence follows.
\end{proof}

\begin{remark}\label{RL1}
\refL{L1} can be extended to convergence (in distribution) of a sequence of \emph{random}
  unlabelled coloured graphs $\GG_n$.
See  \cite[Theorems 3.1 and 4.1]{SJ209} (for uncoloured graphs)
which extend verbatim to coloured graphs;
the limit then is a random coloured graph limit.
\end{remark}

\subsubsection{Coloured graph limits}\label{SSSlim}

We have in \refD{Dconv} defined convergence
without saying what the limit is.
From an abstract point of view, we may define the set 
$\hcWW=\hcWW_\kk$
of $\kk$-coloured graph
limits by considering the set of all convergent sequences $(\GG_n)\xoo$
and taking the quotient modulo the equivalence relation defined by
$(\GG_n)\xoo\cong(\GG_n')\xoo$
when the intermingled sequence $\GG_1,\GG'_1,\GG_2,\GG'_2,\dots$ converges
(this is thus interpreted as the two sequences having the same limit).
\refL{L1} implies two more concrete versions, stated as the following
preliminary definition.
\begin{definition}\label{DhcWW}
The set $\hcWW=\hcWW_\kk$ of $\kk$-coloured graph limits,
for a given $\kk\ge1$, 
can be defined by any of:
\begin{enumerate}
 \renewcommand{\labelenumi}{\textup{(W\arabic{enumi})}}%
 \renewcommand{\theenumi}{\labelenumi}%
\item \label{hcWWa}
Let\/ $\tau:\cU\to\oi^\cU$ be the map defined by
\begin{align}\label{tau}
\tau(\GG):=(t(\FF,\GG))_{\FF\in\cU}.  
\end{align}
Then $\hcWW$ is the subset of $\oi^{\cU}$ 
consisting of all limit points of 
$\tau(\GG_n)$ for convergent sequences
of $\kk$-coloured graphs $\GG_n$
with $|\GG_n|\to\infty$.
If $\hWW\in\hcWW$ and $|\GG_n|\to\infty$, 
then $\GG_n\to\hWW$ iff $\tau(\GG_n)\to\hWW$ in $\oi^\cU$.

\item \label{hcWWb}
$\hcWW$ is the set of probability distributions on $\cLoo$ that are limits 
in distribution of $\hGG_n$ for some convergent sequence 
of $\kk$-coloured graphs $\GG_n$ 
with $|\GG_n|\to\infty$.
If $\hWW\in\hcWW$ and $|\GG_n|\to\infty$,
then $\GG_n\to\hWW$ iff $\hGG_n\to\hWW$ in distribution.
\end{enumerate}  
\end{definition}

These versions are formally different, since they define the coloured graph
limits,
i.e., the elements of $\hcWW$, as different types of objects.
Nevertheless, they are equivalent by \refL{L1}, and we may regard both
\ref{hcWWa} and \ref{hcWWb} as different representations of the set of graph
limits. We will use both representations below.
However, these two representations seem to have limited use;
as the large literature on uncoloured graph limits shows,
the most important representation of uncoloured graph limits is by graphons,
and we extend this to the coloured case in \refT{TW} below after some further
preparations. (This is why \refD{DhcWW} should be regarded as preliminary.)

We note one important simple consequence of \ref{hcWWa}.
We give $\hcWW_\kk$ the topology as a subset of $\oi^\cU$.
  \begin{lemma}\label{LC}
 The space $\hcWW_\kk$ of $\kk$-coloured graph limits is a compact metric space.
  \end{lemma}
\begin{proof}
This is almost immediate, but we have to take care of the condition
$|\GG_n|\to\infty$.
We thus define $\tauu:\cU\to\oi^\cU\times\oi$ by
\begin{align}\label{tauu}
\tauu(\GG):=\bigpar{\tau(\GG),|\GG|\qw}.
\end{align}
Then \ref{hcWWa} yields
\begin{align}\label{lc1}
 \hcWW = \bigset{\ww\in\oi^\cU:(\ww,0)\in\overline{\set{\tauu(\GG):\GG\in\cU}}}
.\end{align}
Hence, $\hcWW$ is a closed subset of $\oi^\cU$,
and the result follows since $\oi^\cU$
is a compact
metric spaces.
\end{proof}

\begin{remark}
  The map $\tauu$ defined in \eqref{tauu} is injective.
To see this, suppose that $\tauu(\GG_1)=\tauu(\GG_2)$, and note that this
means that
$t(\FF,\GG_1)=t(\FF,\GG_2)$ for every $\FF\in\cU$, and also
$|\GG_1|=|\GG_2|$.  
It is easy to see (by induction on $|\FF|$)
that this implies
$\tinj(\FF,\GG_1)=\tinj(\FF,\GG_2)$ 
for every coloured graph $\FF$. Taking $\FF=\GG_1$ shows that 
$\tinj(\GG_1,\GG_2)=\tinj(\GG_1,\GG_1)>0$ and thus there exists a
bijection $\gs:V(\GG_1)\to V(\GG_2)$ that is a homomorphism.
In particular, $e(\GG_2)\ge e(\GG_1)$, and by symmetry
$e(\GG_2)=e(\GG_1)$. It follows that $\gs$ is an isomorphism, and thus
$\GG_1=\GG_2$ as unlabelled coloured graphs.

Since $\tauu$ is injective, we may identify the space $\cU$ of
$\kk$-coloured graphs 
with its image $\tauu(\cU)$. 
We may also identify $\hcWW\subset\oi^\cU$
with $\hcWW\times\set0\subset\oi^\cU\times\oi$.
Then the sets $\cU$ of coloured graphs and $\hcWW$ of coloured
graph limits are both seen as subsets of the same space $\oi^\cU\times\oi$;
$\cU$ and $\hcWW$ are disjoint, and the closure $\overline{\cU}$ equals
$\cU\cup\hcWW$. 
Note that every element of $\cU$ is isolated in $\cU\cup\hcWW$.
Hence,
in this setting, a sequence $\GG_n$ of coloured graphs
converges (in $\cU\cup\hcWW$)
if either $|\GG_n|\to\infty$ and the sequence converges to a graph
limit in the sense above, or  the sequence is constant for large $n$. 
See \cite[Section 2]{SJ209} for further discussion (in the uncoloured case).
\end{remark}

\begin{remark}
  We may also define the topology on $\hcWW$ using \ref{hcWWb}, regarding
  $\hcWW$ as a subset of the compact set of all probability measures on $\cLoo$;
it is easily seen that this yields the same topology. 
\end{remark}

\subsubsection{Graphons and graph limits}\label{SSSgglim}

It is a well-known fact that uncoloured graph limits can be
represented by graphons; we can now extend this to the coloured case.
\begin{theorem}
  \label{TW}
Each coloured graphon\/ $\WW$ 
defines a coloured graph limit\/ $\hWW$
(usually also denoted by\/ $\WW$, by a standard abuse of notation)
such that the following holds.
\begin{romenumerate}
\item \label{TW1}
Every graph limit is given by some coloured graphon $\WW$. Moreover, 
we may choose $\WW$ as a pure coloured graphon on $\gO=\oi$. 
\item \label{TW2}
Two coloured graphons $\WW_1$ and $\WW_2$ define the same graph limit if and
only if they are equivalent.
\item \label{TW3}
  Let\/ $(\GG_n)\xoo$ be a sequence of $\kk$-coloured graphs such that $|\GG_n|\to\infty$.
Then
\begin{align}\label{tw3}
\GG_n\to\WW
\iff
  t(\FF,\GG_n)\to t(\FF,\WW)  \quad 
\forall
\FF\in\cU.
\end{align}
\item \label{TW4}
The sequence of random coloured graphs $\GGx(n,\WW)\to\WW$ almost surely.
\end{romenumerate}
\end{theorem}

\begin{proof}
We use \ref{hcWWa} as a preliminary definition of the set $\hcWW$ of graph limits.
Recall that we have defined $\tau(\GG)$ for coloured graphs $\GG$ in
\ref{hcWWa}.
We similarly define $\tau(\WW)\in\oi^\cU$ 
as $(t(\FF,\WW))_{\FF\in\cU}$ for coloured graphons $\WW$.

Let $\WW=(W,\Psi)$ be a coloured graphon.
We see from \ref{hcWWa} that
for any sequence $\GG_n$ with $|\GG_n|\to\infty$,
we have
\begin{align}\label{lc3}
\GG_n\to\tau(\WW)
&\iff
\tau(\GG_n)\to\tau(\WW)\text{ in $\oi^\cU$}
\notag\\&
\iff
  t(\FF,\GG_n)\to t(\FF,\WW)  \quad 
\forall
\FF\in\cU.
\end{align}

Let $\FF\in\cL_k$ for some $k\ge1$.
As in the proof of \refL{Lmix}, let $X_i$ be as in 
the definition of $\GGx(\infty,\WW)$, and let
$L_i$ be the colour assigned to vertex $i$.
Then the pairs $\tX_i:=(X_i,L_i)$ are i.i.d.\  random variables in the space 
$\tgO:=\gO\times[\kk]$.
We may further assume that the edges in $\GGx(\infty,\WW)$
are given by the pairs $ij$ such that
$\xi_{ij}\le W(X_i,X_j)$ for some family $(\xi_{ij})_{i<j}$ of i.i.d.\
random variables, uniformly distributed on $\oi$ and independent of all
$X_i$ and $L_i$.
Let $n\ge k$ 
and let $i_1,\dots,i_k$ be distinct indices in $[n]$.
Then the indicator of the event that the injective map $[k]\to[n]$
given by  $s\mapsto i_s$ is a homomorphism $\FF\to \GGx(n,\WW)$
can be written as 
\begin{align}\label{fFF}
f\bigpar{(\tX_{i_s})_{1\le s\le k},(\xi_{i_s,i_t})_{1\le s < t\le k}}  
\end{align}
for some deterministic function $f$ (depending on $\FF$ and $\WW$).
Hence, $\tinj(\FF,\GGx(n,\WW))$ is the average of \eqref{fFF} over all 
sequences $i_1,\dots,i_k$ of distinct indices in $[n]$.
Note also that,
in particular, 
for the identity map $i_s=s$, we have
\begin{align}\label{fFF+}
f\bigpar{(\tX_{s})_{1\le s\le k},(\xi_{s,t})_{1\le s < t\le k}} 
= \indic{\FF\subseteq\GGx(k,\WW)}.
\end{align}
The distribution of $\GGx(n,\WW)$ is exchangeable, i.e.,
invariant under all permutations of $[n]$.
Furthermore, if we permute the indices in $[n]$ together 
with their corresponding
variables $\tX_i$ and $\xi_{ij}$, this will not affect
$\tinj(\FF,\GGx(m,\WW))$ for any $m\ge n$.
It follows that 
if $\cF_n$ is the \gsf{} generated by the variables
$\tinj(\FF,\GGx(m,\WW))$ for $m\ge n$, then
for any injective map $s\mapsto i_s$ as in \eqref{fFF}, we
have, using \eqref{fFF+},
\begin{align}
\E\bigsqpar{ f\bigpar{(\tX_{i_s})_{1\le s\le k},(\xi_{i_s,i_t})_{1\le s < t\le k}}  
\mid\cF_n}
&=
\E\bigsqpar{ f\bigpar{(\tX_{s})_{1\le s\le k},(\xi_{s,t})_{1\le s < t\le k}}  
\mid\cF_n}
\notag\\&
=\E\bigsqpar{\indic{\GGx(k,\WW)\supseteq\FF}\mid \cF_n}.  
\end{align}
Hence, taking the average over all such maps $s\mapsto i_s$,
we obtain
\begin{align}\label{pa2}
\tinj(\FF,\GGx(n,\WW))
=\E\bigsqpar{\tinj(\FF,\GGx(n,\WW))\mid\cF_n}
=\E\bigsqpar{\indic{\GGx(k,\WW)\supseteq\FF}\mid \cF_n}  
.\end{align}
Consequently, 
the sequence $\tinj(\FF,\GGx(n,\WW))$, $n\ge1$,
is a reverse martingale 
(for the decreasing sequence of \gsf{s} $\cF_n$). Hence,
by the martingale limit theorem,
there exists a random variable $Z$ such that, as \ntoo,
\begin{align}\label{limZ}
  \tinj(\FF,\GGx(n,\WW))\asto Z,
\end{align}
where furthermore, for every $n\ge1$,
\begin{align}\label{limZE}
\E Z =\E\bigsqpar{ \tinj(\FF,\GGx(n,\WW))}.
\end{align}
Recall again that $\tinj(\FF,\GGx(n,\WW))$ is the average of \eqref{fFF} over
sequences $i_1,\dots,i_k$ of distinct indices in $[n]$.
For any fixed $m$, the fraction of such sequences having some $i_k\le m$
is $O(m/n)=o(1)$ as \ntoo. Hence, the limit in \eqref{limZ} is the same if
we  only average \eqref{fFF} over $i_1,\dots,i_k\in [n]\setminus [m]$.
Consequently, $Z$ is independent of all $\tX_i$ and $\xi_{ij}$ for 
$1\le i \le m$.
Since this holds for every $m$, 
$Z$ is independent of all $\tX_i$ and $\xi_{ij}$, and thus of itself;
hence $Z$ is a.s.\ constant. 
(This is an instance of the Kolmogorov zero-one theorem
\cite[Theorem 1.5.1]{Gut}; more formally, we have shown that for any
$t\in\bbR$, the event \set{Z\le t} belongs to the tail \gsf{} defined 
in the natural way by the families $\tX_i$ and $\xi_{ij}$, and thus 
$\P(Z\le t)=0$ or 1.) 

Consequently, $Z=\E Z$ a.s., and recalling 
\eqref{limZ}, \eqref{limZE}, \eqref{pa2},
and \eqref{pa1a} 
we have
\begin{align}\label{pa3}
  \tinj(\FF,\GGx(n,\WW))
&\asto
\E\bigsqpar{\tinj(\FF,\GGx(k,\WW))}
=\P\bigpar{\GGx(k,\WW)\supseteq\FF}
=t(\FF,\WW)
.\end{align}
(The proof of \eqref{pa3} just given 
is a minor variation of 
a proof of the law of large numbers for $U$-statistics, 
see \cite{Berk} and \cite[Section 10.16.2]{Gut},
which in turn is an extension
of a well-known proof of the standard law of large numbers
using reverse martingales, 
see e.g.\ \cite[Section 10.16.1]{Gut}.)

By \eqref{pa3} and \eqref{tid}, we have
for every $\FF\in\cL$, and thus also for every $\FF\in\cU$ (since the labels
do not matter),
almost surely
\begin{align}\label{pa4}
t\bigpar{\FF,\GGx(n,\WW)}\to t(\FF,\WW).  
\end{align}
Since $\cU$ is a countable set of graphs, it follows that a.s.\  \eqref{pa4}
holds for all $\FF\in\cU$, and thus $\GGx(n,\WW)\to\tau(\WW)$ by \eqref{lc3}.

In particular, there exists a sequence $\GG_n$ with $|\GG_n|\to\infty$
such that $\tau(\GG_n)\to\tau(\WW)$, and thus
$\tau(\WW)\in\hcWW$ (as defined in \ref{hcWWa}).
In other words, $\tau(\WW)$ is a graph limit, which we now simply denote by
$\WW$.

This is the sought graph limit. By \refD{D=} with \refL{L=}\ref{L=t},
and the definition \eqref{tau},
$\WW_1\cong\WW_2 \iff\tau(\WW_1)=\tau(\WW_2)$, which shows \ref{TW2}.
Furthermore, \ref{TW3} follows from
\eqref{lc3}, and  \ref{TW4} from \eqref{pa4}.

It remains to show \ref{TW1}.
For this we use, at least implicitly,  the representation \ref{hcWWb} above. 
More explicitly, given a graph limit $\hWW$, 
choose a sequence $\GG_n$ of  coloured graphs with $|\GG_n|\to\infty$ such that
$\GG_n\to\hWW$. Consider the random infinite graphs $\hGG_n$ defined above by
random labelling of $\GG_n$. By \refL{L1}, $\hGG_n$ converges in
distribution (as random elements of $\cLoo$); the limit is a random infinite
graph in $\cLoo$ which we denote by $\hGG_\infty$.
For any $m\in\bbN$, the distribution of $\hGG_n$ is invariant under
any permutation of the first $m$ vertices for all $n$ that are so large that
$|\hGG_n|\ge m$; consequently, the distribution of the limit $\hGG_\infty$
is invariant under any permutation of finitely many vertices, which means
that the distribution of $\hGG_\infty$ is exchangeable.
The infinite random graph $\hGG_\infty$
is described by the random variables $L_i$ (the colour of $i$) 
and $I_{ij}$ (the indicator of an edge $ij$), for $i\in\bbN$ and
$1\le i<j<\infty$ respectively; 
the combined array $\sqpar{(L_i)_i,(I_{ij})_{i,j}}$ is thus
exchangeable;
more precisely, it is a jointly exchangeable array of the type in 
\cite[Chapter 7]{Kallenberg:symmetries}, which 
allows for combinations of arrays with different numbers of indices as in
our case.
The representation theory for exchangeable arrays is due to 
\cite{Aldous} and \cite{Hoover} and further developed by 
\cite{Kallenberg:symmetries}.
We use here \cite[Theorem 7.15 (or Theorem 7.22)]{Kallenberg:symmetries}
which in our case shows that there exist independent random variables $\xi_\emptyset$,
$\xi_i$ ($i\in\bbN$), and $\xi_{ij}$ $(1\le i<j<\infty)$, all uniformly
distributed on $\oi$,
and measurable functions $f_1:\oi^2\to[\kk]$ and $f_2:\oi^4\to\bbR$
such that
\begin{align}\label{ok1}
  L_i& = f_1(\xioo,\xi_i),
\\\label{ok2}
I_{ij}&=f_2(\xioo,\xi_i,\xi_j,\xi_{ij}),
\end{align}
and furthermore $f_2$ is symmetric in the second and third arguments.
(This theorem in \cite{Kallenberg:symmetries} considers more generally an
exchangeable array $X_J$ indexed by arbitrary finite subsets $J\subset\bbN$.
We let $X_J:=0$ for $|J|>2$.)
Here the variable $\xioo$ represents a global randomization that is required
in general (in our setting it corresponds to a random graphon, see e.g.\
\cite{SJ209}), but is not needed here.
In fact, it is easily seen that if $A$ and $B$ are two disjoint finite
subsets of $\bbN$, then the restrictions of $\hGG_n$ to $A$ and $B$ are
asymptotically independent, 
and thus the restrictions of $\hGG_\infty$ to $A$ and $B$ are
independent. It follows by 
\cite[Lemma 7.35]{Kallenberg:symmetries}, see also
\cite[Theorem 5.5]{SJ209},
that we do not need $\xioo$ in \eqref{ok1}--\eqref{ok2};
we may thus write the functions as $f_1(\xi_i)$ and
$f_2(\xi_i,\xi_j,\xi_{ij})$.

By \eqref{ok1}--\eqref{ok2}, the function $f_1$ takes its values in $[\kk]$
and function $f_2$ takes its values in $\setoi$.
We define, with $\xi\in U(0,1)$,
\begin{align}\label{ar1}
  W(x,y):=\P\bigpar{f_2(x,y,\xi)=1}=\intoi f_2(x,y,z)\dd z.
\end{align}
Then, $W$ is an uncoloured graphon (defined on $\oi$), and
$\WW:=(W,f_1)$ is a pure $\kk$-coloured graphon.
Furthermore, given $(\xi_i)\xoo$, \eqref{ok2} and \eqref{ar1}
imply that 
\begin{align}\label{ar2}
  \P\bigpar{I_{ij}=1\mid (\xi_i)\xoo}
=\P\bigpar{f_2(\xi_i,\xi_j,\xi_{ij})=1\mid (\xi_i)\xoo}
=W(\xi_i,\xi_j),
\end{align}
and these events are (conditionally) independent for all pairs
$(i,j)\in\bbN\pairs$.
It follows that the infinite random coloured graph $\GGx(\infty,\WW)$ has the
same distribution as $\hGG_\infty$. (Take $X_i:=\xi_i$ in the construction
of $\GGx(\infty,\WW)$ and compare with \eqref{ok1} and \eqref{ar2}.)
Hence, by the definition of $\hGG_\infty$ above as a limit of $\hGG_n$,
we have
$ \hGG_n\to \GGx(\infty,\WW)$ in distribution as \ntoo.
This means that for every coloured graph $\FF\in\cL_k$ for some $k\ge1$,
we have
\begin{align}\label{ffgg}
  \P(\FF\subseteq\hGG_n) \to \P\bigpar{\FF\subseteq \GGx(\infty,\WW)}.
\end{align}
If $n$ is so large that $|\GG_n|\ge k$, then 
$\P(\FF\subseteq\hGG_n) =\P(\FF\subseteq\GG_n[k])$,
which together with \eqref{tinj} and \eqref{pa1a} 
shows that \eqref{ffgg} is equivalent to
$\tinj(\FF,\GG_n)\to t(\FF,\WW)$, and thus, using \eqref{tid},
to $t(\FF,\GG_n)\to t(\FF,\WW)$. Hence, \ref{TW3} shows that $\GG_n\to\WW$,
and thus the coloured graph limit $\hWW$ is represented by the coloured
graphon $\WW$. 
\end{proof}

\refT{TW}
gives us a third way to
define the set of graph limits, different from \ref{hcWWa} and \ref{hcWWb}
in \refD{DhcWW} but equivalent to them.

\begin{theorem}\label{TWB}
The set\/ $\hcWW_\kk$ of $\kk$-coloured
graph limits can be identified with
the set of all $\kk$-coloured graphons modulo the equivalence relation
$\cong$ in \refD{D=}; 
furthermore, we may restrict to coloured graphons defined on $\oi$, and to
pure coloured graphons without affecting the resulting space of graph limits.
\end{theorem}

\begin{proof}
  An immediate consequence of \refT{TW}
 together with \refLs{Lmix} and \ref{Loi}.
\end{proof}

\begin{corollary}\label{CTWB}
  The definition in \refT{TWB} of coloured graph limits as equivalence classes 
of coloured graphons means that
$t(\FF,\WW)$ (for $\FF\in\cL$) and 
$\GGx(n,\WW)$ (for $n\in\bbNx$) are well defined for coloured graph limits
$\WW\in\hcWW_\kk$. 
\end{corollary}
\begin{proof}
  An immediate consequence of \refD{D=} and \refL{L=}.
\end{proof}

\begin{remark}\label{Rdef}
As in the standard uncoloured case, it is often convenient to identify a
graphon with the corresponding graph limit, and we may do so without comment
in the sequel.
When it is necessary to be more formal, in particular when it is important to
distinguish between equivalent graphons, we use the definition in \refT{TW}
  and regard graph limits as equivalence classes of graphons.
\end{remark}

It follows from \eqref{tw3} that the embedding in \ref{hcWWa} 
of $\hcWW$ in $\oi^\cU$
is given by $\WW\mapsto(t(\FF,\WW))_{\FF\in\cU}$,
i.e., by extending the map $\tau$ in \eqref{tau} to graphons and graph limits.
Recall that the set $\hcWW$ is a compact metric space, with the topology
defined by the embedding $\tau$ (or $\tauu$).
This topology can be described in terms of the homomorphism numbers
$t(\FF,\WW)$: 
\begin{corollary}\label{CTWBb}
Let\/ $\WW_n$, $n\ge1$, and $\WW$ be coloured graph limits (or graphons).
Then
\begin{align}\label{tw9}
\WW_n\to\WW
\iff
  t(\FF,\WW_n)\to t(\FF,\WW)  \quad 
\forall
\FF\in\cU.
\end{align}
\end{corollary}
\begin{proof}
  Immediate from the definitions and comments above.
\end{proof}

\begin{remark}\label{Rmetric}
\refC{CTWBb} implies, 
since the set $\cU_\kk$ of $\kk$-coloured graphs is countable, 
that we can define
a metric on the space $\hWW_\kk$ of coloured graph limits by the following
standard construction: Let $\FF_1,\FF_2,\dots$ be an enumeration of the
unlabelled $\kk$-coloured graphs, and define
\begin{align}
  d(\WW_1,\WW_2):=\sumi 2^{-i}|t(\FF_i,\WW_1)-t(\FF_i,\WW_2)|.
\end{align}
(Cf.\ \refD{DhcWW}\ref{hcWWa}.)
However, as in other similar cases, this metric is not very interesting;
its main use is to show the existence of a metric.  
\end{remark}

\subsubsection{Uncolouring}\label{SSuncolour}

We define a mapping $\iota$, which we call \emph{uncolouring}, 
from the set $\hW$ of 
coloured graphons to uncoloured graphons by ignoring the colouring, i.e.,
$\iota(W,\Psi):=W$.
We record some simple facts.

\begin{theorem}\label{Tuncolour}
If two coloured graphons $\WW_1$ and $\WW_2$ are equivalent, then so are
the uncoloured graphons $\iota(\WW_1)$ and $\iota(\WW_2)$.

Hence, uncolouring is well defined for graph limits (regarded as equivalence
classes as in \refT{TWB}), and defines for every $\kk\in\bbN$ a mapping 
$\hcWW_\kk\to\hcWo$ of coloured graph limits to uncoloured ones.
Furthermore, this mapping is continuous.
\end{theorem}

\begin{proof}
  It follows from the
definition \eqref{pa1a}--\eqref{pa1b} that for any coloured graphon
$\WW=(W,\Psi)$ and any uncoloured graph $F$, we have
\begin{align}
  t(F,\iota(\WW))=t(F,W)
= \sum_{\FF} t(\FF,\WW),
\end{align}
summing over all $|F|^\kk$ colourings $\FF$ of $F$.
The results now follow easily from the definitions above, see \refD{D=} and
\refC{CTWB}.
\end{proof}

\subsubsection{Pullbacks and equivalence for coloured graphons} 

In the uncoloured case, it is well-known that convergence of graphons in
the space of uncoloured graph limits
$\hcWo$ also is described by the cut distance; see Appendix
\ref{app:uncoloured-pullback} for standard facts on pullbacks, cut norm and
cut distance of uncoloured graphons. In this
subsubsection we consider the problem of
extending this to coloured graphons. 

Let $\gf:\gO_1\to\gO_2$ be a \mpp{} map.
 If $\WW:=(W,\Psi)$ is a coloured graphon on $\gO_2$, we define
$\WW^\gf:=(W^\gf,\Psi^\gf)$; this is a coloured graphon on $\gO_1$.
If $X_i$ are i.i.d.\  points in $\gO_1$, then $\gf(X_i)$ are i.i.d.\  points in
$\gO_2$, and it follows from the definitions that $\GGx(n,\WW)\eqd \GGx(n,\WW^\gf)$
for every $n\le\infty$; thus 
by \refL{L=}\ref{L=n} and  \refD{D=},
\begin{align}\label{ar77}
\WW^\gf\cong\WW.  
\end{align}
In other words, by 
\refL{L=}\ref{L=t},
\begin{align}\label{ar8}
  t(\FF,\WW^\gf)=t(\FF,\WW),\qquad \FF\in\cU.
\end{align}

\begin{example}
  \label{EG2}
Let $\GG$ be a coloured graph, and recall the coloured graphons $\WW'_G$
and $\WW_G$
defined in \refE{EG1}. Then $\WW_G$ is a pullback of $\WW'_G$, and thus
$\WW'_G\cong\WW_G$.
\end{example}

For uncoloured graphons, as said above, equivalence of two graphons does not 
imply that one of them is the pullback of the other, but we have the
substitutes in \refT{Tunique0}. For coloured graphons, we note first that
\refT{Tunique0} cannot generalise to mixed coloured graphons. Indeed, let $\WW$
be any coloured graphon that is not pure-coloured. By \refL{Lmix},
$\WW\cong \WW'$ for some pure-coloured graphon $\WW'$. However, it is clear
that if $\WW_1=\WW_2^\gf$ for two coloured graphons $\WW_1,\WW_2$
and a \mpp{} map $\gf$, then either both or none of $\WW_1$ and $\WW_2$ are
pure-coloured. Hence, nothing similar to \refT{Tunique0} can hold for
arbitrary mixed-coloured graphons.

However, if we restrict to pure-coloured graphons, then
\refT{Tunique0}\ref{Tunique0BR} extends.
We conjecture that \refT{Tunique0}\ref{Tunique0BCL} too holds for pure
coloured graphons, but we leave this as an open problem.

\begin{theorem}
\label{Tunique1}
  Let\/ $\WW=(W,\chi)$ and $\tWW=(\tW,\tchi)$ 
be pure coloured graphons defined on 
$\oi$ such that $\WW\cong \tWW$. 
Then there exist \mpp{} maps $\gf,\tgf:\oi\to\oi$ such that
$\WW^{\gf}=\tWW^{\tgf}$ a.e.
\end{theorem}

\begin{proof}
(This is based on the proof for the uncoloured case in \cite[Theorem 7.1]{SJ209}.)
  The random graph $\GG(\infty,\WW)$ is described by the sequence of vertex
  colours and the array of edge indicators $I_{ij}$, $1\le i<j$.
Let $(\xi_i)_{i\ge1}$ and $(\xi_{ij})_{1\le i<j}$ be a sequence and an array
of random \Uoi{} variables, all independent. In the construction of
$\GG(\infty,\WW)$ in \refSS{SSGG}, we may take $X_i=\xi_i$ and 
\begin{align}\label{ok90}
  I_{ij}:=\indic{\xi_{ij}\le W(\xi_i,\xi_j)}. 
\end{align}
We do the same for $\GG(\infty,\tWW)$. Thus, using \refL{L=}\ref{L=oo},
 the assumption $\WW\cong\tWW$ means
 \begin{align}\label{ok91}
\Bigpar{ (\chi(\xi_i))_{i},\bigpar{\indic{\xi_{ij}\le W(\xi_i,\xi_j)}}_{i,j}}
\eqd
\Bigpar{ (\tchi(\xi_i))_{i},\bigpar{\indic{\xi_{ij}\le \tW(\xi_i,\xi_j)}}_{i,j}}
. \end{align}
The two sides are jointly exchangeable arrays of the type in 
\cite[Chapter 7]{Kallenberg:symmetries}.
(Cf.~the proof of \refT{TW}.)
We apply
\cite[Theorem 7.28(i)$\implies$(ii)]{Kallenberg:symmetries} to \eqref{ok91}
and obtain (after unwrapping the rather heavy notation)
that there exist functions $g_1,\tg_1:\oi\to\oi$ and
$g_2,\tg_2:\oi^3\to\oi$, all of them \mpp{} if all but the last variable is
fixed, such that, for a.e.\ $\xi_i,\xi_j,\xi_{ij}$,
\begin{align}\label{ok92}
  \chi(g_1(\xi_i))&=\tchi(\tg_1(\xi_i))
\\\label{ok93}
\bigindic{g_2(\xi_i,\xi_j,\xi_{ij})\le W(g_1(\xi_i),g_1(\xi_j))}
&=
\bigindic{\tg_2(\xi_i,\xi_j,\xi_{ij})\le \tW(\tg_1(\xi_i),\tg_1(\xi_j))}.
\end{align}
Actually, we have simplified the result of 
\cite[Theorem 7.28(i)]{Kallenberg:symmetries}  a little; the theorem really
yields the above with functions $g_1,\tg_1,g_2,\tg_2$ having an extra
argument $\xi_\emptyset$, which is another $\Uoi$ random variable,
independent of everything else. 
(Cf.\ the proof of \refT{TW}.)
However, since \eqref{ok92}--\eqref{ok93}
with this modification hold for a.e.\ $\xi_\emptyset,\xi_i,\xi_j,\xi_{ij}$,
we can choose a fixed value of $\xi_\emptyset$ such that the formulas hold 
for a.e.\ $\xi_i,\xi_j,\xi_{ij}$, which yields the version
stated in \eqref{ok91}--\eqref{ok92}.

Take the conditional expectation of both sides of \eqref{ok92} given $\xi_i$
and $\xi_j$. 
Since $g_2$ is \mpp{} in the last variable, we have for any fixed $\xi_i$
and $\xi_j$ that $g_2(\xi_i,\xi_j,\xi_{ij})\eqd\xi_{ij}$, and thus
\begin{align}
  \E \bigsqpar{\bigindic{g_2(\xi_i,\xi_j,\xi_{ij})\le  W(g_1(\xi_i),g_1(\xi_j))}
\mid \xi_i,\xi_j}&
=\P\bigsqpar{\xi_{ij}\le W(g_1(\xi_i),g_1(\xi_j))\mid\xi_i,\xi_j}
\notag\\&
=W(g_1(\xi_i),g_1(\xi_j)).
\end{align}
and similarly for the \rhs.
Consequently, for a.e.\ $\xi_i,\xi_j$,
\begin{align}\label{ok95}
  W(g_1(\xi_i),g_1(\xi_j))
= \tW(\tg_1(\xi_i),\tg_1(\xi_j)).
\end{align}
Let $\gf:=g_1$ and $\tgf:=\tg_1$, recalling that these are \mpp; then
\eqref{ok92} and \eqref{ok95} yield, a.e., 
\begin{align}
  \WW^\gf = (W^\gf,\chi^\gf)
=(\tW^{\tgf},\tchi^{\tgf})
=\tWW^{\tgf},
\end{align}
which completes the proof.
\end{proof}

\subsubsection{Simply coloured graphons}

We can obtain a ``uniqueness'' result similar to \refT{Tunique1} 
for another class of coloured graphons.
We consider for simplicity in this subsection only graphons defined on
$\oi$,
except when we say otherwise.
Define, for $x\in\oi$,
\begin{align}\label{Wx}
  W_x(\cdot):=W(x,\cdot).
\end{align}
Thus, every $W_x$ is a measurable function $\oi\to\oi$, and we may regard
$W_x$ as an element of the separable Banach space $L^1\oi$, which we equip
with its Borel \gsf.
We let $\cF_W$ be the \gsf{} generated by the map $x\mapsto W_x$, i.e., the
smallest \gsf{} on $\oi$ such that $x\mapsto W_x$ is measurable
$(\oi,\cF_W)\to L^1\oi$. 
It is well-known 
(and follows by a monotone-class argument;
see also \cite[Lemma III.11.16]{DS}) 
that
$x\mapsto W_x$ is measurable for the Lebesgue \gsf{} $\cLeb$ on $\oi$;
hence $\cF_W\subseteq\cLeb$.
(If $W$ is Borel measurable on $\oi^2$, as we may assume, then 
$x\mapsto W_x$ is measurable for the Borel \gsf{} $\cB$ on $\oi$ and
$\cF_W\subseteq\cB$.) 

\begin{definition}
A \emph{simply} coloured graphon  is a coloured graphon $\WW=(W,\Psi)$
such that the colouring function $\Psi$ is $\cF_W$-measurable.
In this case, we also say that the colouring  $\Psi$ is \emph{simple}.
\end{definition}

\begin{lemma}\label{LSimple}
A coloured graphon is simple if and only if there exists a measurable
function 
$h:L^1\oi \to \gD_\kk$ such that
\begin{align}
    \Psi(x) = h(W_x),
\qquad x\in\oi.
\end{align}
\end{lemma}
In other words, a simple colouring is determined by the section $W_x$ of the
graphon. 

\begin{proof}
  This is a 
  standard fact from measure theory,
see e.g.\
  \cite[Lemma 1.13]{Kallenberg}.
\end{proof}

We next show that every coloured graphon is equivalent to a simply coloured
graphon, by an easy construction.

\begin{theorem}\label{TBWW}
  Let\/ $\WW=(W,\Psi)$ be a coloured graphon, and define a new colouring by
  \begin{align}\label{tbww}
    \bPsi=(\bpsi_i)_1^\kk:=\E(\Psi\mid\cF_W).
  \end{align}
Then $\WW\cong \bWW:=(W,\bPsi)$.
Consequently, every coloured graphon is equivalent to some simply coloured
graphon. 
\end{theorem}
\begin{proof}
  It is clear that \eqref{tbww} defines an $\cF_W$-measurable map
  $\oi\to\gD_\kk$, i.e., a simple colouring of $W$.
Let $\FF\in\cL$ and $n:=|\FF|$, and define, for any (measurable) functions $h_1,\dots,h_n:\oi\to\oi$, 
\begin{align}\label{tid1}
\gU(h_1,\dots,h_n):=
\int_{\gO^n} \prodin h_i(x_i) \prod_{ij\in  E(\FF)} W(x_i,x_j)
\dd x_1\dotsm \dd x_n.
\end{align}
Thus, the definition \eqref{pa1b} says, letting again $\ell_i$ be the colour
of vertex $i\in\FF$, 
\begin{align}\label{tid2}
t(\FF,\WW)=\gU(\psi_{\ell_1},\dots,\psi_{\ell_n}),
\qquad
t(\FF,\bWW)=\gU(\bpsi_{\ell_1},\dots,\bpsi_{\ell_n}).
\end{align}
By Fubini's theorem, we can rewrite \eqref{tid1} as
\begin{multline}\label{tid3}
\gU(h_1,\dots,h_n)=
\intoi h_1(x_1)
\int_{\oi^{n-1}} 
\prod_{1j\in  E(\FF)} W_{x_1}(x_j)
\prod_{i=2}^n h_i(x_i) 
\prod_{\substack{ij\in  E(\FF):\\2\le i<j\le n}} W(x_i,x_j)
\\
\cdot\dd x_2\dotsm \dd x_n \cdot \dd x_1
\end{multline}
Denote the inner integral (over $\oi^{n-1}$) by $G(x_1)$. \refL{LTD} below
shows that $G$ is $\cF_W$-measurable.
Hence, if we define $\bh_i:=\E(h_i\mid\cF_W)$, we have, 
using \eqref{tid3} twice,
\begin{align}\label{tid4}
  \gU(h_1,h_2,\dots,h_n)
=\intoi h_1(x_1)G(x_1)\dd x_1
=\intoi \bh_1(x_1)G(x_1)\dd x_1
=  \gU(\bh_1,h_2,\dots,h_n).
\end{align}
By symmetry we may continue and exchange each $h_i$ for $\bh_i$, one by one;
this yields
\begin{align}\label{tid5}
  \gU(h_1,h_2,\dots,h_n)
=  \gU(\bh_1,\bh_2,\dots,\bh_n).
\end{align}
Now choose $h_i=\psi_{\ell_i}$; then \eqref{tid5} and \eqref{tid2} yield
\begin{align}
  t(\FF,\WW)=t(\FF,\bWW).
\end{align}
Since $\FF$ is arbitrary, 
this proves $\WW\cong\bWW$.
\end{proof}

Before proving \refL{LTD} used in the proof above, we give another lemma.
\begin{lemma}\label{Lgsf}
The following define the same \gsf{} on
$L^\infty\oi$:
\begin{romenumerate}
\item \label{Lgsf1}
The restriction of the Borel \gsf{} in $L^1\oi$.
\item \label{Lgsfw}
The restriction of the weak \gsf{} in $L^1\oi$, i.e., 
the \gsf{} generated by the family of mappings $f\mapsto\intoi fg$ for 
$g\in L^\infty\oi$.
\item \label{Lgsfw*}
The weak-$*$ \gsf{} in $L^\infty\oi$, i.e., 
the \gsf{} generated by the family of mappings $f\mapsto\intoi fg$ for 
$g\in L^1\oi$.
\end{romenumerate}
\end{lemma}

\begin{proof}
Denote (temporarily) the three \gsf{s} by $\cF$, $\cF_w$, and $\cF_{w*}$.
  The Borel \gsf{} and the weak \gsf{} are the same on $L^1\oi$ (as on every
  separable Banach space) by the Pettis measurability  theorem
 \cite[Theorem III.6.11]{DS}, and thus $\cF=\cF_w$.

The \gsf{} $\cF_w$ is included in $\cF_{w*}$, since it is generated by a
smaller family of mappings.
Conversely, if $g\in L^1\oi$, let $g_n\in L^\infty\oi$ with $g_n\to g$ in $L^1$.
Then the mappings $f\mapsto\intoi fg_n$ are $\cF_{w}$-measurable 
and converge pointwise on $L^\infty\oi$ to $f\mapsto\intoi fg$;
hence the latter is $\cF_{w}$-measurable, and thus 
$\cF_{w*}\subseteq\cF_w$.
\end{proof}

\begin{lemma}\label{LTD}
  Let\/ $W$ be an uncoloured graphon, let $0\le k\le n$, and let
  $H:\oi^n\to\bbR$ be an integrable function. Then the function
  \begin{align}\label{ltd}
 G(x):=\int_{\oi^n}\prod_{j=1}^kW_x(y_j) H(y_1,\dots,y_n)\dd y_1\dotsm\dd y_n
  \end{align}
is $\cF_W$-measurable $\oi\to\bbR$.
\end{lemma}
\begin{proof}
  Consider first the case when $H$ factorizes as
  \begin{align}\label{ar4h}
H(y_1,\dots,y_n)=h_1(y_1)\dotsm h_n(y_n)
  \end{align}
  for some functions $h_1,\dots,h_n\in L^1\oi$.
Then
\begin{align}
  G(x)=\prod_{i=1}^k \intoi W_x(y_j) h(y_j)\dd y_j 
\cdot \prod_{i=k+1}^n\intoi h(y_j)\dd y_j 
\end{align}
where each factor is $\cF_W$-measurable by the definition of $\cF_W$ and
\refL{Lgsf}.
Hence $G(x)$ is $\cF_W$-measurable.

Next, it follows by linearity
that the lemma holds for every $H$ that is a linear
combination of such products \eqref{ar4h}. Let $L$ be the subspace of 
$L^1=L^1(\oi^n)$ consisting of such linear combinations.
If $H\in \bL$ (the closure of $L$ in $L^1$), then there is a sequence
$H_n\in L$ such that $\norm{H_n-H}_{L^1}\to0$. It follows from \eqref{ltd}
that for the corresponding integrals $G_n(x)$ we have
$G_n(x)\to G(x)$ for all $x\in\oi$ (in fact, uniformly).
We have just shown that each $G_n(x)$ is $\cF_W$-measurable,
and consequently $G(x)$ is $\cF_W$-measurable.

Finally, we claim that $\bL=L^1$, so that the last case covers all 
$H\in L^1(\oi^n)$.
This is well-known and can easily be verified in several ways. 
For example, it suffices to consider the case $H=\etta_A$ for a measurable
set $A$ in the product; this case follows easily from the monotone class
theorem. 
Alternatively, for any two $\gs$-finite measures $\mu_1$ and
$\mu_2$ on some measure spaces $\gO_1$ and $\gO_2$,
the projective tensor product
$L^1(\gO_1,\mu_1)\hat\otimes L^1(\gO_2,\mu_2)
\cong L^1(\gO_1\times\gO_2,\mu_1\times\mu_2)$, see e.g.\
\cite[Exercise 2.8]{Ryan}; by
induction this extends to any finite number of $L^1$-spaces.
This means more explicitly \cite[Proposition 2.2]{Ryan} 
the stronger statement
that convex combinations of products \eqref{ar4h} with all 
$\norm{h_i}_{L^1(\oi)}\le1$ are dense in the unit ball of $L^1(\oi^n)$.
\end{proof}

This completes the proof of \refT{TBWW}, showing that
every coloured graphon is equivalent to some simply coloured
graphon. If we restrict to simply coloured graphons,
then we also have a ``uniqueness'' result of the same type as
\refTs{Tunique0}\ref{Tunique0BR} and \ref{Tunique1}.

\begin{theorem}
\label{Tunique2}
  Let\/ $\WW=(W,\Psi)$ and $\tWW=(\tW,\tPsi)$ 
be simply coloured graphons defined on 
$\oi$ such that $\WW\cong \tWW$. 
Then there exist \mpp{} maps $\gf,\tgf:\oi\to\oi$ such that
$\WW^{\gf}=\tWW^{\tgf}$ a.e.
\end{theorem}

We prove first two lemmas.
We use the construction and notation
in \refT{TBWW}. 

\begin{lemma}\label{LS1}
Let\/ $\WW=(W,\Psi)$ be a coloured graphon on $\oi$, and let $\gf:\oi\to\oi$ be
\mpp. Then
$\bWW^\gf=\overline{\WW^\gf}$.
\end{lemma}
\begin{proof}
We have, for $x\in\oi$, by the definition \eqref{Wx},
\begin{align}\label{db1}
 (W^\gf)_x=W_{\gf(x)}\bigpar{\gf(\cdot)}=\bigpar{W_{\gf(x)}}^\gf.
\end{align}
The map $f\mapsto f^\gf$ is an isometry of $L^1\oi$ into $L^1\oi$,
and thus a bimeasurable bijection of $L^1\oi$ onto its image.
Hence, \eqref{db1} implies that the \gsf{} $\cF_{W^\gf}$ is generated by the
map $x\mapsto W_{\gf(x)}$ into $L^1\oi$, which implies  that 
\begin{align}\label{db2}
  \cF_{W^\gf} = \set{\gf\qw(A)\subseteq\oi: A\in \cF_W}.
\end{align}
Since $\bPsi$ is $\cF_W$-measurable by the definition \eqref{tbww}, 
it follows from \eqref{db2} that $\bPsi^\gf$ is $\cF_{W^\gf}$-measurable.
Furthermore, if $A\in\cF_W$, then
\begin{align}
  \int_{\gf\qw(A)}\bigpar{\Psi^\gf(x)-\bPsi^\gf(x)}\dd x
&=
 \intoi \etta_{A}(\gf(x))\bigpar{\Psi^\gf(x)-\bPsi^\gf(x)}\dd x
\notag\\&
=
 \intoi \etta_{A}(y)\bigpar{\Psi(y)-\bPsi(y)}\dd y
=0.
\end{align}
Consequently, 
\begin{align}
\overline{\Psi^\gf}
=\E\bigpar{\Psi^\gf\mid\cF_{W^\gf}}
=\bPsi^\gf,
\end{align}
and thus, by the definitions,
\begin{align}
  \bWW^\gf
=(W^\gf,\bPsi^\gf)
=(W^\gf,\overline{\Psi^\gf})
=\overline{(W^\gf,\Psi^\gf)}
=\overline{\WW^\gf}.
\end{align}
\end{proof}
In the next lemma we consider also graphons defined on the probability space
$(\oi\times[\kk],\tmu)$ with the probability measure $\tmu$ in \eqref{tmu}
whose (first) marginal is Lebesgue measure $\gl$ on $\oi$.
This is a Borel space with an atomless measure, and is thus equivalent to
$(\oi,\gl)$ by an a.e.\ defined bijection, see \refR{R01+}. 
Hence the results above hold
also for such graphons, with only notational differences.

\begin{lemma}\label{LS2}
  Let\/ $\WW=(W,\Psi)$ be a coloured graphon on $\oi$,
and let $\WW'$ be the pure coloured graphon on $\tgO:=\oi\times[\kk]$
defined in the proof of \refL{Lmix}.
Then $\overline{\WW'} = \bWW^{\pi_1}$ a.e., where $\pi_1$ is the projection
$\tgO\to\oi$ onto the first coordinate.
\end{lemma}

\begin{proof}
Let $\WW=(W,\Psi)$.
We have by the definitions of $\WW'$ and $\bWW$,
with $\pi_1$ and $\pi_2$ as in the proof of \refL{Lmix},
\begin{align}\label{db6}
  \overline{\WW'}=\overline{\bigpar{W^{\pi_1},\pi_2}}
=\bigpar{W^{\pi_1},\overline{\pi_2}}
=\bigpar{W^{\pi_1},\E\xpar{\pi_2\mid\cF_{W^{\pi_1}}}}
\end{align}
and, using  also \refL{LS1},
\begin{align}\label{db7}
  \overline{\WW^{\pi_1}}
=
 \overline{\WW}^{\pi_1}
=\bigpar{W^{\pi_1},\overline{\Psi}^{\pi_1}}
,\end{align}
where thus $\overline{\Psi^{\pi_1}}=\overline{\Psi^{\pi_1}}$
is $\cF_{W^{\pi_1}}$-measurable.

By \eqref{db2}, 
\begin{align}
\cF_{W^{\pi_1}}=\set{\pi_1\qw(A):A\in\cF_W}  
=\set{A\times[\kk]:A\in\cF_W}. 
\end{align}
We have by \eqref{tmu}, for any $A\in\cF_W$, with $\ee_\ell$ the $\ell$-th
unit vector,
\begin{align}
  \int_{A\times[\kk]} \pi_2\dd\tmu
=\int_A \sum_{\ell=1}^\kk\psi_\ell(x)\ee_\ell \dd x
=\int_A\Psi(x) \dd x
=\int_A\overline\Psi(x) \dd x
=\int_{A\times[\kk]}{\overline\Psi}^{\pi_1}\dd \tmu.
\end{align}
Hence, 
\begin{align}
\overline{\pi_2}:=\E\bigpar{\pi_2\mid\cF_{W^{\pi_1}}}
=  {\overline\Psi}^{\pi_1},
\end{align}
and the result follows from \eqref{db6} and \eqref{db7}.
\end{proof}

\begin{proof}[Proof of \refT{Tunique2}]
Construct pure coloured graphons $\WW'$ and $\tWW'$ 
by the construction in
\refL{Lmix}; these are both defined on $\gO:=\oi\times[\kk]$ with
possibly different measures that we now denote by $\mu$ and $\tmu$.
Then $\WW'\cong\WW\cong\tWW\cong\tWW'$. Hence, \refT{Tunique1} applies
and shows the existence of \mpp{} maps $\tau:\oi\to(\gO,\mu)$ and
$\ttau:\oi\to(\gO,\tmu)$ such that (a.e.)
\begin{align}\label{dby}
(\WW')^{\tau} = (\tWW')^{\ttau}.  
\end{align}
We apply the operation $\WW\to\bWW$ in \refT{TBWW} and obtain by 
\refLs{LS1} and \ref{LS2}, and the assumption that $\WW$ is simply coloured
and thus $\bWW=\WW$,
\begin{align}\label{dby1}
  \overline{\bigpar{\WW'}^{\tau}}
= \bigpar{\overline{\WW'}}^{\tau}
= \bigpar{\overline{\WW}^{\pi_1}}^{\tau}
=\bWW^{\pi_1\circ\tau}
=\WW^{\pi_1\circ\tau},
\end{align}
and similarly 
\begin{align}\label{dby2}
  \overline{\bigpar{\tWW'}^{\ttau}}
=\tWW^{\pi_1\circ\ttau}.
\end{align}
Hence, \eqref{dby} implies
\begin{align}\label{tekla}
\WW^{\pi_1\circ\tau}
= \overline{\bigpar{\WW'}^{\tau}}
= \overline{\bigpar{\tWW'}^{\ttau}}
=\tWW^{\pi_1\circ\ttau},
\end{align}
where $\pi_1\circ\tau$ and $\pi_1\circ\ttau$ are \mpp.
\end{proof}

The following is one way to verify that a graphon is simply coloured.
We recall the following definition, see e.g.\ \cite[Section 13.1.1]{Lovasz}.
\begin{definition}
Let $W$ be an uncoloured graphon defined on $\oi$.  
We say that $x_1,x_2\in\oi$ are \emph{twins} (for $W$)
if $x_1\neq x_2$ and $W_{x_1}=W_{x_2}$ a.e.
We say that $W$ is \emph{twinfree} if there are no twins.
\end{definition}

\begin{lemma}
  If $W$ is twinfree, then any colouring $\Psi$ of $W$ is simple.
\end{lemma}
\begin{proof}
We may assume that $W:\oi^2\to\oi$ is Borel measurable.
Then  the map $\Phi:x\mapsto W_x$ is Borel measurable $\oi\to L^1\oi$, 
and $\Phi$ is injective by the assumption that there are no twins.
Since both $\oi$ and $L^1\oi$ are Polish spaces 
(i.e., complete separable metric spaces) and thus Borel spaces,
it follows \cite[Theorem 8.3.7 and Proposition 8.3.7]{Cohn}
that the image $B:=\Phi(\oi)\subseteq L^1\oi$ is a Borel set, and
that the inverse $\Phi\qw: B\to\oi$ is measurable.
Thus, for all $x\in\oi$,
\begin{align}
  \Psi(x) = \Psi\bigpar{\Phi\qw(W_x)}
=\Psi\circ\Phi\qw(W_x)
=:h(W_x),
\end{align}
where $h:=\Psi\circ\Phi\qw:B\to\gD_\kk$ is measurable
and can be extended (trivially) to a measurable map $h_1:L^1\oi\to\gD_\kk$.
The result follows from \refL{LSimple}.
\end{proof}

\subsection{A topology on coloured graphons}\label{SStop}
We have in \refSS{Slim} defined a topology on the space of coloured graph
limits, which we can identify with equivalence classes of coloured graphons.
In the classical theory of uncoloured graphons, it is well-known that it
is very useful to also consider another topology on graphons, 
defined by the cut norm \eqref{cutnorm1}.
This distinguishes equivalent graphons, but the topology on graph limits can
be obtained as a quotient topology, see \eqref{dcut}.
In this subsection we consider an extension of the cut norm topology to the
coloured case. 

We consider coloured  graphons on a fixed probability space $\OFmu$, and 
let $\cWW=\cWW_\kk\OFmu$ be the set of all coloured graphons on $\OFmu$ with a
given number $\kk$ of colours.

Let $\PPsi=\PPsi_\kk$ be the set of measurable maps 
$\Psi=(\psi_\ell)_1^\kk:\gO\to\gD_\kk
$, i.e., the possible colourings of a
graphon defined on $\gO$.
We regard $\PPsi$ as a subset of $(L^\infty\Omu)^\kk$. 
Recall that $L^\infty\Omu$ is the dual of the Banach space $L^1\Omu$.
We equip $L^\infty\Omu$ with the corresponding \wxt{}, i.e., the topology
generated by the maps $f\mapsto \int fg\dd\mu$ for $g\in L^1\Omu$.

Recall that the closed unit ball of a dual Banach space is compact in the \wxt.
(The Banach-Alaoglu theorem \cite[Theorem 3.15]{Rudin}.)
Furthermore, it is easily seen that $\PPsi$ is a \wx{} closed subset of
$B^\kk$, where $B$ is the closed unit ball of $L^\infty\Omu$.
Hence, $\PPsi$ is a compact Hausdorff space.
Moreover, if $L^1\Omu$ is separable, 
for example if $\gO=\oi$,
then $B$ and, as a consequence, $\PPsi$ are compact metric spaces
\cite[Theorem 3.15]{Rudin}.

The set $\cWW=\cWW_\kk$ of all $\kk$-coloured graphons on $\OFmu$ is
by definition 
$\cWo\times\gL_\kk$, where $\cWo$ is the set of all
uncoloured graphons on $\OFmu$. We give $\cWW=\cWo\times\gL$ the product
topology, where 
$\cWo$ has the topology given by the cut norm in \eqref{cutnorm1}
and $\gL$ has the \wxt{} just defined. 

\begin{lemma}\label{LA1}
  Let\/ $W_1$ and $W_2$ be two uncoloured graphons, and let $\FF\in\cL$. 
Then, for any $\Psi\in\PPsi_\kk$,
  \begin{align}\label{la1}
\bigabs{t(\FF,(W_1,\Psi))-t(\FF,(W_2,\Psi))}
    \le
e(\FF)\cnorm{W_1-W_2}.
  \end{align}
\end{lemma}
\begin{proof}
  Recall the definition \eqref{pa1a}--\eqref{pa1b}. 
Let $\nu:=|\FF|$, $m:=e(\FF)$, and 
number the edges of $\FF$ as $i_1j_1,\dots,i_mj_m$; let $\ell_i$ be the
colour of vertex $i$ in $\FF$. Then, with $\Psi=(\psi_\ell)_1^\kk$, for $k=1,2$,
\begin{align}\label{ao3}
  t\bigpar{\FF,(W_k,\Psi)}&
=\int_{\gO^\nu} \prodinu \psi_{\ell_i}(x_i) \prod_{s=1}^m W_k(x_{i_s},x_{j_s})
\dd \mu(x_1)\dotsm \dd \mu(x_\nu)
.\end{align}
We define for $0\le r\le m$
\begin{align}\label{ao4}
t_r:=
\int_{\gO^\nu} \prodinu \psi_{\ell_i}(x_i) \prod_{s=1}^r
  W_1(x_{i_s},x_{j_s})
\prod_{s=r+1}^m W_2(x_{i_s},x_{j_s})
\dd \mu(x_1)\dotsm \dd \mu(x_\nu)
.\end{align}
Then, with a telescoping sum,
\begin{align}\label{ao5}
  t\bigpar{\FF,(W_1,\Psi)} - t\bigpar{\FF,(W_2,\Psi)}
=t_m-t_0
=\sum_{r=1}^m (t_r-t_{r-1}).
\end{align}
Fix  $r\in[m]$. Comparing \eqref{ao4} for $r$ and $r-1$ we see that in the
integrand, only the factor for $s=r$ differs. Furthermore, no other factor
depends on both $x_{i_r}$ and $x_{j_r}$. Hence, if $\bx'$ denotes
$(x_i)_{i\neq i_r,j_r}$, then we can combine all factors containing
$x_{i_r}$ as $g(x_{i_r},\bx')$ and all others as $h(x_{j_r},\bx')$; 
thus, with $x=x_{i_r}$ and $y=x_{i_s}$,
\begin{align}\label{ao6}
  t_r-t_{r-1}=\int_{\gO^\nu} \bigpar{W_1(x,y)-W_2(x,y)}g(x,\bx')h(y,\bx')
\dd \mu(x)\dd \mu(y)\dd\mu^{\nu-2}(\bx').
\end{align}
Note that $0\le g\le 1$ and $0\le h\le 1$. Hence, for every $\bx'$,
by \eqref{cutnorm2},
\begin{align}\label{ao7}
\Bigabs{\int_{\gO^2} \bigpar{W_1(x,y)-W_2(x,y)}g(x,\bx')h(y,\bx')
\dd \mu(x)\dd \mu(y)}
\le \cnorm{W_1-W_2},
\end{align}
and integrating over $\bx'$ yields 
$|t_r-t_{r-1}|\le\cnorm{W_1-W_2}$. Consequently, by \eqref{ao5},
\begin{align}\label{ao8}
\bigabs{  t\bigpar{\FF,(W_1,\Psi)} - t\bigpar{\FF,(W_2,\Psi)}}
\le\sum_{r=1}^m \bigabs{t_r-t_{r-1}}
\le m\cnorm{W_1-W_2}.
\end{align}
\end{proof}

The next lemma can be seen as a coloured version of the Counting lemma, for
the topology that we just have defined. See also
\refL{Lcounting} below for another version.

\begin{lemma}\label{LA2}
  For any fixed  $\FF\in\cU$,
the functional
$(W,\Psi)\mapsto t(\FF,(W,\Psi))$ is continuous on $\cWo\times\PPsi$
with the topologies given above (cut norm and \wx).
\end{lemma}
\begin{proof}
  First we show that for a fixed $W$, the map $\Psi\mapsto t(\FF,(W,\Psi))$
is continuous. By the definition \eqref{pa1a}--\eqref{pa1b}, 
this follows from the more general \refL{LA3} below.

Consequently, for every $(W,\Psi)\in\cWo\times\PPsi$ and every $\eps>0$,
there exists a neighbourhood $U$ of $\Psi$ such that $\Psi'\in U$ implies
\begin{align}\label{ar6}
  \bigabs{t(\FF,(W,\Psi'))-t(\FF,(W,\Psi))} <\eps/2.
\end{align}
Furthermore, \refL{LA1} shows that for every $W'\in\cWo$ with
$\cnorm{W'-W}<\eps/(2e(\FF))$, we have for all $\Psi'$
\begin{align}\label{ar7}
  \bigabs{t(\FF,(W',\Psi'))-t(\FF,(W,\Psi'))} <\eps/2.
\end{align}
Combining \eqref{ar6} and \eqref{ar7} yields 
$  \bigabs{t(\FF,(W',\Psi'))-t(\FF,(W,\Psi))} <\eps$ for $(W',\Psi')$ in a
suitable neighbourhood of $(W,\Psi)$, which completes the proof.
\end{proof}

The proof used the following simple general lemma.

\begin{lemma}\label{LA3}
Let $B$ denote the unit  ball of $L^\infty\Omu$,
equipped with the \wxt,
and   let $F(x_1,\dots,x_\nu)\in L^1(\gO^\nu,\mu^\nu)$.
Then, the map
  $B^\nu\to\bbR$ given by
  \begin{align}\label{ar3}
    (g_i)_1^\nu \mapsto 
\int_{\gO^\nu}\prodinu g_i(x_i) F(x_1,\dots,x_\nu)\dd\mu(x_1)\dotsm\dd\mu(x_\nu).
  \end{align}
is continuous.
\end{lemma}

\begin{proof}
This is similar to the proof of \refL{LTD}.
Denote the functional defined by \eqref{ar3} by $\Phi_F((g_i)_1^\nu)$.

First, if $F$ factorizes as
  \begin{align}\label{ar4}
    F(x_1,\dots,x_\nu)=f_1(x_1)\dotsm f_\nu(x_\nu)
  \end{align}
for some $f_1,\dots,f_\nu\in L^1\Omu$, then the \rhs{} of \eqref{ar3}
equals the product
$\prodinu \int_\gO g_i f_i\dd\mu$, where each factor is a continuous
functional of $(g_i)_1^\nu$ by the definition of the \wxt.
Hence, $\Phi_F$ is continuous and the  lemma holds for every such $F$.

Next, it follows by linearity
that the lemma holds for every $F$ that is a linear
combination of such products \eqref{ar4}. Let $L$ be the subspace of 
$L^1=L^1(\gO^\nu,\mu^\nu)$ consisting of such linear combinations.
If $F\in \bL$ (the closure of $L$ in $L^1$), then there is a sequence
$F_n\in L$ such that $\norm{F_n-F}_{L^1}\to0$. It follows from \eqref{ar3}
that, for all $(g_i)_1^\nu\in B^\nu$,
\begin{align}
  \bigabs{\Phi_{F_n}\bigpar{(g_i)_1^\nu}-\Phi_{F}\bigpar{(g_i)_1^\nu}}
\le \linorm{F_n-F}.
\end{align}
Hence, 
$\Phi_{F_n}\to\Phi_F$ uniformly on $B^\nu$, and since every $\Phi_{F_n}$ is
continuous, so is $\Phi_F$.

Finally, as said in the proof of \refL{LTD},
it is well-known that $\bL=L^1$, so the last
case covers all  $F\in L^1(\gO^\nu,\mu^\nu)$.
\end{proof}

This leads to our main result in this subsection.

\begin{theorem}\label{TQ1}
  Fix a probability space $\OFmu$ and give the space $\cWW=\cWW\OFmu$
of all $\kk$-coloured  graphons on $\OFmu$ the topology defined above.
(I.e., the product of the cut norm topology on $\cWo$ and 
the \wxt{} on $\gL_\kk\OFmu$.)
Give the space $\hcWW$ of coloured graph limits the topology defined in
\refSS{SSSlim}.
Then the mapping $\Pi:\WW\mapsto\hWW$ mapping a coloured
graphon to the corresponding coloured graph limit 
is  continuous.

Moreover, in the case $\OFmu=\oi$, 
$\Pi$
is a  quotient mapping.
In other words, 
for any topological space $Y$,
a mapping $f:\hcWW\to Y$ of coloured graph limits into $Y$ is continuous
if and only if the corresponding mapping $f\circ\Pi: \cWW\to Y$ of graphons is.
\end{theorem}

Hence, $\hcWW$ is a quotient space of $\cWW(\oi)$ also in the topological
sense.
(See \cite[Section 2.4]{Engelking} for more on quotient spaces and quotient
mappings.)

\begin{proof}
By the definition of the topology on $\hW_\kk$ in \refSS{SSSlim}, 
the continuity of $\Pi$ is equivalent to the continuity of the mapping
$\WW\mapsto t(\FF,\WW)$ for every coloured graph $\FF$,
which is exactly \refL{LA2}.

Now suppose that $\OFmu=\oi$.
Then $\Pi$ maps $\cWW$ onto $\hcWW$ by \refT{TWB}.
To show that $\Pi$ is a quotient mapping, it is 
sufficient to show that if $U\subseteq\hcWW$ is such that $\Pi\qw(U)$ is open in
$\cWW$, then $U$ is open in $\hcWW$.
(See \cite[Proposition 2.4.3]{Engelking}, and note that the converse
implication follows by continuity.)
Suppose, to obtain a contradiction, that this fails for some $U$.
Then $U$ is not open in $\hcWW$, and since $\hcWW$ is a metric space (\refL{LC}),
it follows that there exists some coloured graph limit $\hWW$ and a sequence of
coloured graph limits $\hWW_n$ 
with $\hWW_n\to\hWW$ in $\hcWW$ such that
\begin{align}\label{nt0}
  \hWW \in U \qquad\text{and}\qquad 
\hWW_n\notin U, \; n\ge1.
\end{align}
By \refT{TWB}, we can represent the graph limits $\hWW$ and $\hWW_n$ by 
coloured graphons $(W,\Psi)$ and $(W_n,\Psi_n)$. 
We regard $W$ and $W_n$ also as (uncoloured) 
graph limits, by standard abuse of notation.
Then,  the assumption $\hWW_n\to\hWW$ and \refT{Tuncolour} imply
$W_n\to W$ in $\hcWo$, which as is well-known (see
\refApp{app:uncoloured-pullback}) is equivalent to
\begin{align}\label{nt1}
  \dcut(W_n,W)\to 0.
\end{align}
By \eqref{dcut1}, this implies the existence of \mpp{} bijections
$\gf_n:\oi\to\oi$ such that
\begin{align}\label{nt2}
  \cnorm{W_n^{\gf_n}-W}\to 0.
\end{align}
We have $(W_n^{\gf_n},\Psi_n^{\gf_n})\cong  (W_n,\Psi_n)$,
and thus
the graph limit $\hWW_n$ is also represented by $(W_n^{\gf_n},\Psi_n^{\gf_n})$.
We noted above that (since now $\gO=\oi$) $\gL$ is a compact metric space.
Hence, by selecting a subsequence, we may assume that the sequence
$\Psi_n^{\gf_n}$ in $\gL$ converges, and thus there exists $\Psix\in\gL$
such that 
$\Psi_n^{\gf_n}\to\Psix$.
This together with \eqref{nt2} means that 
\begin{align}\label{nt4}
(W_n^{\gf_n},\Psi_n^{\gf_n})\to(W,\Psix)  
\qquad\text{in $\cWW=\cWo\times\gL$}.
\end{align}
Hence the continuity of $\Pi$ shows that
\begin{align}\label{nt3}
  \hWW_n=\Pi(W_n^{\gf_n},\Psi_n^{\gf_n})\to\Pi(W,\Psix).
\end{align}
Since we also have $\hWW_n\to\hWW$, this implies 
$\Pi(W,\Psix)=\hWW$, and thus, using \eqref{nt0}, $(W,\Psix)\in\Pi\qw(U)$.

On the other hand, 
\begin{align}\label{nt5}
 \Pi(W_n^{\gf_n},\Psi_n^{\gf_n})
=\Pi(W_n,\Psi_n)
=\hWW_n\notin U,
\end{align}
and thus $(W_n^{\gf_n},\Psi_n^{\gf_n})\notin\Pi\qw(U)$.
We have assumed that $\Pi\qw(U)$ is open, and thus its complement is closed.
Consequently, \eqref{nt5} and \eqref{nt4} imply $(W,\Psix)\notin \Pi\qw(U)$.

This contradiction shows that if $\Pi\qw(U)$ is open in $\cWW$, then $U$ is
open in $\hcWW$, which completes the proof.
\end{proof}

In the uncoloured case, it is well-known that the topology on graph limits
is given by the cut distance for graphons, which is defined 
as an infimum over the distance in cut norm 
between \mpp{} rearrangements of the graphons,
see \eqref{dcut} and \eqref{dcut1}.
This is extremely useful, so it is natural to try to extend this to coloured
graphons. It turns out that if we consider only pure coloured graphons, then
this is possible by a natural generalisation of the cut norm and cut distance
\cite{Diao,GrebikP,AHR26}; 
however, it is not possible to extend this to 
all mixed coloured graphons by defining a suitable invariant metric.
See the following \refSS{SSmetrics} for the technical details.

\subsection{Cut norm for coloured graphons}
\label{SSmetrics}

We consider in this subsection
for simplicity only graphons defined on $\oi$.
(Many results extend to general probability spaces,
but some require modifications; we leave this to the reader.)
We address the problem of extending the cut norm and cut distance
to coloured graphons.
There is a straightforward extension of the standard definition, 
see \eqref{ps1}--\eqref{pscut};
however, we will see below that it
works well and yield desired results for pure coloured graphons, but not
for general mixed coloured graphons. Furthermore, it seems that there is no
way to modify the definition and obtain a  definition for mixed
coloured graphons that satisfies all our wishes.

We consider the following extension of the cut norm and cut distance
to coloured graphons. 
For pure coloured graphons, this was introduced
 by \cite{Diao}
(in the form \eqref{ps5} below), and (independently) by 
\cite{GrebikP} and \cite{AHR26}.%
\footnote{In the current arXiv version of \cite{AHR26},
these are defined  for mixed colourings, but then their Lemma 2.1 is wrong. 
We have been informed by the authors that this will be corrected in a
forthcoming update, with the cut norm and distance considered only for
pure colourings.}
(The versions of cut norm
in \cite{GrebikP} and \cite{AHR26} are defined slightly differently, but
it is easy to see that they are equivalent to \eqref{ps1}
within constant factors; we omit the details.)

Let,
for coloured graphons $\WW_j=(W_j,\Psi_j)$, 
\begin{align}\label{ps1}
  \cnorm{{(W_1,\Psi_1)-(W_2,\Psi_2)}}
&:=\cnorm{W_1-W_2}+\linorm{\Psi_1-\Psi_2}
\end{align}
where as usual $\Psi_j=(\psi_{j,i})_{i=1}^\kk$ and we let
$  \linorm{\Psi_1-\Psi_2}
:=\sum_{i=1}^\kk\linorm{\psi_{1,i}-\psi_{2,i}}
$.
Then define,  
\begin{equation}\label{pscut}
 \dcut(\WW_1,\WW_2) = \inf_{\gf_1,\gf_2} \cnorm{\WW_1^{\gf_1} - \WW_2^{\gf_2}},
\end{equation}
taking the infimum over all \mpp{} bijections $\gf_1,\gf_2:\oi\to\oi$.

The cut norm distance \eqref{ps1} is a metric on the set of all coloured
graphons 
(with graphons that are a.e.\ equal identified as usual),
and it is invariant under \mpp{} bijections $\gf:\oi\to\oi$ in the sense
that
\begin{align}
\cnorm{{(W_1,\Psi_1)^\gf-(W_2,\Psi_2)^\gf}}
=  \cnorm{{(W_1,\Psi_1)-(W_2,\Psi_2)}}.  
\end{align}
Hence, \eqref{pscut} defines a pseudometric $\dcut$. 
(I.e., $\dcut$ is symmetric and satisfies the triangle inequality, but it
might be 0 also for distinct graphons.)
Furthermore,
as in the uncoloured case in \refApp{app:uncoloured-pullback}, 
there are several related versions of \eqref{pscut} that yield the same result.
For example, 
\cite[Theorem 6.9 and the preceding lemmas]{SJ249} 
hold also for coloured graphons (defined on general probability spaces) 
with the definition \eqref{ps1};
this follows easily by obvious modifications of the proofs in \cite{SJ249}.
(See also \cite[Lemma 3.5]{BCLSV1} and \cite[Theorem 8.13]{Lovasz}.)
For example, it follows that we can allow $\gf_1,\gf_2$ in \eqref{pscut}
to be any \mpp{} maps $\oi\to\oi$ (not necessarily bijective). 
In the opposite direction, we may assume that $\gf_2$ is the identity.

In particular, this shows that (as in the uncoloured case)
\begin{align}  \label{ps4}
\dcut(\WW,\WW^\gf)=0
\end{align}
for every $\WW$ and every \mpp{} $\gf:\oi\to\oi$. 

We note some simple facts.

\begin{lemma}\label{LPS1}
  
The cut norm in \eqref{ps1} defines a metric topology on the space
$\cWW=\cWo\times\gL_k$ of all coloured graphons on \oi.
This topology is stronger than the product topology in \refSS{SStop}.
Consequently, if\/ $\WW$ and $\WW_n$ are coloured graphons on \oi, then,
for any coloured graph $\FF$, as \ntoo,
\begin{align}\label{lco1}
  \cnorm{\WW_n-\WW}\to 0
\implies 
t(\FF,\WW_n)\to
t(\FF,\WW).
\end{align}
Equivalently, letting $\WW$ and $\WW_n$ denote also the corresponding graph
limits in $\hcWW$,
\begin{align}\label{lco2}
  \cnorm{\WW_n-\WW}\to 0
\implies 
\WW_n\to\WW
\quad\text{in }\hcWW.
\end{align}
\end{lemma}

\begin{proof}
    The cut norm \eqref{ps1} defines a product topology on
  $\cWW=\cWo\times\gL$,
with the usual uncoloured cut norm topology on $\cWo$ and the $L^1$ topology
on $\gL$. The $L^1$ topology is stronger than the \wx{} topology, so the
first claim follows.

Consequently, 
\eqref{lco2} follows by \refT{TQ1}, and \eqref{lco1} is equivalent to
\eqref{lco2} by \refC{CTWBb}.
\end{proof}

We can also give a quantitative version of \eqref{lco2}.
This is a coloured version of the counting lemma; 
the case of pure coloured
graphons is given by \cite{Diao} and \cite{AHR26}
(both with a minor error corrected here). (See \refL{LA2} for another version with a weaker topology.)

\begin{lemma}\label{Lcounting}
If $\WW_1$ and $\WW_2$ are coloured graphons on \oi, then,
for any coloured graph $\FF$,
  \begin{align}\label{lco3}
\bigabs{t(\FF,\WW_1)-t(\FF,\WW_2)}
    \le
(e(\FF)+v(\FF))\cnorm{\WW_1-\WW_2},
  \end{align}
and thus also
  \begin{align}\label{lco4}
\bigabs{t(\FF,\WW_1)-t(\FF,\WW_2)}
    \le
(e(\FF)+v(\FF))\dcut(\WW_1,\WW_2),
  \end{align}
\end{lemma}
\begin{proof}
If $\WW_j=(W_j,\Psi_j)$, then, using \refL{LA1} and \eqref{pa1b},
\begin{align}
&  \bigabs{t(\FF,\WW_1)-t(\FF,\WW_2)}
\notag\\&
\le
\bigabs{t(\FF,(W_1,\Psi_1))-t(\FF,(W_2,\Psi_1))}
+
\bigabs{t(\FF,(W_2,\Psi_1))-t(\FF,(W_2,\Psi_2))}
\notag\\&
\le e(\FF)\cnorm{W_1-W_2}
+ 
\int_{\oi^n} \biggabs{\prod_{i=1}^{v(\FF)} \psi_{1,\ell_i}(x_i) - 
\prod_{i=1}^{v(\FF)} \psi_{2,\ell_i}(x_i)}
\notag\\&
\le e(\FF)\cnorm{W_1-W_2}
+ 
v(\FF)\sumjkk \linorm{\psi_{1,j}-\psi_{2,j}},
\end{align}
which yields \eqref{lco3}. Then \eqref{lco4} follows from \eqref{lco3},
\eqref{pscut}, and the invariance \eqref{ar8} of $t(\FF,\WW)$ under pullbacks.
\end{proof}

\subsubsection{Pure colourings}\label{SSSmetrics-pure}

Consider now  the case of pure colourings.
(Recall that this is the only case studied in 
\cite{Diao,GrebikP}, and that~\cite{AHR26} 
(in a forthcoming corrected version, see above)
consider the cut distance only
for pure coloured graphons.)
Note that in  the pure coloured case, when $\Psi_j$ are given by colour maps 
$\chi_j:\gO\to[\kk]$, we can also write \eqref{ps1} as 
\begin{align}\label{ps5}
\cnorm{{(W_1,\Psi_1)-(W_2,\Psi_2)}}
&=\cnorm{W_1-W_2}+\sumikk\gl\bigpar{\chi_1\qw(i)\Delta\chi_2\qw(i)}
\notag\\&\hskip-2em
=\cnorm{W_1-W_2}+2\gl\set{x\in\oi:\chi_1(x)\neq\chi_2(x)}.
\end{align}

We have the following result, 
which shows that if we use only pure coloured graphons,
then the cut metric \eqref{pscut} yields (by the natural quotient space
construction) the compact topology defined earlier on the space $\cWW$ of
coloured graph limits. 
Various versions have earlier been proved in
\cite{Diao,GrebikP,AHR26}.

\begin{theorem}\label{TMpure}
  \begin{thmenumerate}
  \item\label{TMpure1}
On the set $\cWW\wpure$ of all pure coloured graphons on \oi,
the metric topology defined by the cut norm in \eqref{ps1} 
coincides with the restriction to $\cWW\wpure\subset\cWW=\cWo\times\gL_k$ of 
the product topology in \refSS{SStop}.

\item \label{TMpure2}
 Let $\WW_1$ and $\WW_2$ be two pure coloured graphons on $\oi$.
Then 
\begin{align}  \label{tmp2}
\WW_1\cong\WW_2 \iff \dcut(\WW_1,\WW_2)=0.
\end{align}

\item\label{TMpure3}
Let $\WW,\WW_1,\WW_2,\dots$ be pure coloured graphons on $\oi$, and
let $\WW$ and $\WW_n$ denote also the corresponding graph
limits in $\hcWW$.
Then, as \ntoo,
\begin{align}  \label{tmp3}
\WW_n\to\WW
\enspace\text{in }\hcWW
\iff \dcut(\WW_n,\WW)\to0.
\end{align}
  \end{thmenumerate}
\end{theorem}

\begin{proof}
\pfitemref{TMpure1}
Let $\WW=(W,\Psi)$ and $\WW_n=(W_n,\Psi_n)$ be pure coloured graphons.
Then, as \ntoo,
\begin{align}\label{tmp4}
  \cnorm{\WW_n-\WW}\to0
\iff
  \cnorm{W_n-W}\to0
\quad\text{and}\quad
\linorm{\psi_{n,i}-\psi_i}\to0\;\forall i\in[\kk].
\end{align}
The only difference from the  topology in \refSS{SStop} is that there
we require $\psi_{n,i}\to\psi_i$ in the \wx{} topology instead of in $L^1$.
However, for a pure coloured graphon $\Psi$, the functions $\psi_i$ are
indicator functions taking
only the values 0 and 1, and thus the result follows from \refL{Lweak*} below.

\pfitemref{TMpure2}
If $\WW_1\cong\WW_2$, then 
\refT{Tunique1} shows that 
$\WW_1^{\gf_1} = \WW_2^{\gf_2}$ for some \mpp{} mappings 
$\gf_1,\gf_2:\oi\to\oi$, and thus \eqref{ps4} implies
\begin{align}\label{tmp5}
  \dcut(\WW_1,\WW_2)
\le  \dcut(\WW_1,\WW_1^{\gf_1}) + \dcut(\WW_1^{\gf_1},\WW_2^{\gf_2}) 
+ \dcut(\WW_2^{\gf_2},\WW_2)
=0.
\end{align}

Conversely,
if $ \dcut(\WW_1,\WW_2)=0$, then \eqref{lco4} yields
$t(\FF,\WW_1)-t(\FF,\WW_2)=0$ for every coloured graph $\FF$, and thus
$\WW_1\cong\WW_2$ by \refD{D=}.

\pfitemref{TMpure3}
If $\dcut(\WW_n,\WW)\to0$, then \eqref{lco4} yields
$t(\FF,\WW_n)-t(\FF,\WW)\to0$ for every coloured graph $\FF$, and thus
$\WW_n\to\WW$ by \refC{CTWBb}. This proves
\begin{align}\label{tmp6}
\dcut(\WW_n,\WW)\to0\implies \WW_n\to\WW \enspace\text{in }\hcWW  
\end{align}

Let $\cWWX$ be the quotient space obtained from the space
of pure coloured graphons by identifying any  
pure coloured graphons $\WW_1$ and $\WW_2$ with
$\dcut(\WW_1,\WW_2)=0$. Thus $(\cWWX,\dcut)$ is a metric space. Note that
\cite[Theorem 3.7]{Diao} and \cite[Theorem 7.1]{GrebikP}
both show that this metric space is compact,
by an extension of the proof in the uncoloured case 
\cite[Theorem 9.23]{Lovasz}
(see also \cite[Lemma 8.2 and Remark 8.3]{ADW25}).
Furthermore, by \eqref{tmp2}, this quotient construction of $\cWWX$
means that we identify two graphons precisely when they are equivalent,
so by \refT{TWB}, we have a natural injective map 
$\iota:(\cWWX,\dcut)\to\hcWW$ into the set of coloured graph limits.
Moreover, by \refT{TW}\ref{TW1}, this map $\iota$ is onto,
and \eqref{tmp6} shows that it is continuous.
Hence, $\iota$ is a continuous bijection, and since $\cWWX$ is compact,
$\iota$ is a homeomorphism, which means \eqref{tmp3}.
\end{proof}

We used above the following simple lemma.

\begin{lemma}\label{Lweak*}
  Let $\psi=\sindic{A}$ and 
$\psi_n=\sindic{A_n}$ be indicator functions defined on some probability  space
$\OFmu$. Then, as \ntoo, $\psi_n\to \psi$ in the \wx{} topology on $L^\infty$
if and only if $\psi_n\to \psi$ in $L^1$.
\end{lemma}
\begin{proof}
Suppose that $\psi_n\to\psi$ \wx.
On $A$, we have $\psi=1\ge\psi_n$, and thus
\begin{align}
      \int_A|\psi_n-\psi|\dd\mu
= \int_A(\psi- \psi_n)\dd\mu
=\int_\gO \sindic{A}\cdot(\psi - \psi_n)\dd\mu
\to 0,
\end{align}
by the definition of the \wx{} toplogy.
Similarly, 
$\int_{A^{\textsf c}}|\psi_n-\psi|\to0$.
Thus $\psi_n\to\psi$ also in $L^1$. 

Conversely, suppose that $\psi_n\to\psi$ in $L^1$; then
$\mu(A_n\Delta A)=\linorm{\psi_n-\psi}\to0$. Hence, for any function $g\in L^1$,
\begin{align}
\Bigabs{  \int_\gO (\psi_n-\psi)g \dd\mu}
\le\int_\gO |\psi_n-\psi|\,|g| \dd\mu
=\int_{A_n\Delta A} |g|\dd\mu
\to0,
\end{align}
which shows that $\psi_n\to\psi$ pwx{} in $L^\infty$.
\end{proof}

\subsubsection{Mixed colourings}\label{SSSmetrics-mixed}

Unfortunately, for general (i.e.,  mixed) coloured graphons, there are
serious problems with $\dcut$. 
We show this by some examples, where we let $\kk=2$ and 
 so a colouring  $\Psi=(\psi_1,1-\psi_1)$ is determined by $\psi_1:\oi\to\oi$.
For simplicity we consider graphons $(W,\Psi)$ where
$W$ is the constant graphon $W:=1$.
(Taking $W$ to be another constant in $(0,1]$ would yield the same results
with trivial modifications.)
Then  the definition \eqref{pa1a}--\eqref{pa1b} yields, for any coloured graph
$\FF$ with vertices coloured $\ell_1,\dots,\ell_\nu$,
\begin{align}\label{tmm1}
  t\bigpar{\FF,(W,\Psi)}
=\int_{\oi^\nu} \prodinu \psi_{\ell_i}(x_i) \dd x_1 \dotsm \dd x_n
=\prodinu a_{\ell_i},
\end{align}
where $a_\ell:=\intoi\psi_\ell(x)\dd x$. Hence $a_2=1-a_1$, and 
$t\bigpar{\FF,(W,\Psi)}$ depends only on $a_1=\intoi\psi_1$.

\begin{example}\label{EB1}
Let
$\WW:=(W,\Psi)$ and $\WW_1=(W,\Psi_1)$, where 
$W=1$, and  
$\Psi$ and $\Psi_1$ are
determined by $\psi_1:=\frac12$ and 
$\psi_{1,1}:=\indic{x>\frac12}$.
Then $\intoi\psi_1=\intoi\psi_{1,1}$,
and consequently, by \eqref{tmm1},
$t\bigpar{\FF,\WW}=t\bigpar{\FF,\WW_1}$ for every
coloured graph $\FF$. 
Hence, $\WW\cong\WW_1$ by \refD{D=}.
However, $\linorm{\psi_{1}^{\gf_1}-\psi_{1,1}^{\gf_2}}=\frac12$ for any
\mpp{} mappings $\gf_1$ and $\gf_2$; hence \eqref{ps1}--\eqref{pscut} yield
$\dcut(\WW,\WW_1)\ge\frac12\neq0$. 
This shows that equivalence in \refT{TMpure}\ref{TMpure2} does not hold for
mixed coloured graphons.
\end{example}

\begin{example}\label{EB2}
Also the equivalence in \refT{TMpure}\ref{TMpure3} does not hold for
mixed coloured graphons.
One example is to take $\WW$ and $\WW_1$ as in \refE{EB1} and
simply let $\WW_n:=\WW_1$ for all $n$; then $\WW_n\cong\WW$ 
and thus $\WW_n=\WW$ in  $\hcWW$ (i.e. regarded as graph limits),
and therefore trivially $\WW_n\to\WW$,
but $\dcut(\WW_n,\WW)\ge\frac12$ for every $n$.
\end{example}

\begin{example}\label{EB3}
Another, perhaps more instructive, example showing the same thing is
obtained by taking $\WW$ and $\WW_1$ as in \refE{EB1}, and further define
$\WW_n:=(W,(\psi_{n,1},1-\psi_{n,1}))$ with 
$\psi_{n,1}(x):=\psi_1(\frax{2^{n-1}x})$
(where $\frax\cdot$ denotes the fractional part).
Then $\WW_n=\WW_1^{\gf_{n}}$ where $\gf_n$ is the \mpp{} mapping
$x\mapsto\frax{2^{n-1}x}$, and hence $\WW_n\cong\WW_1\cong\WW$,
so again $\WW_n\to\WW$ in $\hcWW$ (trivially), but 
$\dcut(\WW_n,\WW)=\dcut(\WW_1,\WW)\ge\frac12$.

Note that in this example, $\psi_{n,1}\to\psi_1$ in the \wx{} topology on
$L^\infty$, and thus $\WW_n\to\WW$ in the product
topology defined in \refSS{SStop}.
Since we have shown that $\cnorm{\WW_n-\WW}\ge\frac12$,
this shows that the topology defined by the cut metric on the set $\cWW$ of
all coloured graphons is strictly stronger
than the topology in \refSS{SStop}. In other words, 
also the equivalence in \refT{TMpure}\ref{TMpure1} does not hold
for mixed coloured graphons.
\end{example}

\begin{remark}
One might hope that it would be possible to modify the definition
\eqref{ps1} of cutnorm by replacing the second term by 
some metric $d(\Psi_1,\Psi_2)$ on $\gL$ that defines the \wx{} topology.
(Recall that $\gL$ is a metric compact space since we assume $\gO=\oi$.)
However, this seems impossible, since it seems that we need the metric to be
invariant under \mpp{} bijections in order to show that \eqref{pscut}
defines a metric, and no such metric exists by the following example.
\end{remark}

\begin{example}
  Let again $\kk=2$ and let $\Psi$ and $\Psi_n$ be as in \refE{EB3}.
Then, as noted there, $\Psi_n\to\Psi$ in $\gL$.
Let $d$ be a metric on $\gL$ defining the \wx{} topology.
Then thus $d(\Psi_n,\Psi)\to0$.

Suppose now that $d$ is invariant under \mpp{} bijections $\oi\to\oi$.
It is easy to construct such bijections $\gf_n$ such that
$\Psi_n=\Psi_1^{\gf_n}$. (The mappings $\gf_n$ in \refE{EB3} are not
bijections, but it is easy to construct $\gf_n$ as permutations of intervals
of length $2^{-n}$.)
Then, since $\Psi^{\gf_n}=\Psi$, we have 
\begin{align}
d(\Psi_n,\Psi)=d(\Psi_1^{\gf_n},\Psi^{\gf_n})=d(\Psi_1,\Psi)>0,  
\end{align}
(since $\Psi_1\neq\Psi$). This is a contradiction, and shows that no such
metric $d$ exists.
\end{example}

\section{Graph Partition Parameters via Graphon Colourings}\label{sec:graphon-partition}

Many graph parameters are obtained by optimising an objective over vertex partitions. The graphon-colouring framework gives a continuum version of this setup: the graph is represented by a graphon $W$, while the partition is represented by either a pure or mixed colouring of the underlying probability space. Thus finite optimisation over vertex partitions becomes optimisation over pure or mixed colourings in the graphon case. 

We first illustrate this principle with maximum cut in Section~\ref{SSSmax}. Then in Section~\ref{sec:cheeger-constant} we turn to the Cheeger constant, which presents a difficulty because it is a ratio of the crossing-edge mass to the smaller degree volume, whose denominator may degenerate under cut convergence. We see that bounding the minimum degree gives
the control needed to address this issue. Next, minimum $k$-section illustrates how balance constraints enter the colouring framework. The last graph partition problem we consider is modularity, whose objective subtracts a within-class degree-product contribution from the
within-class edge mass and allows the number of parts to vary. The relation between the modularity of graphs and graphons is more delicate, and we defer it to Section~\ref{sec3:modularity}.

We note that a family of paritition problems, which includes maximum cut and minimum $k$-section, was 
studied in \cite{BCLSV2}.  We include maximum cut and minimum $k$-section mostly to illustrate how our coloured graphon framework applies; we obtain shorter proofs and some new results -- see Remark~\ref{rem:conn-bclsv}.

\smallskip
\paragraph{Notation and setup.}
Note that in this section, 
graphons are generally uncoloured, unless they are explicitly said to be coloured.

For an uncoloured graphon $W$ on a probability space
$(\gO,\cF,\mu)$, 
recall that the degree function is defined as
\begin{align}\label{m1}
  k(x)=k_W(x):=\int_{\gO}W(x,y)\dd\mu(y),
\qquad x\in\gO,
\end{align}
and the edge density is
\begin{align}\label{m2}
\rho=\rho_W:=\iint_{\gO^2}W(x,y)\dd\mu(x)\dd\mu(y) =\int_\gO k(x)\dd\mu(x).
\end{align}
It is well known 
that $\rho_{W_1}=\rho_{W_2}$ when $W_1$ and $W_2$ are equivalent.
(This follows for example from \refT{Tunique0}, or from $\rho=t(K_2,W)$.) 
Furthermore, if $W_1$ and $W_2$ are any uncoloured graphons defined on the
same probability space $\gO$, then taking $S=T=\gO$ in \eqref{cutnorm1} yields
\begin{align}\label{m30}
  \bigabs{\rho_{W_1}-\rho_{W_2}}\le \cnorm{W_1-W_2}.
\end{align}

\subsection{Maximum cut}\label{SSSmax}
We begin by considering the maximum cut in graphs.
It is straightforward  to 
define the corresponding graphon version, and to formulate this in our
coloured graphon setup.
We will study this simple example in some detail, partly as an
introductory example, and partly to illustrate in a simple case
some features that will reappear in later applications, 
in particular the relation
between pure and mixed colourings.

For a graph $G$ on vertex set $[n]$, define the normalised maximum cut value by
\begin{align}\label{MC-G}
\MC(G)=\frac{1}{n^2}\max_{S \subseteq V(G)}e_G(S,S\comp),  
\end{align}
where $e_G(S,S\comp)$ is the number of unordered edges crossing the cut. 
We define
the graphon analogue by, 
for an uncoloured graphon $W$ defined on a general probability space $\OFmu$,
\begin{equation}\label{eq:graphon-MC}
\MC_\pure(W)=\sup_{S\subseteq \gO}\iint_{S \times S\comp} W(x,y) 
\dd\mu(x) \dd\mu(y),
\end{equation}
where the supremum is over all measurable sets $S \subseteq \gO$. 
We can express these as maxima over pure 2-colourings of $G$ or $W$.
A pure 2-colouring of $W$ is a measurable map $\chi:\gO\to\{1,2\}$ 
that
assigns each point of the graphon domain $\gO$
one of the two colours 1 and 2. We can identify this map with the
cut defined by the colour classes
$S=\chi^{-1}(1)$ and $S\comp=\chi^{-1}(2)$.
(Thus the indicator function $\sindic{S}=2-\chi$. 
We could use indicator functions and the colours $1$ and $0$,  
but we prefer to keep the colours $[\kk]$ as in the rest of the paper.) 
We can do the same with a colouring $\chi:V(G)\to\set{1,2}$ of $G$.
Hence, we have: 
\begin{align}\label{MC1pG}
  \MC(G)&:=\max_{\chi:V(G)\to\set{1,2}}\C_2(G,\chi),
\\\label{MC1pW}
  \MC_\pure(W)&:=\sup_{\chi:\gO\to\set{1,2}}\C_2(W,\chi),
\end{align}
where the cut values for pure $k$-colourings are defined by  (for $k=2$ we usually drop the subscript)
\begin{align}\label{MC2G}
  \C_k(G,\chi)&
  :=\frac{1}{n^2} \sum_{uv \in E(G)} \indic{\chi(u)\neq \chi(v)},
\\\label{MC2W}
\C_k(W,\chi)&:=
\frac{1}{2}\iint_{\gO^2} W(x,y)\indic{\chi(x)\neq \chi(y)}\dd\mu(x)\dd\mu(y) 
.\end{align}
(The factor $1/2$ appears because the double integral counts both ordered
regions $S\times S\comp$ and $S\comp\times S$, while the cut value counts the edges and edge
density across the bipartition only once.)

We have used pure colourings above. We may also
use mixed colourings of graphons, and the corresponding fractional cuts of a graph $G$.
For a mixed $\kappa$-colouring $\Psi=(\psi_1, \ldots, \psi_\kappa)$ of a graph or graphon,
define the corresponding cut values by
\begin{align}\label{MCGx}
\C_{\kappa}(G,\Psi)&:= 
\frac{1}{n^2} \sum_{1\leq i<j\leq \kappa} \sum_{uv\in E(G)} \psi_i(u)\psi_j(v).\\
\C_\kappa(W,\Psi)&:=  \sum_{1\leq i<j\leq \kappa}
\iint_{\gO^2} W(x,y) \psi_i(x)\psi_j(y)\dd\mu(x)\dd\mu(y).\label{MCWx2}
\end{align}
For the maximum cut and Cheeger constant we are interested only in $\kappa=2$.
Here and below, we may identify
a mixed 2-colouring $\Psi=(\psi,1-\psi)$ with its first component $\psi$,
since every 2-colouring is of this type.  Hence, for the case of two colours the definitions above simplify to in the graph case,
\begin{align}\label{MCG2mx}
\C(G,\Psi)= \C(G,\psi):= 
\frac{1}{n^2} \sum_{u,v\in V(G)}\sindic{ uv\in E(G)} \psi(u)(1-\psi(v)).
\end{align}
Similarly, for a graphon $W$ 
and a mixed 2-colouring
$\Psi=(\psi,1-\psi)$ of it,
\begin{align}\label{MCW2x}
\C(W,\Psi)= \C(W, \psi) :=   
\iint_{\gO^2} W(x,y) \psi(x)(1-\psi(y))\dd\mu(x)\dd\mu(y).
\end{align}
(For $\kappa=2$ we drop the subscript.) $\C(G,\psi)$ and $\C(W,\psi)$
are thus functionals of mixed 2-coloured graphs and graphons. 
Note that for pure cuts or colourings, \eqref{MCG2mx} and \eqref{MCW2x}
reduce to \eqref{MC2G} and \eqref{MC2W}, 
using our standard identification of
a pure colouring and the corresponding mixed colouring.
We take the supremum over all mixed colourings and define
\begin{align}\label{MCGmx}
\MC_\fra(G)&:=\sup_{\psi:V(G)\to\oi}\C(G,\psi),
\\\label{MCWmx}
   \MC_\mix(W)&:=\sup_{\psi:\gO\to\oi}\C(W,\psi)
.\end{align}
Since every pure colouring can be seen as a mixed one, we trivially have
\begin{align}\label{MCG+}
  \MC_\fra(G)&\ge \MC(G),
\\\label{MCW+}
  \MC_\mix(W)&\ge \MC_\pure(W)
.\end{align}
We can also express these functionals using coloured homomorphism densities.
Let $\KK_{2;\ell\glm}$ be the coloured graph with 
one edge and two vertices coloured $\ell$ and $\glm$.
 (We will consider general $\kappa$ in Section~\ref{subsec.min_k}, so it useful to state the lemma for general $\kappa$.)
\begin{lemma}\label{LMC3}
For every $\kappa$-coloured graphon $(W,\Psi)$, we have
\begin{align}\label{MC3}
  \C_\kappa(W,\Psi)=
\sum_{1\leq i < j \leq \kappa} t\bigpar{\KK_{2;ij},(W,\Psi)}.
\end{align}
\end{lemma}
\begin{proof}
This follows directly by comparing the definitions
\eqref{pa1a}--\eqref{pa1b} and \eqref{MCW2x}.
\end{proof}
Hence, $\MC_\pure(W)$ and $\MC_\mix(W)$ are obtained 
as $\sup t\bigpar{\KK_{2;12},(W,\Psi)}$
over all pure or mixed colourings of $W$, respectively. 

We next formalise the  correspondence between the finite maximum-cut problem
and its graphon formulation. We first show that 
the various definitions above yield the same result
for a graph and the corresponding graphons, and also that
the pure and mixed versions yield the same maximum.
Recall that we in \refE{EG1} have defined two (equivalent)
graphons $W'_G$ and $W_G$
corresponding to a graph $G$, one on the atomic space $V(G)$ and one on
the continuous space $\oi$.

\begin{theorem}\label{TMC1}
  Let $G$ be a graph and let $W'_G$ and $W_G$ be the corresponding atomic and continuous
graphons defined in \refE{EG1}.
Then
\begin{align}\label{tmc1}
  \MC(G)=\MC_\fra(G)
=
  \MC_\pure(W'_G)=\MC_\mix(W'_G)
=
  \MC_\pure(W_G)=\MC_\mix(W_G).
\end{align}
\end{theorem}

\begin{proof}
  We note first that 
  \begin{align}\label{tmc12}
\MC(G)=\MC_\pure(W'_G) 
\qquad\text{and}\qquad
\MC_\fra(G)=\MC_\mix(W'_G),    
  \end{align}
since in these cases, the definitions in \eqref{MC2G}--\eqref{MC2W} 
and \eqref{MCG2mx}--\eqref{MCW2x} 
for $G$ and $W'_G$ 
are the same up to notational differences.

Next, if $G$ is a graph with vertex set $[n]$
and a mixed 2-colouring $\Psi=(\psi_1,\psi_2)$, 
then
we define a random pure 2-colouring $\CHI:[n]\to[2]$ of $G$ by assigning each
vertex $i$ colour $j$ with probability $\psi_j(i)$, independently for
different vertices. Then, by \eqref{MC2G} and \eqref{MCG2mx},
the corresponding cut value $\C(G,\CHI)$ has the expectation
\begin{align}\label{tmc10}
  \E \bigsqpar{\C(G,\CHI)}
&= 
\frac{1}{n^2}
\E\sum_{u,v=1}^n\sindic{uv\in E(G)}\indic{\chi(u)=1 \text{ and }\chi(v)=2}
\notag\\&= 
\frac{1}{n^2}\sum_{u,v=1}^n\sindic{uv\in E(G)}
\P\bigpar{\chi(u)=1 \text{ and }\chi(v)=2}
\notag\\&= 
\frac{1}{n^2}\sum_{u,v=1}^n\sindic{uv\in E(G)}\psi_1(u)\bigpar{1-\psi_1(v)}
\notag\\&
=\C(G,\psi_1).
\end{align}
This implies that there exists a realization $\chi$ of $\CHI$ 
such that
$\C(G,\chi)\ge \C(G,\psi_1)$,
and thus $\MC(G)\ge \C(G,\psi_1)$.
Hence, by taking the supremum over all $\psi_1$, we have
$\MC(G)\ge \MC_\fra(G)$.
The opposite inequality is given by \eqref{MCG+}, and thus
$\MC(G)= \MC_\fra(G)$, which together with
\eqref{tmc12}  shows 
the equality of the
first four quantities in \eqref{tmc1}.

It remains to consider $W_G$.
We may show the remaining equalities directly by noting that,
with the intervals $I_v$ as in \refE{EG1}, $W_G$ is constant on each
$I_i\times I_j$, and thus $\C(W,\psi)$ in \eqref{MCW2x} remains the
same if we replace $\psi$  by its average $\bpsi$ on each $I_i$; this
average $\bpsi$ can then be regarded as a mixed colouring of $[n]$ and thus
of $W'_G$. We omit the details, 
and use instead the more general \refL{LMC} below which is based on 
a more general version of this argument.

We thus apply \refL{LMC} below to $W'_G$ and $W_G$, recalling 
from \refE{EG1} that $W_G$ is a pullback of $W'_G$.
This
yields $\MC_\pure(W'_G) \le \MC_\pure(W_G)$ and
$\MC_\mix(W'_G) = \MC_\mix(W_G)$. 
Using also \eqref{MCW+}, we thus have
\begin{align}\label{tmc5}
  \MC_\pure(W'_G) 
\le \MC_\pure(W_G)
\le \MC_\mix(W_G)
= \MC_\mix(W'_G).
\end{align}
We have already shown that the left-hand side of \eqref{tmc5} equals the
right-hand side. Hence all quantities in \eqref{tmc5} are equal, which
completes the proof.
\end{proof}

We next give the lemma used above, in a general form that is used later.
Recall the definition of pullback from \refApp{app:uncoloured-pullback}.

\begin{lemma}\label{LMC}
Let $W$ be a graphon defined on a probability space $\gO_2$, and let
 $W^\gf$ be the pullback of $W$ by
a \mpp{} map $\gf:(\gO_1,\cF_1,\mu_1)\to(\gO_2,\cF_2,\mu_2)$.
Then
\begin{align}\label{lmc1pure}
  \MC_\pure(W) &\le \MC_\pure(W^\gf),
\\\label{lmc1mix}
  \MC_\mix(W) &= \MC_\mix(W^\gf),
\end{align}
where in general there might be strict inequality in \eqref{lmc1pure}.
\end{lemma}
\begin{proof}
  First, if $\Psi=(\psi,1-\psi)$ is a pure or mixed 2-colouring of $W$, then
  $\Psi^\gf$ is  
a  2-colouring of $W^\gf$, and a change of variables yields
\begin{align}\label{lmc2}
  \C(W^\gf,\psi^\gf)
&=\iint_{\gO_1^2}W(\gf(x),\gf(y))\psi(\gf(x))\bigpar{1-\psi(\gf(y))}
\dd\mu_1(x)\dd\mu_1(y)
\notag\\&
=\iint_{\gO_2^2}W(x',y')\psi(x')\bigpar{1-\psi(y')}
\dd\mu_2(x')\dd\mu_2(y')
\notag\\&
=\C(W,\psi).
\end{align}
Taking the supremum over all mixed 2-colourings of $W$ yields
$\MC_\mix(W^\gf)\ge \MC_\mix(W)$, and taking the supremum over pure
2-colourings only yields (since then $W^\gf$ too is pure)
$\MC_\pure(W^\gf)\ge \MC_\pure(W)$, which is \eqref{lmc1pure}.

To complete the proof of \eqref{lmc1mix}, we will show that the supremum in
the definition of $\MC_\mix(W^\gf)$ is attained by a pullback colouring.
In fact, let $\psi:\gO_2\to\oi$ be a 2-colouring of $W^\gf$.
Let $\cF_\gf$ be the \gsf{} on $\gO_1$ generated by $\gf:\gO_1\to\gO_2$,
i.e., $\cF_\gf=\set{\gf\qw(B):B\in\cF_2}$. We have $\cF_\gf\subseteq\cF_1$
since $\gf$ is measurable.
Let $\bpsi:\gO_1\to\oi$ be the conditional expectation
$\bpsi:=\E(\psi|\cF_\gf)$.
Then, since $W^\gf(x,y)$ is $\cF_\gf$-measurable in both variables,
we obtain, using Fubini's theorem and integrating one variable at a time, in
both orders,
\begin{align}\label{lmc3}
  \C(W^\gf,\psi)
&
=\iint_{\gO^2} W^\gf(x,y) \psi(x)(1-\psi(y))\dd\mu_1(x)\dd\mu_1(y)
\notag\\&
=\iint_{\gO^2} W^\gf(x,y) \bpsi(x)(1-\psi(y))\dd\mu_1(x)\dd\mu_1(y)
\notag\\&
=\iint_{\gO^2} W^\gf(x,y) \bpsi(x)(1-\bpsi(y))\dd\mu_1(x)\dd\mu_1(y)
\notag\\&
=\C(W^\gf,\bpsi).
\end{align}

Since $\bpsi$ is $\cF_\gf$-measurable, there exists a measurable
function $\psi_2:\gO_2\to\oi$ such that $\bpsi=\psi_2\circ\gf=\psi_2^\gf$.
(See e.g.\ \cite[Lemma 1.13]{Kallenberg}.)
Hence, \eqref{lmc3} and \eqref{lmc2} yield
\begin{align}\label{lmc4}
  \C(W^\gf,\psi)  
=\C(W^\gf,\psi_2^\gf)
=\C(W,\psi_2)
\le\MC_\mix(W),
\end{align}
and thus $\MC_\mix(W^\gf) \le \MC_\mix(W)$,
which together with the opposite inequality shown above yields \eqref{lmc1mix}.

Finally, the \refE{EMCbad} below shows that there may be strict inequality
in \eqref{lmc1pure}.
\end{proof}

\begin{example}\label{EMCbad}
  Let $\gO:=\set{\bullet}$ be a singleton set (equipped with Dirac measure).
and let $W(\bullet,\bullet):=1$.
Then the definitions yield
$\MC_\pure(W):=\max_{\psi\in\set{0,1}}\psi(1-\psi)=0$
and
$\MC_\mix(W):=\max_{\psi\in\oi}\psi(1-\psi)=\frac14$.
Hence 
\begin{align}\label{emcbad}
\MC_\pure(W)<\MC_\mix(W).  
\end{align}
Furthermore, let $\gf$ be the unique map $\gf:\oi\to\bullet$; this map is
trivially \mpp, and $W^\gf$ is the graphon on $\oi$ that is constant 1.
Hence, for any mixed 2-colouring $\Psi=(\psi,1-\psi)$ of $\oi$, 
\eqref{MCW2x} yields
\begin{align}
  \C(W^\gf,\psi)=\intoi\intoi \psi(x)\bigpar{1-\psi(y)}\dd x\dd y
= \intoi \psi(x)\dd x \cdot\Bigpar{1-\intoi\psi(y)\dd y}.
\end{align}
This depends only on $\intoi\psi$, and the supremum over mixed 2-colourings
is obtained when $\intoi\psi=\frac12$, yielding $\MC_\mix(W^\gf)=\frac14$,
in accordance with \eqref{lmc1mix}. Moreover, this supremum is attained by
some pure colourings, for example by $\psi(x):=\sindic{x\in [0,\frac12]}$.
Hence, 
\begin{align}
  \MC_\pure(W^\gf) =   \MC_\mix(W^\gf) = \tfrac14
=\MC_\mix(W) > \MC_\pure(W)=0.
\end{align}
Cf.\ \refL{LMC}, and note that we here have strict inequality in
\eqref{lmc1pure}. 

Note that we did not see these bad behaviours in \refT{TMC1} because the
graph $G$ 
has no loops, and thus $W'_G(i,i)=0$ for all $i$.
\end{example}

\refE{EMCbad} shows also that two equivalent graphons may have different
$\MC_\pure$,
and thus $\MC_\pure$ is not well-defined for equivalence classes of
graphons, i.e., for graph limits. (Recall \refR{Rdef}.)
The next theorem shows that $\MC_\mix$ is better behaved.

\begin{theorem}
  \label{TMC2}
If $W_1$ and $W_2$ are two equivalent graphons (defined on arbitrary
probability spaces), then $\MC_\mix(W_1)=\MC_\mix(W_2)$.
\end{theorem}

\begin{proof}
  By \refT{Tunique0}\ref{Tunique0BCL} 
there exists a third graphon $W$ such that
$W_1=W^{\gf_1}$ and $W_2=W^{\gf_2}$ for some \mpp{} maps
$\gf_1$ and $\gf_2$, and thus
\refL{LMC} yields $\MC_\mix(W_1)=\MC_\mix(W)=\MC_\mix(W_2)$.
(Alternatively, we may avoid the deep \refT{Tunique0}\ref{Tunique0BCL} 
from \cite{BCL:unique}
and instead use the 
more elementary
\cite[Theorem 8.3]{SJ249} which provides a chain of 4 pullbacks between
$W_1$ and $W_2$.)
\end{proof}

We repeat that
\refE{EMCbad}  shows that \refT{TMC2} does not hold for the pure
version $\MC_\pure(W)$. 
We take that as an indication that the natural definition
for graphons is $\MC_\mix$ and not $\MC_\pure$, and we will in the sequel define,
for any graphon~$W$,
\begin{align}\label{MCW0}
  \MC(W):=\MC_\mix(W).
\end{align}
\refT{TMC2} thus shows that $\MC$ is well-defined for equivalence classes of
graphons, i.e., for graph limits.
We next show that $\MC$ is continuous, as a  functional of graphons or graph
limits.

\begin{lemma}\label{Lmaxcut} 
If\/ $W_1$ and $W_2$ are graphons on the same probability space, then 
\begin{align}\label{lmc31}
|\operatorname{MC}(W_1)-\operatorname{MC}(W_2)| \le \|W_1- W_2\|_\square.  
\end{align}
\end{lemma}
\begin{proof}
  Let $\Psi$ be any mixed 2-colouring of the probability space $\gO$.
Then \eqref{MC3} and \refL{LA1} yield
\begin{align}\label{lmc32}
\bigabs{\C(W_1,\Psi) - \C(W_2 ,\Psi)}
\le \cnorm{W_1-W_2}.
\end{align}
The result follows by taking the suprema over all colourings $\Psi$.
\end{proof}

\begin{theorem}\label{TMC3}
If\/ $W_1$ and $W_2$ are 
graphons, then 
\begin{align}\label{tmc31}
|\operatorname{MC}(W_1)-\operatorname{MC}(W_2)| \le \dcut(W_1,W_2).
\end{align}
Hence, if\/ $W_n\to W$ in cut distance, then $\MC(W_n)\to\MC(W)$.
In other words, $\MC=\MC_\mix$ is a continuous functional
on the set of graphons equipped with cut distance, 
and thus also on the quotient space $\hcWo$ of graph limits.
\end{theorem}

\begin{proof}
\refLs{LMC} and \ref{Lmaxcut} show that, for any pullbacks $W_1^{\gf_1}$
and $W_2^{\gf_2}$ of $W_1$ and $W_2$, 
\begin{align}
  \bigabs{\MC(W_1)-\MC(W_2)}
=  \bigabs{\MC(W_1^{\gf_1})-\MC(W_2^{\gf_2})}
\le \cnorm{W_1^{\gf_1}-W_2^{\gf_2}},
\end{align}
and thus \eqref{tmc31} follows from the definition \eqref{dcut} of the cut
distance.
\end{proof}

\begin{corollary}\label{CMC3}
  If\/ $G_n$ is a sequence of graphs such that $G_n\to W$ for some graphon
  $W$, then
  \begin{align}\label{cmc3}
    \MC(G_n) \to \MC(W).
  \end{align}
\end{corollary}

\begin{proof}
  We have $\MC(G_n)=\MC(W_{G_n}) $ by \refT{TMC1}, and the assumption $G_n\to W$
is equivalent to $W_{G_n}\to W$ in cut distance. Hence the result 
\eqref{cmc3} follows from \refT{TMC3}.
\end{proof}

Finally, we return to the relation between $\MC_\pure$ and $\MC_\mix$.
We saw in \refE{EMCbad} that these may differ for a graphon.
However, if we consider only graphons defined on $\oi$, there is no such
problem.
(This extends to graphons defined on any probability space with a non-atomic
measure, by the same proof with only notational changes.)

\begin{theorem}\label{TMC4}
  Let\/ $W$ be a graphon on $\oi$. Then
$\MC_\pure(W)=\MC_\mix(W)$.
\end{theorem}

\begin{proof}
Suppose first that $W$ is a step graphon in the sense that,
for some $N\ge1$, $W$ is 
constant on every square $I_{i}\times I_{j}$, $i,j\in[N]$,
where $I_{i}:=[(i-1)/N,i/N)$.
Given a mixed $2$-colouring $\Psi$ of $W$, we can find a pure 
$2$-colouring $\chi$
such that for every $i\in[N]$ and $j\in[2]$,
\begin{align}\label{tmc4}
  \gl\set{x\in I_{i}:\chi(x)=j} = \int_{I_{i}}\psi_j(x)\dd x.
\end{align}
(This can be achieved by partitioning each $I_i$ into subintervals of the
lengths \eqref{tmc4}.)
Since $W(x,y)$ is  constant on each square 
$I_{i}\times I_{j}$,
it follows from \eqref{tmc4} (by integrating over each such square separately)
that the  integral in \eqref{MCW2x} is the same for 
the colourings $\Psi$ and $\chi$, and thus
\begin{align}\label{mcx1b}
\C(W,\Psi)=\C(W,\chi).
\end{align}

Consider now a general graphon $W$ and let $\Psi$ be a
mixed $\kk$-colouring of it. 
Let $N\ge1$ and let $W_N$ be the step
graphon 
obtained by averaging $W$ over each square $I_{i}\times I_{j}$,
with $I_i$ as above.
 Let $\chi_N$ be the 
pure $\kk$-colouring constructed by \eqref{tmc4}, so that \eqref{mcx1b} says
\begin{align}\label{tmc40}
  \C(W_N,\Psi)= \C(W_N,\chi_N).
\end{align}
Consequently, using \eqref{lmc32} twice, we find
\begin{align}\label{tmc43}
\MC_\pure(W)&\ge
\C(W,\chi_N) \ge 
\C(W_N,\chi_N)-\cnorm{W_N-W}
\notag\\&
= \C(W_N,\Psi) -\cnorm{W_N-W}
\notag\\&
\ge \C(W,\Psi) -2\cnorm{W_N-W}.
\end{align}
Now let \Ntoo. It is well-known that $\cnorm{W_N-W}\le\linorm{W_N-W}\to0$,
and thus it follows from \eqref{tmc43} that $\MC_\pure(W)\ge \C(W,\Psi)$.
Since $\Psi$ is arbitrary, this implies
$\MC_\pure(W)\ge\MC_\mix(W)$.
As noted in \eqref{MCW+}, the converse inequality is trivial.
\end{proof}

\begin{theorem}
  \label{TMC5}
Let\/ $W$ be a graphon on a probability space $\OFmu$. 
Then the supremum in \eqref{MCWmx} is attained. In other words, there exists
a mixed colouring $\Psi$ of $W$ such that $\MC(W)=\C(W,\Psi)$.
\end{theorem}

\begin{proof}
  For a fixed graphon $W$ on a probability space $\OFmu$,  \eqref{MC3} and
  \refL{LA2} show that the mapping 
$\Psi\mapsto \C(W,\Psi)=t\bigpar{\KK_{2;12},(W,\Psi)}$ is continuous on the
set $\gL$ of all 2-colourings of $\gO$ equipped with the \wx{} topology as a
subset 
of $(L^\infty)^2$.
As noted in the beginning of \refSS{SStop}, $\gL$ is a compact space, and
thus the mapping attains its supremum $\MC(W)$.
\end{proof}

\begin{remark}
For pure colourings, we do not know whether the supremum in \eqref{MC1pW}
always is attained, i.e., 
we do not know whether there always
is a pure 2-colouring $\chi$ such that
$\MC_\pure(W)=\C(W,\chi)$. 
We leave this as an open problem.
In particular, this is an interesting problem in the case of graphons $W$ on
$\oi$, when we know that $\MC_\pure(W)=\MC_\mix(W)$ by \refT{TMC4}.

Note that if $W$ is a graphon on $\oi$, then \refT{TMC5} shows the
existence of a mixed colouring $\Psi$ such that 
$\MC(W)=\C(W,\Psi)$, and then \refL{Loi} shows the existence of a pure
coloured graphon $(W',\chi')$ on $\oi$ that is equivalent to $(W,\Psi)$;
hence \refL{LMC3} and \refT{TMC2} show that
$\C(W',\chi')=\C(W,\Psi)=\MC(W)=\MC(W')$.
However, this only shows that there is \emph{another} graphon $W'$
on $\oi$, equivalent to $W$, for which $\MC(W')=\MC(W)$ is attained by a
pure colouring, which does not answer our question for the original graphon
$W$. 
\end{remark}

\subsection{Cheeger constant}\label{sec:cheeger-constant}
The Cheeger constant of a graphon was introduced by~\cite{Khetan}, who established the corresponding Cheeger--Buser inequalities, investigated its relation to finite graphs, and showed that cut-norm convergence alone does not imply convergence of Cheeger constants. We formulate the parameter as an infimum over mixed $2$-colourings, expressed through coloured homomorphism densities, and prove that it is invariant under measure-preserving pullbacks. We then bound the difference between the Cheeger constants of a finite graph and its associated step graphon and prove cut-metric stability for graphons whose degrees are uniformly bounded away from zero. Consequently, we prove in \refT{thm:cheeger-conv} that the Cheeger constants converge along every cut-convergent sequence of dense graphs with uniformly linear minimum degree. We first recall the standard finite-graph definition \cite{alon1985lambda1}. 
\begin{definition}\label{def:cheeger}
Let $G$ be a finite simple graph with no isolated vertices. Its Cheeger constant is 
\begin{equation}\label{che-G}
h(G):= 
\min_{\varnothing\ne A\subsetneq V(G)}
\frac{e_G(A,A\comp)}
{\min\{\operatorname{vol}_G(A),\operatorname{vol}_G(A\comp)\}},
\end{equation}
where, as in \eqref{MC-G}, $e_G(A,A\comp)$  is the number of edges between $A$ $A\comp$, and the volume of a set $A$ is $\vol_G(A)=\sum_{v\in A}d_v = 2e(A)+e(A,A^c)$.
\end{definition}
For a graphon $W$ on a probability space $(\Omega,\mathcal F, \mu)$, recall
the following
\begin{align}\label{che2}
k_W(x):=\int_\Omega W(x,y) \dd \mu(y), \qquad \rho_W:=\int_\Omega k_W(x)\dd\mu(x).  
\end{align}
Let $\Lambda_2$ denote the set of 
mixed $2$-colourings of
$\Omega$. Every $\Psi \in \Lambda_2$ has the form 
$\Psi=(\psi_1,\psi_2)=(\psi_1,1-\psi_1)$ 
for a measurable function $\psi_1:\Omega \to [0,1]$. 

For $a,b \in [2]$, let $\KK_{2;ab}$  denote an edge whose endpoints have colours $a$ and $b$ and write
\begin{equation}\label{eq.tab_defn} 
t_{a,b}=t_{a,b}(W,\Psi):=t(\KK_{2;ab},(W,\Psi)).
\end{equation}
Thus,
\begin{align}\label{tab2}
t_{a,b}=\int_{\Omega^2}W(x,y)\psi_a(x)\psi_b(y) \dd\mu(x)\dd\mu(y).  
\end{align}
By symmetry of $W$ and the identity $\psi_1+\psi_2=1$, we have that 
\begin{align}
t_{1,2}=t_{2,1}, \qquad \sum_{a,b=1}^2 t_{a,b}=\rho_W, \qquad \sum_{b=1}^2
t_{a,b}=\int_\Omega k_W(x) \psi_a(x) \dd\mu(x).
\end{align}

In particular, the crossing-edge mass and the degree volumes of the two colour classes are determined by coloured homomorphism densities. For $\Psi=(\psi,1-\psi)$, recalling the definition from~\eqref{MCW2x} and noting~\eqref{tab2}, 
\begin{align}\label{chC}
\C(W,\psi)&=\iint_{\Omega^2} W(x,y)\psi(x)(1-\psi(y)) \dd\mu(x)\dd\mu(y)=t_{1,2}.
\end{align}
We also define
\begin{align}
m(W,\psi)\label{chm}
&:=
\int_\Omega k_W(x)\psi(x)\,\dd\mu(x)
=t_{1,1}+t_{1,2},\\\label{chD}
D(W,\psi)
&:=
\min\{m(W,\psi),\rho_W-m(W,\psi)\}=
\min\{t_{1,1}+t_{1,2},t_{2,2}+t_{1,2}\}.
\end{align}
Whenever $D(W,\psi)>0$, define the coloured Cheeger ratio by
\begin{equation}\label{cheh}
H(W,\Psi)
:=
\frac{\C(W,\psi)}{D(W,\psi)}
=
\frac{t_{1,2}}
{\min\{t_{1,1}+t_{1,2},t_{2,2}+t_{1,2}\}}.
\end{equation}
When $\rho_W>0$, we define the \emph{(mixed) Cheeger constant} 
of $W$ by
\begin{equation}\label{che3}
h(W)
:=
\inf_{\substack{
\Psi=(\psi,1-\psi)\in\Lambda_2\\
D(W,\psi)>0}}
H(W,\Psi).
\end{equation}
We similarly define the \emph{pure Cheeger constant} 
$h_\pure(W)$ by taking the infimum in \eqref{che3} over pure colourings
$\Psi$ only. This version is
the Cheeger constant for graphons as defined by \cite{Khetan} 
(who consider only graphons on $\oi$).
We will mainly used the mixed version; we may add the subscript ${}_\mix$ for
emphasis, but we usually omit it.
We will see in \refT{TCH4} that for graphons defined on $\oi$, the pure and
mixed Cheeger constants are the same; in particular, our definition
\eqref{che3} yields the same result as \cite{Khetan} for the graphons that
they consider.

Comparing the definitions above, immediately gives $h(G)=h_\pure(W'_G)$, where $W'_G$ is the corresponding atomic graphon; see \refE{EG1}.
Moreover, the \emph{fractional Cheeger constant} of a finite
graph $G$ was defined in \cite{Khetan} as follows
(using our notation).
We may assume $V(G)=[n]$; let $(A_{uv})$ be the adjacency matrix of $G$, recall that $d_v$ denotes the degree of vertex $v$ and write 
$d_{\rm max}:=\max_{v \in [n]}d_v$. Also define,
for $p=(p_1,\ldots,p_n)\in[0,1]^n$, 
\begin{align}\label{che5}
F(p):=\sum_{u,v=1}^n A_{uv}p_u(1-p_v),
\qquad
L(p):=\sum_{v=1}^n d_v p_v.\end{align}
Then, the fractional Cheeger constant is defined by \cite{Khetan} as 
\begin{equation}\label{eq:frac-cheeger-khetan}
\widetilde h(G):=
\inf_{\substack{p\in[0,1]^n\\0<L(p)<T}}
\frac{F(p)}{\min\{L(p),\vol(G)-L(p)\}}.
\end{equation}
It is easily seen that this equals $h_\mix(W'_G)$.
More generally,
we have the following relations between these notions for a graph and 
corresponding graphons. 
Note the similarity to \refT{TMC1}, but also the
difference, since the Cheeger constant $h(G)$ may differ from the fractional
version $\thx(G)$ for a finite graph. (A simple example is $G=K_3$: we have
$h(K_3)=1$ and $\thx(K_3)=1/2$.)

\begin{theorem}\label{TCH1}
  Let $G$ be a graph and let $W'_G$ and $W_G$ be the corresponding
atomic and continuous
graphons defined in \refE{EG1}.
Then
\begin{align}\label{tch1p}
  h(G)&=h_\pure(W'_G),
\\\label{tch1m}
\thx(G)
&=h_\mix(W'_G)
=h_\pure(W_G)=h_\mix(W_G).
\end{align}
Furthermore, $\thx(G)\le h(G)$, where we may have strict inequality.
\end{theorem}

\begin{remark}
The equality $\thx(G)=h_\pure(W_G)$ was noted by \cite[Section 4.1]{Khetan}.
Furthermore, \cite{Khetan} more generally consider edge-weighted graphs;
\refT{TCH1} extends to this setting, with the natural definitions of $W'_G$
and $W_G$.
\end{remark}

\begin{proof}
  This is similar to the proof of \refT{TMC1}.
As noted above, \eqref{tch1p}  and $\thx(G)=h_\mix(W'_G)$
follow immediately from the definitions.

It remains to consider $W_G$.
With the intervals $I_v$ as in \refE{EG1}, $W_G$ is constant on each
$I_i\times I_j$, and thus $\C(W,\psi)$ and $m(W,\psi)$, 
and therefore also $H(W,\Psi)$, 
remain the same if we replace $\psi$  by its average $\bpsi$ on each $I_i$; this
average $\bpsi$ can then be regarded as a mixed colouring of $[n]$ and thus
of $W'_G$; conversely any mixed colouring of $W'_G$  
can be regarded as a colouring of $W_G$ that is constant on each interval
$I_i$. It follows that $h(W'_G)=h(W_G)$.
(This also follows from the more general \refL{lem:cheeger-pullback} below.)
Similarly, given any mixed colouring $\Psi$ of $W_G$ we may replace it with
a pure colouring $\chi$ satisfying \eqref{tmc4}, and then 
$H(W_G,\Psi)=H(W_G,\chi)$;
it follows that $h_\pure(W_G)\le H(W_G,\Psi)$. Hence,
$h_\pure(W_G)\le h_\mix(W_G)$, and the opposite inequality is obvious.
(This also follows from the more general \refT{TCH4} below.)

Finally, it is obvious from the definition that $\thx(G)\le h(G)$, and we
may have strict inequality as remarked before the theorem.
\end{proof}

We note another simple result.
\begin{lemma}  \label{LCH0}
For every graphon $W$ with $\rho_W>0$,
\begin{align}
  0\le h(W)\le\tfrac12.
\end{align}
\end{lemma}
\begin{proof}
The lower bound is obvious, and the upper follows because the 
constant colouring $\psi(x)=1/2$ yields $\C(W,\psi)=\rho_W/4$,
$m(W,\psi)=\rho_W/2$, and thus $H(W,\psi)=1/2$.
\end{proof}
For graphons $W$ defined on $\oi$,
\cite[Lemma 3.2]{Khetan} shows $h_\pure(W)\le\frac12$
with a more complicated argument; this also follows from \refL{LCH0} and
\refT{TCH4} below.

By the definition \eqref{che3},
$h(W)$ is obtained by optimising a function of coloured
homomorphism densities over mixed $2$-colourings.
The continuity of coloured homomorphism densities implies
that $H(W,\Psi)$ is continuous 
(in both $W$ and $\Psi$)
wherever its denominator is
positive.  This does not immediately imply continuity of
$h(W)$, since the denominators associated with
near-minimizing colourings may approach zero. A uniform lower-bound on the degree function will provide the required control.

For a connected graph $G$ on $n$ vertices, \cite[Lemma~4.4]{Khetan} prove that
\begin{align}\label{kh1}
1 \ge \frac{h(W_G)}{h(G)} =
\frac{\widetilde h(G)}{h(G)}\ge
(1-\varepsilon)\left(1-\frac{2\gamma}{n\varepsilon^2}\right),
\qquad 0<\varepsilon<1,
\end{align}
where 
$\gamma=d_{\max}(G)/d_{\min}(G)$.
(Their Equation~(1.2) is the regular case, for which
$\gamma=1$.)
Their proof of \eqref{kh1} uses independent Bernoulli rounding
and Chebyshev's inequality to control deviations of the rounded
degree volume from half the total volume; for regular graphs,
they also give an improvement using Azuma's inequality
\cite[Remark~4.7]{Khetan}. We use the same rounding construction, but apply Cauchy--Schwarz directly to the mean absolute deviation. This avoids introducing a deviation threshold and yields the following estimates. In particular, for regular graphs the relative error is at most $n^{-1/2}$.

\begin{lemma}\label{lem:cheeger-graph-graphon}
Let $G$ be a finite simple graph with no isolated vertices.
Then,
\begin{align}\label{kh2}
  \frac{h(W_G)}{h(G)}
=  \frac{\widetilde h(G)}{h(G)}
\ge
1-\frac{\sqrt{\sum_{v}d_v^2}}{\sum_{v} d_v }.
\end{align}
and
\begin{equation}\label{chg}
0\le h(G)-h(W_G)
\le
\frac{\sqrt{\sum_v d_v^2}}{\vol(G)}
\le
\sqrt{\frac{ d_{\rm max}
}{
{ \vol(G)}}}.
\end{equation}
\end{lemma}

\begin{proof}
Recall that $h(W_G)=\thx(G)$ by \eqref{tch1m}.
If $\widetilde h(G)=h(G)$, the desired bound is immediate. Suppose therefore
that $\widetilde h(G)< h(G)$. Since $\widetilde h(G) \ge 0$, we have
$h(G)>0$. As $G$ has no isolated vertices, this implies that $G$ is
connected. By \cite[Lemmas~4.1 and 4.3]{Khetan}, there exists an optimiser
$p \in [0,1]^n$ such that, using the notation in \eqref{che5},
$L(p)= e(G)$ and $F(p)=\widetilde h(G) e(G).$ 

Independently sample $R_v\sim\operatorname{Bernoulli}(p_v)$ for $v=1,\ldots,n$,
and define
\begin{align}\label{chg1}
Z:=\sum_{v=1}^n d_v R_v,
\qquad
Y:=\sum_{u,v=1}^n A_{uv}R_u(1-R_v).
\end{align}
Thus, $Z$ is the degree volume of the selected vertices, and $Y$ is their crossing-edge count. Since $G$ is loopless, independence gives
\begin{equation}\label{chg2}
\E[Z]=L(p)
=e(G), \qquad \E[Y]=F(p)=\widetilde h(G)\, e(G)
.
\end{equation}
For every realization, the definition of $h(G)$ gives
\begin{align}\label{chg3}
Y\ge h(G)\min\{Z, \vol(G)-Z\}.
\end{align}
This also holds for the empty and full cuts, when both sides are zero.
Using $\min\{Z, \vol(G)-Z\}
=e(G)-\left|Z-e(G)\right|$
and taking expectations, we obtain
\begin{equation}
\widetilde h(G)\,e(G)
\ge
h(G)\bigl(e(G)-\mathbb E\left|Z-e(G)\right|
\bigr).
\label{eq:cheeger-rounding-inequality}
\end{equation}
By Cauchy--Schwarz and independence,
\begin{align}\label{chg5}
\mathbb E\left|Z-
 e(G)
\right|
&\le \sqrt{\operatorname{Var}(Z)}=\sqrt{\sum_{v=1}^n d_v^2p_v(1-p_v)}\le \frac12\sqrt{\sum_{v=1}^n d_v^2}.
\end{align}
Substituting into \eqref{eq:cheeger-rounding-inequality} yields
\begin{align}\label{chg6}
\widetilde h(G)
\ge
h(G)\left(
1-\frac{\sqrt{\sum_{v=1}^n d_v^2}}{
\vol(G)}
\right),\end{align}
which yields \eqref{kh2}. Moreover,
since $h(G)\le1$, it follows that
\begin{align}\label{chg7}
0\le h(G)-\widetilde h(G)
\le
\frac{\sqrt{\sum_{v=1}^n d_v^2}}{
 \vol(G)}.
\end{align}
Finally,
$\sum_{v}
d_v^2
\le  d_{\rm max}\sum_{v}
d_v
= d_{\rm max}\vol(G)$.
Together with \eqref{tch1m}, this gives
\begin{align}\label{chg9}
0 \le h(G)-h(W_G) \le \frac{\sqrt{\sum_{v=1}^n d_v^2}}{\vol(G)} \le \sqrt{\frac{ d_{\rm max}}{\vol(G)}},
\end{align}
as required.
\end{proof}

\begin{lemma}\label{lem:cheeger-cut-three}
Let $W_1$ and $W_2$ be graphons on the same probability space. Then, for every measurable function $\psi:\Omega\to [0,1]$,
\begin{align}
|\C(W_1,\psi)-\C(W_2,\psi)|&\le \|W_1-W_2\|_\square,\label{chfc}\\
|D(W_1,(\psi)-D(W_2,\psi)|&\le \|W_1-W_2\|_\square,\label{chfd}\\
|m(W_1,\psi)-m(W_2,\psi)|&\le\|W_1-W_2\|_\square.\label{chfm}
\end{align}
\end{lemma}

\begin{proof}
Recall that for every measurable kernel $K:\Omega^2\to[-1,1]$ and
measurable functions $a,b:\Omega\to[0,1]$, we have,
see \eqref{cutnorm2},
\begin{align}\label{chf1}
\left|\iint_{\Omega^2}K(x,y)a(x)b(y)\dd\mu(x)\dd\mu(y)\right|\le \|K\|_\square,
\end{align}
Set $K:=W_1-W_2$ and $\varepsilon:=\|W_1-W_2\|_\square$. Applying this bound with 
\begin{equation}
(a,b)=(\psi,1-\psi), \qquad (a,b)=(\psi,1), \qquad (a,b)=(1-\psi,1)
\end{equation}
gives, respectively,
\begin{align}
|\C(W_1,\psi)-\C(W_2,\psi)|&\le \varepsilon,\\
|m(W_1,\psi)-m(W_2,\psi)|&\le \varepsilon,\\
|(\rho_{W_1}-m(W_1,\psi))-(\rho_{W_2}-m(W_2,\psi))|&\le \varepsilon.
\end{align}
Finally, using $|\min\{r,s\}-\min\{r',s'\}|\le \max \{|r-r'|,|s-s'|\}$, we obtain 
\begin{align}|D(W_1,\psi)-D(W_2,\psi)|\le \varepsilon.\end{align}
\end{proof}

The next lemma gives the uniform denominator bound needed for stability.
\begin{lemma}\label{lem:cheeger-denominator}
 Let $W$ be a graphon satisfying $k_W \ge c>0$ a.e. 
Every measurable mixed colouring $\Psi=(\psi,1-\psi)$ satisfying
$D(W,\psi)>0$ and $H(W,\Psi) \le \frac12$
also satisfies
\begin{align}\label{chh1}
D(W,\psi) \ge \frac{c^2}{2}.
\end{align}

Furthermore, for every $\eta>0$, there exists a colouring such that \eqref{chh1} holds and
\begin{align}\label{chh2}
H(W,\Psi)\le h(W)+\eta.
\end{align}
\end{lemma}
\begin{proof}
Let $\Psi=(\psi,1-\psi)$ satisfy $D(W,\psi)>0$ and $H(W,\Psi)\le1/2$. By symmetry of~$W$,
\begin{align}
\C(W,1-\psi)=\C(W,\psi),
\qquad
D(W,1-\psi)=D(W,\psi).
\end{align}
Thus, replacing $\psi$ by $1-\psi$ if necessary, we may assume that
\begin{align}
D(W,\psi)=m(W,\psi)\le\frac{\rho_W}{2}.
\end{align}
Set $a:=\int_\Omega \psi(x)\,\dd\mu(x)$.
Since $m(W,\psi)>0$, we have $a>0$.
The degree lower bound gives $m(W,\psi)\ge ca$.
Moreover,
\begin{align}
\C(W,\psi)=m(W,\psi)-\iint_{\Omega^2}
W(x,y)\psi(x)\psi(y)\,\dd\mu(x)\dd\mu(y).
\end{align}
Since $\C(W,\psi)=H(W,\psi)D(W,\psi)\le D(W,\psi)/2=m(W,\psi)/2$, it follows that
\begin{align}
\frac{m(W,\psi)}2 \le \iint_{\Omega^2}
W(x,y)\psi(x)\psi(y)\,\dd\mu(x)\dd\mu(y).
\end{align}
Using $0\le W\le1$, we obtain
\begin{align}
\frac{ca}{2}
\le \frac{m(W,\psi)}2
\le \iint_{\Omega^2}\psi(x)\psi(y)\,\dd\mu(x)\dd\mu(y)=a^2.
\end{align}
Dividing by $a>0$ yields $a\ge c/2$, and therefore
\begin{align}
D(W,\psi)=m(W,\psi)\ge ca\ge\frac{c^2}{2}.
\end{align}
Since replacing $\psi$ by $1-\psi$ leaves $D(W,\psi)$ unchanged,
this proves the bound for the original colouring as well.

Finally, fix $\eta>0$. 
If $h(W)<1/2$, then
by the definition of $h(W)$, there exists an admissible colouring
$\Psi=(\psi,1-\psi)$ such that $H(W,\Psi)<\min\set{h(W)+\eta,1/2}$. 
This colouring has the required properties by the first part of the proof.

Otherwise, if $h(W)\ge1/2$, take the constant colouring
$\Psi:=(\frac12,\frac12)$. We then have, 
as in the proof of \refL{LCH0},
$\C(W,\psi)=\rho_W/4$, $D(W,\psi)=m(W,\psi)=\rho_W/2$, and thus
\begin{align}
H(W,\Psi)&=\frac12<h(W)+\eta,
\\
D(W,\psi)& =\int_\gO\tfrac12  k_W(x)\dd\mu(x)\ge \frac{c}{2} \ge\frac{c^2}{2}.
\end{align}
since $c\le k_W\le1$ a.e.
\end{proof}

\begin{lemma}\label{lem:cheeger-pullback}
Let\/ $W$ be a graphon on $(\Omega_2,\mathcal F_2,\mu_2)$ with
$\rho_W>0$, and let $\varphi: (\Omega_1,\mathcal F_1,\mu_1) \to
(\Omega_2,\mathcal F_2,\mu_2)$ be a measure-preserving function. Define
$W^\varphi(x,y):=W(\varphi(x),\varphi(y))$ as in \eqref{Wpull}. 
Then,  
\begin{align}\label{chpu6}
h(W^\varphi)=h(W).\end{align}
Moreover,
\begin{align}\label{chpu7}
k_{W^\varphi}=k_W \circ\varphi \quad \text{a.e., and}\quad \rho_{W^\varphi}=\rho_W.\end{align}
\end{lemma}
\begin{proof}
We follow the conditional-expectation argument used in the
proof of Lemma~\ref{LMC}.

Since $\varphi$ is measure-preserving,
\begin{align}\label{chpu8}
k_{W^\varphi}(x)
=\int_{\Omega_1}
W(\varphi(x),\varphi(y))\,\dd\mu_1(y)=k_W(\varphi(x)) \quad \text{a.e}.
\end{align}
Integrating this identity gives
$\rho_{W^\varphi}=\rho_W$. For every measurable $g:\Omega_2\to[0,1]$, change of
variables gives, see \eqref{lmc2} 
\begin{align}\label{chpu9}
\C(W^\varphi,\, g\circ\varphi)=\C(W,g),
\qquad
m({W^\varphi},\, g\circ\varphi)=m(W,g).
\end{align}
Consequently,
\begin{align}\label{chpu10}
D(W^\varphi,\, g\circ\varphi)=D(W,g).
\end{align}
Thus every admissible mixed colouring of $W$ pulls back to
a mixed colouring of $W^\varphi$ with the same Cheeger ratio. Taking infima yields
\begin{align}\label{chpu11}
h(W^\varphi)\le h(W).
\end{align}

For the reverse inequality, let $\psi:\Omega_1\to[0,1]$ be
measurable with $D(W^\varphi,\psi)>0$. 
As in the proof of Lemma~\ref{LMC}, let
\begin{align}\label{chpu12}
\bpsi:=\mathbb E_{\mu_1}[\psi\mid\mathcal F_\varphi],
\end{align}
where $\mathcal F_\varphi$ is the $\sigma$-field generated by $\varphi$. The
same argument gives a measurable $g:\Omega_2 \to [0,1]$ s.t. $\bpsi=g
\circ\varphi$ and 
\begin{align}\label{chpu13}
\C(W^\varphi,\,\psi)
=\C(W^\varphi,\, \bpsi)= \C(W,g).\end{align}
Moreover, since $k_{W^\varphi}=k_W\circ\varphi$ is
$\mathcal F_\varphi$-measurable, the defining property of conditional expectation and change of variables give
\begin{align}\label{chpu14}
m(W^\varphi,\psi)
&=\int_{\Omega_1}k_{W^\varphi}(x)\psi(x)\,\dd\mu_1(x)=\int_{\Omega_1}k_{W^\varphi}(x)\bpsi(x)\,\dd\mu_1(x)=m(W,g).
\end{align}
Together with $\rho_{W^\varphi}=\rho_W$, this implies $D(W^\varphi,\psi)=D(W,g)>0$. Consequently,
\begin{align}
H(W^\varphi,(\psi,1-\psi))
=H(W,(g,1-g))
\ge h(W).
\end{align}
Taking the infimum over all admissible $\psi$ gives $h(W^\varphi)\ge h(W)$. Combining this with the opposite inequality proves the claim.
\end{proof}

\begin{proposition}\label{prop:cheeger-stability}
Let $W_1$ and $W_2$ be graphons on possibly different probability spaces. 
Suppose that $k_{W_1},k_{W_2} \ge c$ a.e.
Then, 
\begin{align}|h(W_1)-h(W_2)|\le \frac{4}{c^2}\delta_\square(W_1,W_2).\end{align}
\end{proposition}

\begin{proof}
First, assume that $W_1$ and $W_2$ are defined on the same probability space $(\Omega,\mathcal F,\mu)$, and let $\varepsilon:=\|W_1-W_2\|_\square$. For $\eta>0$, Lemma \ref{lem:cheeger-denominator} provides a mixed colouring
$\Psi=(\psi,1-\psi)$ such that
\begin{align}H(W_2,\Psi)\le h(W_2)+\eta,
\qquad
D(W_2,\psi)\ge\frac{c^2}{2}.
\end{align} 
Since $D(W_2,\psi)>0$, both $\psi$ and $1-\psi$ have positive integrals.
The assumption $k_{W_1}\ge c$ therefore gives $D(W_1,\psi)>0$.
Also, by symmetry, $0\le \C(W_1,\psi)\le D(W_1,\psi)$. Using \refL{lem:cheeger-cut-three}, we obtain
\begin{align}
H(W_1,\Psi) - H(W_2,\Psi)&=\frac{ \C(W_1,\psi)-\C(W_2,\psi) + H(W_1,\Psi)\bigl(D(W_2,\psi)-D(W_1,\psi)\bigr)}
     {D(W_2,\psi)}\notag\\
     &\le \frac{2\varepsilon}{D(W_2,\psi)}
\le \frac{4\varepsilon}{c^2}.
\end{align}
Consequently,
\begin{align}
h(W_1)
\le H(W_1,\Psi)
\le H(W_2,\Psi)+\frac{4\varepsilon}{c^2}
\le h(W_2)+\eta+\frac{4\varepsilon}{c^2}.
\end{align}
Letting $\eta\to 0$ and repeating the argument with
$W_1$ and $W_2$ interchanged, we obtain
\begin{align}
|h(W_1)-h(W_2)|
\le \frac{4\varepsilon}{c^2}
= \frac{4}{c^2}\|W_1-W_2\|_{\square}.
\end{align}
For graphons on possibly different probability spaces, apply the
preceding inequality to any common pullbacks $W_1^{\varphi_1}$ and
$W_2^{\varphi_2}$, see \eqref{dcut1}.
Lemma \ref{lem:cheeger-pullback} preserves their Cheeger constants and degree lower bounds, so
\begin{align}|h(W_1)-h(W_2)|\le \frac{4}{c^2}\|W_1^{\varphi_1}-W_2^{\varphi_2}\|_\square.\end{align}
Taking the infimum over all such pullbacks proves the result.
\end{proof}

In order to show continuity of the Cheeger constant, we impose a lower bound
on the degrees for both graphs and graphons.

\begin{definition}\label{Def:cWc}
  For $c\in\oi$, let $\cWc$ be the set of all graphons
(on any probability space)
such that $k_W\ge c$ a.e.
\end{definition}
\noindent
(In formal set theory sense, this is strictly speaking not a set but a
class; this makes no difference for our purposes.)
Note that for a graphon $W$ on a probability space $\OFmu$,
we have $W\in\cWc$ if and only if, recalling \eqref{chm},
\begin{align}\label{chc}
m(W,\etta_A):= \int_A k_{W}(x)\dd\mu(x)
\ge c\mu(A)  
\end{align}
for every (measurable) $A\subseteq\gO$.
Furthermore, 
for a graph $G$ with $|G|=n$, 
we have $W_G\in\cWc$ if and only if 
\begin{align}\label{chc1}
 d_{\rm min}(G):=\min_{1\le v\le n} d_v
\ge cn.  
\end{align}
We note next that
this definition is invariant under equivalence of graphons.

\begin{lemma}\label{LCWc}
  Let $W_1$ and $W_2$ be two graphons, possibly on different probability
  spaces, and let $c\in\oi$.
If\/ $W_1\cong W_2$, then $W_1\in\cWc \iff W_2\in\cWc$.
\end{lemma}
\begin{proof}
Let $W_j$ be defined on the probability space $(\gO_j,\cF_j,\mu_j)$ 
($j=1,2$).

Suppose first that $W_1=W_2^\gf$ is a pullback of $W_2$ by a \mpp{} map
$\gf:\gO_1\to\gO_2$.
As seen in \refL{lem:cheeger-pullback}, then $k_{W_1}=k_{W_2}\circ\gf$, and a
change of variables yields, for any $A\subseteq\gO_2$,
\begin{align}\label{chc3}
  \int_{A}k_{W_2}(x)\dd\mu_2(x)
=\int_{\gf\qw(A)} k_{W_2}(\gf(y))\dd\mu_1(y)
=\int_{\gf\qw(A)} k_{W_1}(y)\dd\mu_1(y).
\end{align}
Hence, if $W_1\in\cWc$ and thus
\eqref{chc} holds for $W_1$ and any $A\subseteq \gO_1$, then 
\eqref{chc} holds for $W_2$ too, and thus $W_2\in\cWc$.

Conversely, suppose that $W_2\in\cWc$, and let $A\subseteq\gO_1$.
We let $\psi:=\sindic{A}$, and argue  
as in the proof of \refL{lem:cheeger-pullback}: we let
$\cF_\varphi$ be the $\sigma$-field generated by $\varphi$, and
define $\bpsi$ by \eqref{chpu12}.
Then, again, there exists a measurable $g:\gO_2\to\oi$ such that
$\bpsi=g\circ\gf$, and thus, using \eqref{chpu14},
\begin{align}\label{chc4}
 m(W_1,\etta_A)&
= m(W_2,g)
=  \int_{\gO_2}g(y)k_{W_2}(y)\dd\mu_2(y)
\notag\\&
\ge c\int_{\gO_2}g(y)\dd\mu_2(y)
=c\int_{\gO_1}g(\gf(x))\dd\mu_1(x)
=c\int_{\gO_1}\bpsi(x)\dd\mu_1(x)
\notag\\&
=c\int_{\gO_1}\psi(x)\dd\mu_1(x)
=c\mu_1(A).
\end{align}
Hence \eqref{chc} is satisfied for $W_1$, and thus $W_1\in\cWc$.

This proves the result when $W_1$ is a pullback of $W_2$.
The general case follows because if $W_1\cong W_2$, then there is a third
graphon $W$ such that  both $W_1$ and $W_2$ are pullbacks of $W$, see
\refT{Tunique0}, and hence $W_1\in\cWc\iff W\in\cWc\iff W_2\in\cWc$.
\end{proof}

\refL{LCWc} shows that we also can regard $\cWc$ as a set of graph limits
(recall that these can be defined as equivalence classes of graphons).
The following lemma shows that then
$\cWc$ is a closed subset of the compact space of graph limits.

\begin{lemma}\label{LCWc2}
Let $W_n$ and $W$ be graphons and suppose that $\dcut(W_n,W)\to0$ as \ntoo.
If\/ $c\in\oi$ and all $W_n\in\cWc$, then $W\in\cWc$.
\end{lemma}

\begin{proof}
  Using \refL{LCWc} and \eqref{dcut1}, 
we may assume that all $W_n$ and $W$ are defined on the
  same probability space $\OFmu$ (for example $\oi$) and that
  $\cnorm{W_n-W}\to0$. Then, for every $A\subseteq\gO$, by \eqref{chfm},
  \begin{align}\label{lcw9}
\bigabs{m(W_n,\etta_A)-m(W,\etta_A)}
\le\cnorm{W_n-W} \to0.
  \end{align}
We have $m(W_n,\etta_A)\ge c\mu(A)$ by \eqref{chc} since  $W_n\in\cWc$, and
thus it follows from \eqref{lcw9} 
that $m(W,\etta_A)\ge c\mu(A)$. Consequently, $W\in\cWc$.
\end{proof}

We can now show that
the Cheeger constant is a continuous functional on each set $\cWc$ with $c>0$.
\begin{theorem}\label{TCH-conv}
Let $W_n$ and $W$ be graphons and suppose that $\dcut(W_n,W)\to0$ as \ntoo.
Assume that for some $c>0$, $W_n\in\cWc$ for all sufficiently large $n$. 
Then,  
\begin{align}\label{tchc}
h(W_n) \to h(W).
\end{align}
\end{theorem}
\begin{proof} 
By \refL{LCWc2}, we have $W\in\cWc$.
Hence, $k_W \ge c$ a.e.\ and $k_{W_n}\ge c$ a.e.\ for all sufficiently
large $n$. Proposition \ref{prop:cheeger-stability} therefore 
implies $h(W_n)\to h(W)$.
\end{proof}

\begin{remark}\label{rem.h_conv_not_easy}
Suppose we have a sequence of $(H_n)_n$ such that $|H_n|=n$ and $H_n \to W_H$ for some graphon $W_H$ with $\rho_{W_H}>0$. Construct the graph
$G_{n'}$ be taking the disjoint union of $H_n$ with a clique $S$ on $n/\log n$ vertices, then adding a single edge connecting $V(H_n)$ and $S$. Then we may note that $G_{n'}\to W_H$.

Since $\rho_{W_H}>0$, for large enough $n'$, we have $\vol_{G_{n'}}(S)\leq \vol(G)/2$. This implies that for $n'$ large enough we have \begin{align}
h(G_{n'})\leq \frac{e_{G_{n'}(S, S\comp)}}{\vol_{G_{n'}}(S)}=O\Bigpar{\frac{\log^2 n}{n^2}}=o(1)\end{align}
Thus, if we take any sequence $(H_n)_n$ with $h(H_n)=\Omega(1)$, then we have that $H_n, G_n\to W_H$ i.e.,\ the graph limits coincide while the corresponding limits of the Cheeger constants of the sequences do not, since $h(G_n)\to 0$ and $h(H_n)\not\to 0$. 
\end{remark}

This leads to a similar continuity result for graphs.  We note that by Remark~\ref{rem.h_conv_not_easy} in the theorem below it would not be enough to replace the condition on $d_{\rm min}(G_n)$ with the condition that $k_W(x)\geq c$ a.e.. Indeed, it would not be enough to replace the min-degree condition on $G_n$ with any condition on the graphon for which $G_n\to W$. 

\begin{theorem}\label{thm:cheeger-conv}
Let $G_n$ be graphs such that $|G_n|\to\infty$, and suppose that 
$G_n\to W$
for a graphon $W$. 
Assume that for some $c>0$,
\begin{align}\label{emma1}
 d_{\rm min}(G_n)  
\ge c|G_n|   
\end{align}
for all sufficiently large $n$. 
Then, $k_W \ge c$ a.e.\ and
\begin{align}\label{emma}
h(G_n) \to h(W),
\end{align}
where $h(W)$ is the mixed Cheeger constant.
\end{theorem}
\begin{proof} 
We may assume that \eqref{emma1} holds for all $n$.
Assume also, for notational convenience, that $|G_n|=n$.
As noted in \eqref{chc1}, \eqref{emma1} is equivalent to $W_{G_n}\in\cWc$.
Furthermore, 
the assumption $G_n\to W$ is equivalent to
\begin{align}\label{emma3}
\dcut(W_{G_n},W)\to0
\end{align}
Consequently, \refT{TCH-conv} yields $h(W_{G_n})\to h(W)$.
Furthermore, by
Lemma \ref{lem:cheeger-graph-graphon}
and the bounds $d_{\rm max}(G_n)\leq n$ and $\vol(G_n)\ge cn^2$
(which follows from \eqref{emma1}), we have
\begin{align}\label{emma4}
0\le h(G_n)-h(W_{G_n})
\le\sqrt{\frac{d_{\rm max}(G_n)}{\vol(G_n)}}
\le\frac1{\sqrt{cn}}
\to 0.
\end{align}
Combining these two limits proves $h(G_n)\to h(W)$.

\noindent Finally, $W\in\cWc$ by \refL{LCWc2}, and thus $k_W\ge c$ a.e.
\end{proof}

\begin{theorem}\label{TCH4}
  Let $W$ be a graphon on $[0,1]$ with $\rho_W>0$. 
Then $h_\pure(W)=h_\mix(W)$.
\end{theorem}
\begin{proof}
We follow the block-approximation argument in the proof of  \refT{TMC4}. Fix
a 
mixed colouring $\Psi=(\psi,1-\psi)$ of $W$ with $D(W,\Psi)>0$.

For $N\ge1$, partition $[0,1]$ into intervals
$I_1,\ldots,I_N$ of length $1/N$, and let $W_N$ be obtained
by averaging $W$ over each rectangle $I_i\times I_j$. Then, as $\Ntoo$,
\begin{align}\label{emma5}
\varepsilon_N:=\|W-W_N\|_{L^1([0,1]^2)} \to 0.\end{align}
Choose a pure 2-colouring $\chi$ such that \eqref{tmc4} holds,
and set $\psi_N:=\etta_{\chi\qw(1)}$
and $\Psi_N:=(\psi_N,1-\psi_N)$. 
Since $W_N$ is constant on each partition rectangle,
integrating over these rectangles gives,
similarly to \eqref{mcx1b},
\begin{align}\label{emma6}
\C(W_N, \psi_N)=\C(W_N,\psi),
\qquad
m(W_N, \psi_N)=m(W_N,\psi).\end{align}
Consequently, $D(W_N,\psi_N)=D(W_N,\psi)$.
Lemma \ref{lem:cheeger-cut-three} together with \eqref{emma6} and
$\|W-W_N\|_\square\le\varepsilon_N$ imply
\begin{align}\label{emma7}
|\C(W,\psi_N)-\C(W,\psi)|\le2\varepsilon_N,
\qquad
|D(W,\psi_N)-D(W,\psi)|\le2\varepsilon_N.\end{align}
Since $D(W,\psi)>0$, \eqref{emma7} and \eqref{emma5} imply
that $D(W,\psi_N)>0$ for all sufficiently large $N$.
Furthermore, recalling \eqref{cheh},
for all sufficiently large $N$, and 
furthermore, recalling \eqref{cheh},
\begin{align}H(W,\Psi_N)\to H(W,\Psi).\end{align}
Thus the infimum $h_\pure(W)$ over pure colourings is at most
$H(W,\Psi)$ for every admissible mixed colouring $\Psi$. 
Taking the infimum over all such colourings $\Psi$, this yields
$h_\pure(W) \le h_\mix(W)$.
The reverse inequality follows because every pure colouring
is also mixed.
\end{proof}

\subsection{Minimum $k$-section}\label{subsec.min_k}
Fix $k\ge 2$. For a graph $G$ on vertex set $[n]$, define the normalised
minimum $k$-section value by
\begin{align}\label{mink}
\MS_k(G):=\frac{1}{n^2}\min_{\substack{A_1,\ldots,A_k\text{ partition of }[n]\\ |A_i|=\lfloor n/k\rfloor \text{ or } \lceil n/k\rceil}}
\sum_{1\le i<j\le k} e_G(A_i,A_j),
\end{align}
where $e_G(A_i,A_j)$ is the number of unordered edges with one endpoint in $A_i$ and the other in $A_j$. Let $W$ be a graphon on a probability space $(\gO,\mathcal F,\mu)$.
As in Section~\ref{SSSmax}, we distinguish pure and mixed colourings. Define the pure minimum $k$-section value by
\begin{align}\label{MS1}
\MS_{k,\text{pure}}(W):=\inf_{\substack{
S_1,\ldots,S_k\text{ measurable partition of }\gO\\
\mu(S_i)=1/k,\ i\in[k]}}
\sum_{1\le i<j\le k}
\iint_{S_i\times S_j}
W(x,y)\,\dd\mu(x)\,\dd\mu(y).
\end{align}
We use the convention $\inf\varnothing=+\infty$, since a
general probability space need not admit such a partition. For $k=2$, this recovers the \emph{minimum bisection problem}. Recall from~\eqref{MCGx} that for a pure $k$-colouring $\chi$,
\begin{align}
C_k(W,\chi)=
\frac12\iint_{\gO^2}
W(x,y)\mathbf 1[\chi(x)\neq\chi(y)]\,
\dd\mu(x)\,\dd\mu(y).
\end{align}
(The factor $1/2$ again appears because the double integral counts both
ordered directions between distinct colour classes.) Thus, similar to maximum cut, we can cast minimum $k$-section as a constrained pure $k$-colouring problem:
\begin{align}\label{MS2}
\MS_{k,\text{pure}}(W)=
\inf_{\substack{
\chi:\gO\to[k]\text{ measurable}\\
\mu(\chi^{-1}(i))=1/k,\ i\in[k]}}
{C_k(W,\Psi)}
\end{align}
Thus, the finite balance condition becomes the colour-mass constraint
$\mu(S_i)=1/k$ for each colour class.
Recall from~\eqref{MCWx2} and~\eqref{MC3} that for a mixed $k$-colouring $\Psi=(\psi_1,\ldots,\psi_k)$, we have 
\begin{align}
\C_k(W,\Psi)
&=\sum_{1\le i<j\le k}\iint_{\gO^2}
W(x,y)\psi_i(x)\psi_j(y)\,
\dd\mu(x)\,\dd\mu(y)\label{msk1}\\
&=\sum_{1\le i<j\le k}
t\bigl(K_{2;ij},(W,\Psi)\bigr).\label{msk2}
\end{align}
 We will sometimes call $\C_k(W, \Psi)$ the colouring score of $\Psi$.
We call $\Psi$ balanced if
$\int_{\gO}\psi_i\,\dd\mu=1/k$ for every $i\in[k]$, and define
\begin{align}\label{msk21}
\MS_{k,\mathrm{mix}}(W):=
\inf_{\substack{
\Psi\in\Lambda_k(\gO),\\
\Psi \text{ balanced}}}\C_k(W,\Psi).
\end{align}
The mixed admissible set is always nonempty, since the
constant colouring $\psi_i\equiv1/k$ is balanced. Restricting to pure colourings
recovers $\MS_{k,\pure}(W)$.

As for maximum cut, we take the mixed version as the default:
\begin{align}\label{msk22}
\MS_k(W):=\MS_{k,\mathrm{mix}}(W).
\end{align}
We next show the correspondence between finite and graphon minimum $k$-sections. Because of the slightly different forms of the balance constraint, we
  cannot expect to have equality as for maximum cut in \refT{TMC1},
but the difference vanishes asymptotically for large $n$. The first lemma identifies the normalised finite minimum $k$-section with
that of the associated step graphon up to an $o(1)$ error, while the second
proves cut-norm continuity and characterises the limit along cut-convergent
graph sequences. 
\begin{lemma}\label{lem:minksection-graph-step}
Let $k \ge 2$ be fixed. Let $G$ be a graph on vertex set $[n]$, and let
$W_G$ be the associated step graphon
defined in \refE{EG1}.
Then 
\begin{align}\label{aaaL}
\MS_k(G)=\MS_k(W_G)+o(1),
\end{align}
as $n\to \infty$.
\end{lemma}
\begin{proof}
A finite $k$-section $A_1,\ldots,A_k$ corresponds to the measurable partition
of $\gO$ given by
$S_i:=\bigcup_{v \in A_i} I_v$, for $i \in [k]$. 
For this partition, 
\begin{align}\sum_{1 \le i<j \le k}\int_{S_i \times S_j}W_G(x,y) \dd\mu(x)\dd\mu(y)=\frac{1}{n^2}\sum_{1\le i<j\le k}e_G(A_i,A_j).\end{align}
The only discrepancy is the balance condition: the sets $S_i$ associated with a
finite balanced $k$-section have measures 
$\mu(S_i)=|A_i|/n$ which is 
within $1/n$ of $1/k$. 
Since each class measure differs from $1/k$ by at most $1/n$, the total
mass that must be redistributed to make all class measures equal to $1/k$
is $O_k(1/n)$. This produces a balanced mixed colouring
$\Psi^0=(\psi_1^0,\ldots,\psi_k^0)$ such that
\begin{align}\label{aaaa}
\sum_{i=1}^k
\int_{\gO}|\psi_i^0-\mathbf1_{S_i}|\,\dd\mu
=
\sum_{i=1}^k|\mu(S_i)-1/k|
\le \frac{k}{n}.
\end{align}
Since $0\le W_G\le1$, the corresponding change in the
crossing score satisfies
\begin{align}
&\Big|
\C_k(W_G,\Psi^0)
-
\sum_{i<j}\iint_{S_i\times S_j}
W_G(x,y)\,\dd\mu(x)\,\dd\mu(y)
\Big|
\le
\sum_{i=1}^k
\int_{\gO}|\psi_i^0-\mathbf1_{S_i}|\,\dd\mu
=O_k(n^{-1}).
\end{align}
Taking the minimum over all finite $k$-sections of $G$, we obtain
\begin{align}\label{aaa0}
    \MS_k(W_G)
    \le
    \MS_k(G)+O_k(1/n).
\end{align}
Conversely, let $\Psi=(\psi_1,\ldots,\psi_k)$ be a balanced
mixed colouring of $W_G$. Define
\begin{align}
a_v(i):=n\int_{I_v}\psi_i(x)\,\dd\mu(x),
\qquad v\in[n],\ i\in[k].
\end{align}
Then 
\begin{align}\label{aaa1}
0\le a_{v}(i)\le1,\qquad
\sum_{i=1}^k a_{v}(i)=1,\qquad
\sum_{v=1}^n a_{v}(i)=n/k.  
\end{align}
Since $W_G$ is constant on each rectangle $I_u\times I_v$,
we may compute the mixed colouring score
$\C_k(W_G,\Psi)$ in \eqref{msk1}
edge by edge. For each edge $\{u,v\}$, summing over $1\le i<j\le k$ gives 
a contribution that is $n^{-2}$ times
\begin{align}\label{aaa2}
  \sum_{1\le i<j\le k}\bigl(a_u(i)a_v(j)+a_u(j)a_v(i)\bigr)
=
1-\sum_{i=1}^k a_u(i)a_v(i).
\end{align}
Thus the score of this mixed colouring is
\begin{align}\label{aaa3}
  \C_k(W_G,\Psi)=
 \frac1{n^2}
  \sum_{\{u,v\}\in E(G)}
  \left(1-\sum_{i=1}^k a_u(i)a_v(i)\right).
\end{align}
Now independently assign each vertex $v\in G$ a random
colour $C_v\in[k]$ with 
\begin{align}\label{msk3}
\Prob(C_v=i)=a_v(i),  
\end{align}
cf.\ \eqref{aaa1}, and let $A_i:=\set{v:C_v=i}$. Denote the corresponding normalised crossing value by
\begin{align}
\C_k\bigpar{G;(A_i)_1^k}:=
\frac{1}{n^2}\sum_{1\le i<j\le k} e_G(A_i,A_j).
\end{align}
Then the expected normalised crossing value is, 
using \eqref{aaa2}--\eqref{msk3},
\begin{align}\label{aaa5}
&\E\bigsqpar{ \C_k\bigpar{G;(A_i)_1^k}}
= \frac1{n^2}\E
\sum_{1\le i<j\le k}
  \sum_{\{u,v\}\in E(G)}
\bigpar{\sindic{C_u=i}\sindic{C_v=j}+\sindic{C_u=j}\sindic{C_v=i}}
\notag\\&\qquad =
\C_k\bigpar{W_G,\Psi}  
.\end{align}
Let
$N_i:=|\{v:C_v=i\}|$. Then, using \eqref{aaa1} again, 
$\E N_i=\sum_{v=1}^na_v(i)=n/k$ and $\operatorname{Var}(N_i)\le n$,
so $\mathbb E|N_i-n/k|\le \sqrt n$ and hence $\mathbb E\sum_{i=1}^k |N_i-n/k|=O_k(\sqrt n)$. Choose target sizes $b_1,\ldots,b_k\in\{\lfloor n/k\rfloor,\lceil n/k\rceil\}$
with $\sum_i b_i=n$. Since $|b_i-n/k|\le1$, we also have $\E\sum_{i=1}^k |N_i-b_i|=O_k(\sqrt n)$.
By recolouring vertices from classes with $N_i>b_i$ into classes with $N_i<b_i$,
the colouring can be made balanced after changing $O_k(\sqrt n)$ vertices in
expectation. Recolouring one vertex changes the crossing-edge count by at most
$n-1$, so the normalised value changes by $O_k(n^{-1/2})$ in expectation.
Hence, there exists a random balanced partition $(A_i')_1^k$ such that,
using \eqref{aaa5},
\begin{align}\label{aaa6}
\E\bigsqpar{\C_k\bigpar{G;(A'_i)_1^k}}
&\le
\E\bigsqpar{\C_k\bigpar{G;(A_i)_1^k}}
+ O_k(n^{-1/2})
\notag\\&
=
\C_k\bigpar{W_G,\Psi}  
+ O_k(n^{-1/2})
.\end{align}
Hence, some balanced realization $(A'_i)_1^k$
has value at most 
$\C_k\bigpar{W_G,\Psi} + O_k(n^{-1/2})$, and thus
\begin{align}\label{aaa6b}
&\MS_k(G)
\le \C_k\bigpar{G;(A'_i)_1^k}
\le
\C_k\bigpar{W_G,\Psi}  
+ O_k(n^{-1/2})
.\end{align}
Taking the infimum over all balanced mixed colourings
$\Psi$ gives
\begin{align}\label{aaa7}
  \MS_k(G)
  \le
  \MS_k(W_G)+O_k(n^{-1/2}).  
\end{align}

Combining the two inequalities 
\eqref{aaa0} and \eqref{aaa7}
yields \eqref{aaaL}.
\end{proof}

We obtain some results similar to corresponding result for maximum cut in
\refSS{SSSmax}.

\begin{lemma}\label{LMS1} 
If\/ $W_1$ and $W_2$ are graphons on the same probability space
$(\gO,\mathcal F,\mu)$, then 
\begin{align}\label{lms1}
|\MS_k(W_1)-\MS_k(W_2)| \le \tbinom{k}{2}\|W_1- W_2\|_\square.  
\end{align}
\end{lemma}
\begin{proof}
For every balanced mixed $k$-colouring $\Psi$ of $\gO$, \eqref{msk1} and
\eqref{cutnorm2} give
\begin{align}
|\C_k(W_1,\Psi)-\C_k(W_2,\Psi)| 
&\le
\sum_{1\le i<j\le k}
\left|
\iint_{\gO^2}(W_1-W_2)(x,y)\psi_i(x)\psi_j(y)
\,\dd\mu(x)\,\dd\mu(y)
\right| \notag\\
&\quad\le \tbinom{k}{2}\|W_1-W_2\|_\square.
\end{align}
Taking infima over the same set of balanced mixed colourings
proves \eqref{lms1}.
\end{proof}

\begin{theorem}\label{TMS3}
If\/ $W_1$ and $W_2$ are 
graphons on possibly different probability
spaces, then 
\begin{align}\label{tms31}
|\MS_k(W_1)-\MS_k(W_2)| 
\le \tbinom k2 \dcut(W_1,W_2).
\end{align}
In particular, if $W_1\cong W_2$, then $\MS_k(W_1)=\MS_k(W_2)$.

Consequently, if\/ $W_n\to W$ in cut distance, then $\MS_k(W_n)\to\MS_k(W)$.
In other words, $\MS_k$ is a continuous functional
on the set of graphons equipped with cut distance, 
and thus also on the quotient space $\hcWo$ of graph limits.
\end{theorem}

\begin{proof}
The pullback and conditional-expectation arguments in the proof of \refL{LMC}, applied coordinatewise, preserve both the score and each colour mass. Hence, they also preserve the balance constraints, and give 
\begin{equation}\label{eq:pullback-ms}\MS_k(W^\varphi)=\MS_k(W),
\end{equation}
for every measure-preserving map $\varphi$.

For any common pullbacks $W_1^{\varphi_1}$ and $W_2^{\varphi_2}$, \eqref{eq:pullback-ms}  and \refL{LMS1} give
\begin{equation}
|\MS_k(W_1)-\MS_k(W_2)|
\le \tbinom{k}{2}
\|W_1^{\varphi_1}-W_2^{\varphi_2}\|_\square.
\end{equation}
Taking the infimum over all such pullbacks, as in \eqref{dcut}, proves \eqref{tms31}. This bound also shows that $\MS_k$ is invariant under equivalence
and continuous with respect to cut distance.
\end{proof}

\begin{theorem}\label{Lmink}
Let $k\ge2$ be fixed, and let $W$ be a graphon on an arbitrary
probability space $(\gO,\mathcal F,\mu)$. If\/ $G_n$ is a sequence of graphs with
$|G_n|\to\infty$ and $G_n\to W$, then
  \begin{align}\label{lmink}
\MS_k(G_n)\longrightarrow \MS_k(W).
  \end{align}
\end{theorem}
\begin{proof}
Represent each $G_n$ by its usual step graphon $W_{G_n}$ on $\oi$, while $W$ remains on $(\gO,\mathcal F,\mu)$. By \refT{TMS3}, which allows different underlying probability spaces, $G_n\to W$ implies $\MS_k(W_{G_n})\to\MS_k(W)$.
Moreover, \refL{lem:minksection-graph-step}, 
applied to $W_{G_n}$ on $\oi$, yields
\begin{equation}
\MS_k(G_n)
=\MS_k(W_{G_n})+o(1)
.\end{equation}
Thus, $\MS_k(G_n)\to\MS_k(W)$, as required.
\end{proof}

\begin{theorem}\label{TMS4}
  Let\/ $W$ be a graphon on $\oi$. Then
$\MS_{k,\text{mix}}(W)=\MS_{k,\text{pure}}(W)$.
\end{theorem}

\begin{proof}
This is similar to the proof of \refT{TMC4}, so we will omit some details.
Suppose again first that $W$ is a step graphon,
constant on every square $I_{i}\times I_{j}$, $i,j\in[N]$,
with $I_{i}:=[(i-1)/N,i/N)$ as above.
Given a balanced mixed $k$-colouring $\Psi$, choose a measurable
partition $S_1,\ldots,S_k$ of $\oi$ by splitting each interval $I_r$
so that
\begin{align}
\lambda(I_r\cap S_j)=\int_{I_r}\psi_j\,d\lambda,
\qquad r\in[N],\ j\in[k].
\end{align}
Define $\chi=j$ on $S_j$.
This colouring is balanced, since
$\lambda(S_j)=\sum_{r=1}^N\lambda(I_r\cap S_j)
=\int_0^1\psi_j\,d\lambda=1/k$.
Then, by \eqref{MC3},
\begin{align}\label{msx1b}
\C_k(W,\Psi) =
\sum_{1\le i<j\le k} t\bigpar{\KK_{2;ij},(W,\Psi)}
=
\C_k(W,\chi):=
\sum_{1\le i<j\le k} t\bigpar{\KK_{2;ij},(W,\chi)}.
\end{align}
The result now follows as in \eqref{tmc43}, using 
the approximating step graphons $W_N$ as before and \refL{LMS1}.
\end{proof}

\begin{theorem}\label{TMS5}
Let\/ $W$ be a graphon on a probability space $\OFmu$, and fix $k\ge2$.
Then the supremum in \eqref{msk21} is attained. In other words, there exists
a balanced mixed colouring $\Psi$ of $W$ such that $\MS_k(W)=\C_k(W,\Psi)$.
\end{theorem}

\begin{proof}
The mappings $\Psi\to\int_{\gO}\psi_j\dd\mu$, $j=1,\dots,k$, 
are continuous of the set
$\gL=\gL_k$ of mixed $k$-colourings of $W$ for the \wx-topology defined in
\refSS{SStop}. Hence, the set of balanced colourings is a closed subset of
$\gL$, which as noted in \refSS{SStop} is compact.
Consequently, the set of balanced mixed colourings of $W$ is compact,
and the existence of an optimising $\Psi$ follows from the continuity of 
the map $\Psi\mapsto\C_k(W,\Psi)$, which is a consequence of \eqref{msk2}.
\end{proof}

\begin{remark}\label{rem:conn-bclsv} In this remark we describe the family of partition problems studied in~\cite{BCLSV2}.
For a fixed number of colours $\kappa\ge 1$, symmetric interaction matrix $J\in\mathbb R^{\kappa\times\kappa}$ and $h\in \IR^\kappa$, consider the optimisation problem over pure colourings $\Psi$ of a graph~$G$ 
\begin{align}\label{eq.ground_state_energy}
     \max_{\Psi:V(G)\to [\kappa]} \bigg( \frac{1}{|G|}\sum_{u\in V(G)} h_{\varphi(u)} + \frac{2}{|G|^2}\sum_{uv\in E(G)} J_{\Psi(u)\Psi(v)}\bigg). 
\end{align}
which is known as the ground state energy.
Consider also the case where $h$ is the zero vector and the optimsation is over the set pure colourings $\Psi:V(G)\to [\kappa] $ that assign colour $i$ to approximately $a_i|G|$ vertices, which is known as the micro-canonical version of the ground state energy. For partition problems covered by~\eqref{eq.ground_state_energy} and the micro-canonical version, Theorem~3.7 of \cite{BCLSV2} shows for a convergent graph sequence, the limits of the partition problem converge to 
their graphon analogues. Note that for $\kappa=2$,  $J=[0,1;1,0]$ and $h=(0,0)$, \eqref{eq.ground_state_energy} is precisely twice the normalised maximum cut~\eqref{MC-G}. Likewise minimum $k$-section is covered by the microcanonical version. However, the Cheeger constant and modularity are not directly covered by this energy formulation in \eqref{eq.ground_state_energy}.

In detail, note that \refL{lem:minksection-graph-step} 
can also be obtained from their Theorem 5.5, \refC{CMC3} and \refT{Lmink}
from their Theorem 3.7, and \refT{TMC3} and \refT{TMS3}  from their Equation
(4.12). While their proofs rely on fractional $q$-quotient as intermediate
steps, our results emerge naturally from the theory of coloured graphon
developed in Section \ref{sec:coloured-graphons}. We also prove
that optimisations over pure and mixed colourings on atomless Borel spaces
give the same optimal value in Theorems~\ref{TMC4} and~\ref{TMS4}; and
that this supremum is attained by a mixed colouring in Theorems \ref{TMC5}
and \ref{TMS5}.
\end{remark}

\section{Modularity scores from graphon colourings}\label{sec3:modularity}
As mentioned earlier,  modularity  is defined by a
similar partition-optimisation 
structure as 
maximum cut and minimum $k$-section 
above, but it is more delicate since  
besides the within-class edge mass, the modularity 
score includes a within-class degree-product correction. 
The precise definition for a finite graph is given in Definition~\ref{D1}. 

Our goal is to define a corresponding  modularity $\hq(W)$ for
uncoloured graphons. To do this we first fix a colouring $\Psi$ 
of $W$ and define 
(in \refD{def:mod-score} below)
a modularity score $\hq(\WW)$ for the 
coloured graphon $\WW=(W,\Psi)$.
The uncoloured modularity $\hq(W)$ is then the
supremum of these scores over all choices of $\Psi$. 
(The direct analogue of the vertex partitions for finite graphs in \refD{D1}
is pure coloured graphons, but we will allow general mixed coloured graphons
in our definitions below.)

Note that in \refD{D1}, the number of parts in the partition $\cA$ is
arbitrary, and bounded only by the number of vertices in $G$.
Hence, for the corresponding notion for graphons, we have to allow
colourings with an arbitrary finite number $\kk$ of colours. 
(Unlike the examples of maximum cut, minimum $k$-section, 
and Cheeger constant above where the
number of colours is fixed.)
Thus, in the present section (in contrast to the earlier section),
the number $\kk$ of colours is not fixed, and
we sometimes have to be careful  to include $\kk$ in the notation.

\begin{definition}[Modularity score of a coloured graphon]\label{def:mod-score} 
For a $\kk$-coloured graphon $\WW=(W,\Psi)$ with $\rho>0$, 
we define its \emph{modularity score} by
\begin{align}\label{m3}
\hskip-0.5em 
\hq(\WW)=\hq_\kappa(\WW)&:=\frac{1}{\rho}\iint_{\gO^2}\Bigpar{W(x,y)-\frac{k(x)k(y)}{\rho}}
\innprod{\Psi(x),\Psi(y)}\dd\mu(x)\dd \mu(y)
.
\end{align}
\noindent For a pure colouring $\chi$, the definition is as above, noting that the inner product in \eqref{m3} becomes
\begin{align}
\label{m3pure}
\innprod{\Psi(x),\Psi(y)}=\sum_{\ell=1}^{\kappa}\psi_\ell(x)\psi_\ell(y)=\mathbf{1}[\chi(x)=\chi(y)].
\end{align}
\end{definition}

\begin{remark}\label{RG1}
The definition \eqref{m3} is a natural extension of the definition of
modularity score for graphs with a given partition.
In fact, if $\GG=(G,\chi)$
is a coloured graph, and  $\WW'_\GG$
is the corresponding (atomic) pure  coloured graphon defined in
\refE{EG1},
then 
the modularity score of $\WW'_{\GG}$ defined in \eqref{m3} coincides with the
score \eqref{D1a} of $\GG$ 
for the partition defined by the colouring. 
\end{remark}

For convenience (to avoid having to make too many trivial exceptions below)
we define $\hq(W,\Psi):=0$ when $\rho=0$, i.e., when $W=0$ a.e.
We note also that the uncoloured case $\kk=1$ is trivial:
\eqref{m3} and \eqref{m1}--\eqref{m2} yield $\hq_1(\WW)=0$ 
for every 1-coloured graphon $\WW$. It is easy to see that
$-1 \le \hq(\WW)\le1$; we omit the proof since we prove sharper bounds in 
\refProp{LMOv2} below.

We can express the modularity score
$\hat q$ in terms of homomorphism densities.
Let again $\KK_{2;\ell\glm}$ be the coloured graph with 
one edge and two vertices coloured $\ell$ and $\glm$.

\begin{lemma}\label{LM}
For any $\kk$-coloured graphon $\WW$ with $\rho>0$,
\begin{align}\label{lm}
  \hq(\WW)=\frac{1}{\rho} \sumlkk t(\KK_{2,\ell\ell},\WW)
-\sumlkk\Bigpar{\frac{1}{\rho} \sumglkk t(\KK_{2,\ell\glm},\WW)}^2
.\end{align}
\end{lemma}
\begin{proof}
By the definition \eqref{pa1a}--\eqref{pa1b},
  \begin{align}\label{lm1}
        t(\KK_{2,\ell\glm},\WW)
=
\iint_{\gO^2}W(x,y)\psi_\ell(x)\psi_\glm(y)\dd\mu(x)\dd\mu(y).
  \end{align}
Hence, recalling $\sumlkk\psi_\ell(x)\psi_\ell(y)=\innprod{\Psi(x),\Psi(y)}$
and $\sumlkk\psi_\ell(y)=1$,
we have
\begin{align}\label{lm2}
  \sumlkk t(\KK_{2,\ell\ell},\WW)&
=\iint_{\gO^2}W(x,y)\innprod{\Psi(x),\Psi(y)}\dd\mu(x)\dd\mu(y),
\\\label{lm3}
  \sumglkk t(\KK_{2,\ell\glm},\WW)&
=\iint_{\gO^2}W(x,y)\psi_\ell(x)\dd\mu(x)\dd\mu(y)
=\int_{\gO}k(x)\psi_\ell(x)\dd\mu(x),
\\\label{lm4}
\sumlkk \sumglkk t(\KK_{2,\ell\glm},\WW)&
=\int_{\gO}k(x)\dd\mu(x)
=\rho.
\end{align}
The definition \eqref{m3} can be written
\begin{align}\label{lm5}
  \hq(\WW)&
=\frac{1}{\rho}\iint_{\gO^2}W(x,y)\innprod{\Psi(x),\Psi(y)}\dd\mu(x)\dd \mu(y)
\notag\\&\hskip6em
-
\frac{1}{\rho^2}\sumlkk\iint_{\gO^2}k(x)k(y)\psi_\ell(x)\psi_\ell(y)\dd\mu(x)\dd
\mu(y)
\notag\\&
=\frac{1}{\rho}\iint_{\gO^2}W(x,y)\innprod{\Psi(x),\Psi(y)}\dd\mu(x)\dd \mu(y)
\notag\\&\hskip6em
-
\frac{1}{\rho^2}\sumlkk\Bigpar{\int_{\gO}k(x)\psi_\ell(x)\dd\mu(x)}^2
\end{align}
and \eqref{lm} follows by \eqref{lm2}--\eqref{lm3}.
\end{proof}

\begin{proposition}\label{LMOv2}
For any $\kk$-coloured graphon $\WW=(W,\Psi)$,
\begin{align}
  \label{lm0}
-\tfrac{1}{2} \le \hq(\WW)\le1-\frac{1}{\kappa}.
\end{align}
\begin{romenumerate}
\item\label{LMOa}
Furthermore, $\hq(\WW)=-\tfrac{1}{2}$ 
if and only if
then there exists $\lambda, \lambda'\in[\kappa]$ with $\lambda \neq
\lambda'$ 
and a tripartition $\gO=A_1\cup A_2\cup B$ into disjoint sets such that
$\psi_{\gl}=1$ on $A_1$, $\psi_{\gl'}=1$ on $A_2$, and
$W(x,y)=0$ a.e.\ on 
$\gO^2\setminus \bigpar{(A_1\times A_2) \cup (A_1\times A_2)}$. 

\item \label{LMOb}
Similarly, $\hq(\WW)=1-\frac{1}{\kappa}$ 
if and only if
there exists a partition $\gO=\bigcup_{\ell=1}^\kk A_\ell \cup B$
such that 
$W=0$ a.e.\ on $\gO^2\setminus \bigcup_{\ell=1}^\kk (A_\ell\times A_\ell)$,
$\int_{A_\ell}k_W(x)=c$ for some $c>0$, and
$\psi_\ell=1$ on $A_\ell$, for every $\ell\in[\kk]$. 
\end{romenumerate}
\end{proposition}
Thus, ignoring null sets and possibly a set $B$ such that $W=0$ on $B\times\gO$,
$\hq(\WW)$ is minimal if it is a pure coloured bipartite graphon, with one
colour on each part,
and $\hq(\WW)$ is maximal if it is a pure coloured graphon, 
with equal edge density for monochromatic edges of each colour and
density 0 for edges that are not monochromatic.
(The latter case can be expressed by saying that $W$ is the sum of some graphons
$W_1,\dots,W_\ell$ in the sense of \cite{SJ213}, 
with all $W_\ell$ having the same volume, 
and  we  give $W_\ell$ colour $\ell$.)
\begin{proof}
We may assume $\rho>0$.
We use again the notation 
\begin{align}
   t_{\ell, \lambda} & :=  t(\KK_{2,\ell\lambda},\WW)
\end{align}
from \eqref{eq.tab_defn},
and we further define
\begin{align}\label{eq.wts_non_neg-2}
f_\ell:=  
 t_{\ell,\ell} \rho
-\Bigpar{ \sumgllkk  t_{\ell,\gl}}^2 
+ \frac{\rho}{2}   \sumgllkk t_{\ell,\gl}
.\end{align}
Then Lemma~\ref{LM} implies, using \eqref{lm4}, that
\begin{align}
  \label{ele1}
\sumlkk f_\ell = \rho^2\bigpar{\hq(\WW)+\tfrac12}.
\end{align}
For the lower bound, it is thus
sufficient to show $f_\ell \geq 0$ for each $\ell$. 
We define
the following short-hand for sums of densities. (These will
correspond to the densities of edges with 
exactly one of colour $\ell$, and neither of colour $\ell$.)
\begin{align}
 t_{\ell, \lnot \ell} 
 := \sum_{\lambda\neq \ell} t_{\ell,\gl}
\quad \mbox{and} \quad t_{\lnot \ell, \lnot \ell} 
:= 
\sum_{\lambda \neq \ell} \sum_{\lambda' \neq \ell} t_{\gl,\gl'}.
\end{align}
We may rewrite $f_\ell$ in this new notation,
\begin{align}\label{eq.wts_non_neg_v2}
f_\ell &= t_{\ell, \ell}\rho 
-\bigpar{t_{\ell, \ell} + t_{\ell, \lnot \ell}}^2 + \tfrac{1}{2}\bigpar{t_{\ell, \ell} + t_{\ell, \lnot \ell}}\rho
\notag\\&
= t_{\ell, \ell}\, \rho  + 
\bigpar{t_{\ell, \ell} + t_{\ell, \lnot \ell}}\bigpar{\tfrac{1}{2}\rho- (t_{\ell, \ell} + t_{\ell, \lnot \ell})}.
\end{align}
By~\eqref{lm1}, since $\sum_{\gl\neq \ell} \psi_\gl =(1-\psi_\ell)$, 
\begin{align}\label{mag1}
t_{\ell, \lnot \ell } 
=
\sum_{\lambda \neq \ell}\iint_{\gO^2}W(x,y)\psi_\ell(x)(1-\psi_\ell(y))\dd\mu(x)\dd\mu(y),
\end{align}
and similarly,
\begin{align}\label{mag2}
t_{\lnot \ell, \lnot \ell } 
=
\iint_{\gO^2}W(x,y)(1-\psi_\ell(x))(1-\psi_\ell(y))\dd\mu(x)\dd\mu(y).  
\end{align}
Thus 
\begin{align}\label{mag3}
\rho = t_{\ell, \ell} + 2
t_{\ell, \lnot \ell} + t_{\lnot \ell, \lnot \ell} 
.\end{align}
Substituting this into~\eqref{eq.wts_non_neg_v2} we obtain
\begin{align}
f_\ell&= t_{\ell, \ell}\, \bigpar{ t_{\ell, \ell} + 2
t_{\ell, \lnot \ell} + t_{\lnot \ell, \lnot \ell} }  + 
\tfrac{1}{2}\bigpar{t_{\ell, \ell} + t_{\ell, \lnot \ell}}\bigpar{ t_{\lnot \ell, \lnot \ell} - t_{\ell, \ell}} 
\notag\\&
\label{eq.wts_non_neg_v4}
=\tfrac{1}{2}t_{\ell, \ell}^2+ 
\tfrac{3}{2}t_{\ell, \ell} t_{\ell, \lnot \ell} 
+ \tfrac{3}{2}t_{\ell,\ell} t_{\lnot \ell, \lnot \ell} + \tfrac{1}{2}t_{\ell, \lnot \ell}t_{\lnot \ell, \lnot \ell} 
.\end{align}
Hence $f_\ell\geq 0$, since it is a linear combination of non-negative densities, all with positive coefficients; this establishes the lower bound.

We now consider the case that $\hq(\WW)=-1/2$.
By \eqref{ele1}, $\hq(\WW)=-1/2$ implies $f_\ell=0$ for each
$\ell\in [\kappa]$. Thus by~\eqref{eq.wts_non_neg_v4}, for each $\ell\in
[\kappa]$,  
\begin{equation}\label{eq.must_be_zero} t_{\ell, \ell}=0 \quad \mbox{and} \quad t_{\ell, \lnot \ell}\, t_{\lnot \ell, \lnot \ell}=0.\end{equation}
Since $t_{\ell,\ell}=0$ for each $\ell$, 
i.e., the density of all monochromatic edges is $0$, 
we thus have that $\rho=2\sum_{\ell<\ell'} t_{\ell,\ell'}$. Recall that by
assumption $\rho>0$ and hence there must be some pair $\lambda\neq \lambda'$ such
that $t_{\lambda, \lambda'}>0$.
We will now argue that 
\begin{align}\label{mag4}
t_{j,j'}=0
\qquad\text{for any $\{j,j'\}\neq \{\lambda,\lambda' \}$}. 
\end{align}
By symmetry we may assume that $\lambda\notin \set{j,j'}$.
Then $t_{j,j'} \le t_{\lnot\gl,\lnot\gl}$. Furthermore,
$t_{\gl,\lnot\gl}\ge t_{\gl,\gl'}>0$, and thus \eqref{eq.must_be_zero} 
yields $t_{\lnot\gl,\lnot\gl}=0$, whence \eqref{mag4} follows.

For notational convenience, assume $(\gl,\gl')=(1,2)$.
Then, \eqref{mag4} and \eqref{lm1} yield
\begin{align}\label{mag5}
  \iint_{\gO^2}W(x,y)\psi_j(x)\psi_{j'}(y)\dd\mu(x)\dd\mu(y)=0,
\qquad \set{j,j'}\neq\set{1,2}.
\end{align}
In other words, a.e.\ on the set $E:=\set{(x,y):W(x,y)>0}\subset\gO^2$
we have
\begin{align}\label{mag6}
  \psi_j(x)\psi_{j'}(y)=0,
\qquad \set{j,j'}\neq\set{1,2}.  
\end{align}
Let $E'$ be the subset of $E$ where \eqref{mag6} holds 
(for all $\set{j,j'}\neq\set{1,2}$); 
thus $E$ and $E'$ differ only by a null set.
Let $A_j:=\set{x\in\gO:\psi_j(x)=1}$, and $B:=\gO\setminus(A_1\cup A_2)$.

Assume $(x,y)\in E_1$.
If $j\ge3$, then  \eqref{mag6} holds for all $j'\in[\kk]$, and by summing
over all $j'$ we obtain $\psi_j(x)=0$. Hence, $\psi_1(x)+\psi_2(x)=1$.
Similarly, $\psi_1(y)+\psi_2(y)=1$.
Suppose now that $\psi_1(x)>0$. Then \eqref{mag6} with $j=j'=1$
implies $\psi_1(y)=0$ and thus $\psi_2(y)=1$; then \eqref{mag6} with $j=j'=2$
implies $\psi_2(x)=0$ and thus $\psi_1(x)=1$.
Hence, if $(x,y)\in E'$, then $\psi_1(x)\in\set{0,1}$,
and by symmetry also $\psi_1(y)\in\set{0,1}$.
Consequently, $E'\subseteq (A_1\times A_2)\cup (A_2\times A_1)$,
and thus $W$ and $\Psi$ have the properties claimed in \ref{LMOa}.

Conversely, if these properties hold, then \eqref{mag4} holds, 
and it follows that $t_{\ell,\ell}=0$ for every $\ell$,
$t_{\ell,\lnot\ell}=0$ for $\ell\notin\set{1,2}$, and 
$t_{\lnot\ell,\lnot\ell}=0$ for $\ell\in\set{1,2}$.
Hence \eqref{eq.wts_non_neg_v4}  shows that $f_\ell=0$ for every $\ell$, and
thus, by \eqref{ele1}, $\hq(\WW)=-\frac12$.

To see the upper bound in \eqref{lm0}, 
note that by \eqref{lm} it is equivalent to prove that
\begin{align}\label{ele3}
\rho\sum_\ell t_{\ell, \ell} - \sum_\ell(t_{\ell, \ell} + t_{\ell, \lnot \ell})^2  
\leq \bigpar{1-\frac{1}{\kappa} }\rho^2. 
\end{align}
Since $\sum_\ell t_{\ell, \ell}\leq \rho$ by \eqref{lm4},
it is thus enough to show that $\sum_\ell (t_{\ell, \ell} + t_{\ell, \lnot
  \ell})^2 \geq \rho^2 / \kappa$. However, by Cauchy--Schwarz,
\begin{equation}\label{eq.at_most_one_on_k_CS}
\sum_{\ell} (t_{\ell, \ell} + t_{\ell, \lnot \ell})^2 \geq \frac{1}{\kappa} \Bigpar{\sum_{\ell} (t_{\ell, \ell} + t_{\ell, \lnot \ell})}^2 = \frac{1}{\kappa} \rho^2, 
\end{equation}
and hence we have shown $\q(\WW)\leq 1-\frac{1}{\kappa}$.

Lastly, we consider the case that $\hq(\WW)=1-\frac{1}\kk$. 
By the calculations above, this holds if and only if 
$\sum_\ell t_{\ell, \ell}=\rho$ and equality holds 
in~\eqref{eq.at_most_one_on_k_CS}.
The first condition implies that $t_{\ell, \ell'}=0$ for all $\ell\neq\ell'$. 
Assuming this, we have $t_{\ell,\lnot\ell}=0$, and thus the Cauchy--Schwarz
inequality in \eqref{eq.at_most_one_on_k_CS} then 
reduces to  
\begin{align}
  \kk\sumlkk t_{\ell,\ell}^2 = \Bigpar{\sumlkk t_{\ell,\ell}}^2 = \rho^2,
\end{align}
which holds if and only if  $t_{\ell, \ell}=\rho/\kappa$ for each $\ell$.
Consequently, $\hq(\WW)=1-\frac1\kk$ if and only if, for some $c>0$,
\begin{align}
  \label{ele2}
t_{\ell,\ell'}=
  \begin{cases}
    c, & \ell=\ell',\\
    0, & \ell\neq\ell'.
  \end{cases}
\end{align}
Assume \eqref{ele2} and let, as above, $E:=\set{(x,y):W(x,y)>0}$. 
Then \eqref{ele2} implies
\begin{align}\label{mag7}
  \iint_{\gO^2}W(x,y)\psi_\ell(x)\psi_{\ell'}(y)\dd\mu(x)\dd\mu(y)
=t_{\ell,\ell'}
=0,
\qquad \ell\neq\ell',
\end{align}
and thus, for a.e.\  $(x,y)\in E$,
\begin{align}\label{mag8}
\psi_\ell(x)\psi_{\ell'}(y)=0,
\qquad \ell\neq\ell',
\end{align}
Let $E'$ be the subset of $E$ where \eqref{mag8} holds,
and define $A_\ell:=\set{x\in\gO:\psi_\ell(x)=1}$ 
($\ell\in[\kk]$) and $B:=\gO\setminus\bigcup_{\ell=1}^\kk A_\ell$.
Assume $(x,y)\in E'$. If $\psi_\ell(x)>0$ for some $\ell$, then 
\eqref{mag8} shows that $\psi_{\ell'}(y)=0$ for all $\ell'\neq\ell$, and
thus $\psi_\ell(y)=1$. By symmetry this implies $\psi_\ell(x)=1$.
Consequently, $E'\subseteq\bigcup_{\ell=1}^\kk A_\ell\times A_\ell$,
and thus 
$W=0$ a.e.\ on $\gO^2\setminus \bigcup_{\ell=1}^\kk (A_\ell\times A_\ell)$.
This implies also, using \eqref{tab2} and \eqref{m2},
\begin{align}\label{mag9}
t_{\ell,\ell}=\int_{A_{\ell}\times A_{\ell}}W(x,y)\dd\mu(x)\dd\mu(y)
=\int_{A_{\ell}\times \gO}W(x,y)\dd\mu(x)\dd\mu(y)
=\int_{A_{\ell}}k_W(x)\dd\mu(x).
\end{align}
Since $t_{\ell,\ell}=c$ by \eqref{ele2}, this yields
$\int_{A_{\ell}}k_W(x)\dd\mu(x)=c$, and thus $W$ and $\Psi$ satisfy the 
properties in \ref{LMOb}.

Conversely, if these properties hold, then \eqref{mag9} and \eqref{ele2}
hold, which implies that we have equality in \eqref{ele3} and thus
$\hq(\WW)=1-\frac1\kk$.
\end{proof}

\begin{theorem}
Fix $\kk\ge1$.
  If\/ $\WW_1$ and $\WW_2$ are two equivalent $\kk$-coloured graphons, then
$\hq(\WW_1)=\hq(\WW_2)$. In other words, $\hq$ is well-defined on 
coloured graph limits.
Furthermore, the functional $\WW\mapsto\hq_\kk(\WW)$ is continuous at every 
$\kk$-coloured graph limit $\WW$ with $\rho_{\WW}>0$.
\end{theorem}

\begin{proof}
This follows from \eqref{lm} and \eqref{lm4},
recalling that for each coloured graph $\FF$
we have $t(\FF,\WW_1)=t(\FF,\WW_2)$ when  $\WW_1\cong\WW_2$
(\refD{D=}),
and 
$t(\FF,\WW_n)=t(\FF,\WW)$ when  $\WW_n\to\WW$
(\refC{CTWBb}).
\end{proof}

We define modularity for uncoloured graphons by optimising the coloured
functional $\hq(\WW)$ over all colourings of the underlying graphon W. This
gives the graphon analogue of graph modularity, which optimises the same
score over all vertex partitions. As in \refS{SSSmax}, we consider two versions, one with only pure colourings
and one with general mixed colourings, and again we regard the mixed
coloured version as the standard one; we therefore may omit the subscript
${}_\mix$ from the notations.

\begin{definition}[Graphon modularity]\label{def:graphon-modularity}
Let $W$ be an uncoloured graphon, defined on $\OFmu$.
For each $\kappa\ge1$, define the \emph{(mixed) $\kappa$-modularity} of $W$ by
 \begin{align}\label{mod2}
\hqx_\kk(W)
=\hqx_{\kk,\mix}(W)
:=\sup_{\Psi\in\PPsi_\kk}\hq(W,\Psi),
\end{align}
where $\Lambda_\kappa$ is the set of mixed $\kappa$-colourings of
$\Omega$. Define the \emph{(mixed) graphon modularity} by 
\begin{align}\label{mod3}
\hqx(W)=  
\qhatstarmix(W):=\sup_{\kappa\ge1}\hqx_\kk(W).
\end{align}
Similarly, define the \emph{pure $\kappa$-modularity} by
\begin{align}\label{mod4}
    \hqx_{\kappa,\mathrm{pure}}(W)
  :=\sup_{\chi:\Omega\to[\kappa]}
  \qhat(W,\chi),
\end{align}
and the \emph{pure graphon modularity} by
\begin{align}\label{de:gmodp}
  \qhatstarpur(W)
  :=
  \sup_{\kappa\ge1}
  \hqx_{\kappa,\mathrm{pure}}(W).
\end{align}
\end{definition}
Note that $\hat q(W,\chi)$, the modularity score associated with the measurable partition induced by $\chi$, agrees with the graphon modularity score in \cite{Klimm}. In particular,
 we recover the graphon modularity defined in \cite{Klimm} from the pure graphon modularity $\qhatstarpur(W)$ as given in \eqref{de:gmodp}.

\begin{remark}\label{RM2}
We will see in \refT{LX1} that for graphons defined on $\oi$,
$\hqx_{\kk,\pure}(W)=\hqx_{\kk,\mix}(W)$
and thus
$\hqx_{\pure}(W)=\hqx_{\mix}(W)$, so pure and mixed colourings give the same 
result.
However, we will also see in \refE{ex:C5} that
for graphons defined on a general probability space, the supremum over pure
colourings may be strictly smaller than the one over mixed colourings.
We will further see in \refT{TLC2} that for our definition with mixed
colourings, two equivalent graphons have the same modularity, while 
(by \refE{ex:C5} again) this is
not true for the suprema over pure colourings only.
Hence, as for maximum cut in \refSS{SSSmax}, 
the definition above using mixed colourings seems to be the natural one
for general graphons.

Klimm, Jones, and  Schaub
\cite{Klimm} consider, as mentioned above, only pure colourings (in our
terminology) of graphons (and graphs).
However, they consider only graphons defined on $\oi$; hence their
supremum of graphon modularity over all (pure) colourings equals our 
modularity $\hqx$ by \refT{LX1}.
\end{remark}

It follows, see \refR{RG1}, 
that for a finite graph $G$, the modularity $\hqx(G)$ defined in \eqref{D1b}
equals $\hqx_\pure(W'_G)$. We study the relations between modularity for
graphs and graphons further in \refSS{SSMGG}, but we first give
some results for graphons, beginning with
some immediate consequences of the definition.
\begin{lemma}\label{LF1}
For any uncoloured graphon $W$, 
the $\kk$-modularity $\hqx_\kk(W)$ is a non-decreasing function of $\kk$:
\begin{align}\label{mod30}
  0=\hqx_1(W)\le\hqx_2(W)\le\dots,
\end{align}
and thus
\begin{align}
  \hqx_\kk(W)\to\hqx(W)
\qquad\text{as $\kk\to\infty$}.
\end{align}
Furthermore, for every $\kk\ge1$,
\begin{align}
  \label{lf1a}
0 &\le \hqx_\kk(W)
\le1-\tfrac1\kk<1,
\intertext{and thus}
  \label{lf1b}
0 &\le \hqx(W)\le1.
\end{align}
\end{lemma}
\begin{proof}
Any $\kk$-colouring may be regarded as an $(\kk+1)$-colouring (not using the
last colour), and thus $\hqx_\kk(W)\le\hqx_{\kk+1}(W)$, which gives the
inequalities in \eqref{mod30}.
Furthermore,
when $\kk=1$, the only colouring $\Psi$ is the monochromatic colouring with
$\psi_1(x)=1$ for all $x$, which  
as noted above 
has $\hq_1(W,\Psi)=0$, as a simple consequence of 
\eqref{m3} and \eqref{m1}--\eqref{m2}. Hence $\hqx_1(W)=0$.

The upper bound in \eqref{lf1a} follows from \eqref{lm0} and \eqref{mod2},
and the lower bound follows from \eqref{mod30}.
Finally, \eqref{lf1b} follows from \eqref{lf1a} and \eqref{mod3}. 
\end{proof}

The following results are similar to results for graph partition parameters in
\refS{sec:graphon-partition}.

\begin{theorem}\label{LX1}
  Let $W$ be an uncoloured graphon defined on $\oi$ (equipped with Lebesgue
  measure $\gl$, as   usual). Then, 
$\hqx_{\kk,\pure}(W)=\hqx_{\kk,\mix}(W)$
for every $\kk\ge1$, 
and consequently\/
$\hqx_{\pure}(W)=\hqx_{\mix}(W)$.
\end{theorem}

\begin{proof}
This is similar to the proof of \refT{TMC4}.
We may assume that $\rho_W>0$.
Suppose first that $W$ is a step graphon in the sense that,
for some $N\ge1$, $W$ is 
constant on every square $I_{i}\times I_{j}$, $i,j\in[N]$,
where $I_{i}:=[(i-1)/N,i/N)$. 
Given a mixed $\kk$-colouring $\Psi$ of $W$, we can find a pure 
$\kk$-colouring $\chi$
such that for every $i\in[N]$ and $j\in[\kk]$, we have, as in \eqref{tmc4},
\begin{align}\label{lx1}
  \gl\set{x\in I_{i}:\chi(x)=j} = \int_{I_{i}}\psi_j(x)\dd x.
\end{align}
Since $W(x,y)$ and $k(x)k(y)$ are constant on each square 
$I_{i}\times I_{j}$,
it follows from \eqref{lx1} (by integrating over each such square separately)
that the last integral in \eqref{m3} is the same for 
the colourings $\Psi$ and $\chi$, and thus
\begin{align}\label{lx1b}
\hq_\kk(W,\Psi)=\hq_\kk(W,\chi).
\end{align}

Consider now a general graphon $W$
and a $\kk$-colouring $\Psi$ of it.
Let $N\ge1$ and let $W_N$ be the step graphon 
obtained by averaging $W$ over each square $I_{i}\times I_{j}$,
with $I_i$ as above.
Let $k_N(x):=\int_\gO W_N(x,y)\dd\mu(y)$ be the corresponding degree
function,
and let $\chi_N$ be the 
pure $\kk$-colouring constructed by \eqref{lx1}.
Then, using \eqref{m3} and \eqref{m2}, and noting that $\rho_{W_N}=\rho_W$,
\begin{align}\label{lx3}
&\bigabs{\hq_\kk(W,\Psi)-\hq_\kk(W_N,\Psi)}
\notag\\&
\le  \frac{1}{\rho_W}\iint_{\gO^2}
\Bigabs{W(x,y)-W_N(x,y)-\frac{k(x)k(y)-k_N(x)k_N(y)}{\rho_W}}\dd\mu(x)\dd\mu(y)
\notag\\&
\le  \frac{1}{\rho_W}\iint_{\gO^2}|W(x,y)-W_N(x,y)|\dd\mu(x)\dd\mu(y)
\notag\\&\qquad
+  \frac{1}{\rho_W}\iint_{\gO^2}
\frac{|k(x)-k_N(x)|\,k(y)+k_N(x)\,|k(y)-k_N(y)|}{\rho_W}
\dd\mu(x)\dd\mu(y)
\notag\\&
= \rho_W\qw\norm{W-W_N}_{L^1(\gO^2)} + 2\rho_W\qw\norm{k-k_N}_{L^1(\gO)},
\end{align}
and similarly for $\bigabs{\hq(W,\chi_N)-\hq(W_N,\chi_N)}$.
It is well-known that $\norm{W-W_N}_{L^1(\gO^2)}\to0$ as \Ntoo, and it
follows from \eqref{m1} that 
$\norm{k-k_N}_{L^1(\gO)}\le\norm{W-W_N}_{L^1(\gO^2)}$.
Since we have shown that
$\hqx(W_N,\Psi)=\hqx(W_N,\chi_N)$, it follows that
\begin{align}\label{lx4}
\bigabs{\hq_\kk(W,\Psi)-\hq_\kk(W,\chi_N)}
&\le \bigabs{\hq_\kk(W,\Psi)-\hq_\kk(W_N,\Psi)} + \bigabs{\hq_\kk(W,\chi_N)-\hq_\kk(W_N,\chi_N)}
\notag\\&
\le 2\rho_W\qw\norm{W-W_N}_{L^1(\gO^2)} + 4\rho_W\qw\norm{k-k_N}_{L^1(\gO)}
\to0
\end{align}
as \Ntoo. Consequently, the supremum of $\hq_\kk(W,\chi)$ over pure
$\kk$-colourings $\chi$ is at least $\hq_\kk(W,\Psi)$ for every mixed
$\kk$-colouring 
$\Psi$, and thus the suprema over pure and mixed colourings are the same.
\end{proof}

By \refR{R01+}, \refT{LX1} holds more generally for any probability space
that is a Borel space with an atomless
probability measure.

\begin{lemma}\label{LC1}
  Let $\gf:\gO_1\to\gO_2$ be a \mpp{} map.
If\/ $W$ is an uncoloured graphon on $\gO_2$, then
$\hqx_{\kkpure}(W)\le\hqx_\kkpure(W^\gf)$ 
and $\hqx_\kkmix(W)=\hqx_\kkmix(W^\gf)$ for every $\kk$,
and consequently
\begin{align}\label{lc1pure}
\hqx_\pure(W)&\le\hqx_\pure(W^\gf),\\
\hqx_\mix(W)&=\hqx_\mix(W^\gf), \label{lc1mix}
\end{align}
where in general there might be strict inequality in \eqref{lc1pure}.
\end{lemma}
\begin{proof}
This is similar to the proof of \refL{LMC}. Note first that \eqref{m1} implies, since $\gf$ is \mpp,
for $x\in\gO_1$,
\begin{align}\label{vq1}
  k_{W^\gf}(x)=\int_{\gO_1}W\bigpar{\gf(x),\gf(y)}\dd\mu_1(y)
=\int_{\gO_2}W\bigpar{\gf(x),z}\dd\mu_2(z)
=k_W^\gf(x).
\end{align}
Similarly, $\rho_{W^\gf}=\rho_W$.

We may assume $\rho=\rho_W>0$.
It then follows from \eqref{m3} and \eqref{vq1}, 
by a change of variables similar to \eqref{vq1}, 
that for every $\kk\ge1$ and 
$\Psi\in \PPsi_\kk(\gO_2)$,
  \begin{align}\label{vq2}
  \hq(W^\gf,\Psi^\gf)&
=\frac{1}{\rho}\iint_{\gO_1^2}\Bigpar{W^\gf(x,y)-\frac{k_W^\gf(x)k_W^\gf(y)}{\rho}}
\innprod{\Psi^\gf(x),\Psi^\gf(y)}\dd\mu_1(x)\dd \mu_1(y)
\notag\\&
=\frac{1}{\rho}\iint_{\gO_2^2}\Bigpar{W(x',y')-\frac{k_W(x')k_W(y')}{\rho}}
\innprod{\Psi(x'),\Psi(y')}\dd\mu_2(x')\dd \mu_2(y')
\notag\\&
=\hq(W,\Psi).
  \end{align}

If $\Psi$ is a pure colouring, then so is $\Psi^\gf$, and thus it follows by
taking the supremum over pure $\Psi$ that 
$\hqx_{\kkpure}(W)\le\hqx_\kkpure(W^\gf)$, which implies \eqref{lc1pure}. 
We will see in \refE{ex:C5} that the inequality might be strict.

Consider now general mixed colourings.
Then \eqref{vq2} implies, by \eqref{mod2}
\begin{align}\label{vq3}
  \hqx_\kk(W)=\sup_{\Psi\in\PPsi_\kk(\gO_2)}\hq_\kk(W^\gf,\Psi^\gf).
\end{align}
Hence, again using \eqref{mod2},
we want to show that the restricted supremum over $\Psi^\gf$
on the \rhs{} of \eqref{vq3}
equals the supremum over all $\Psi\in\PPsi_\kk(\gO_1)$.

Thus,  consider any $\Psi\in\PPsi_\kk(\gO_1)$.
To ease notation, 
let
\begin{align}\label{vq4}
  H(x,y):=W(x,y)-\frac{k_{W}(x)k_{W}(y)}{\rho}.
\end{align}
Then \eqref{m3} and \eqref{vq1} yield, similarly to \eqref{vq2},
\begin{align}\label{vq5}
  \hq(W^\gf,\Psi)&=\frac{1}\rho\iint_{\gO_1^2}H\bigpar{\gf(x),\gf(y)}
  \innprod{\Psi(x),\Psi(y)}\dd\mu_1(x)\dd\mu_1(y).
\end{align}
Let $\cG:=\set{\gf\qw(A):A\in\cF_2}$; this is a sub-$\gs$-field of $\cF_1$,
and a (real-valued, say)
function on $\gO_1$ is $\cG$-measurable if and only if it can be
written as $g^\gf$ for some $\cF_2$-measurable function $g$.
(See e.g.\ \cite[Lemma 1.13]{Kallenberg}.)
Let $\Psi':=\E(\Psi\mid\cG)$ (defined by taking the conditional expectation
$\E(\psi_\ell\mid\cG)$ of each component separately);
then $\Psi'\in\PPsi_\kk(\gO_1)$. Furthermore, $\Psi'$ is $\cG$-measurable,
and thus $\Psi'=\gU^\gf$ for some function $\gU$ on $\gO_2$; it is easy to
see that (after adjustment on a null set) $\gU\in\PPsi_\kk(\gO_2)$.
Consequently, since $H\bigpar{\gf(x),\gf(y)}$ is $(\cG\times\cG)$-measurable,
\eqref{vq5} implies
\begin{align}\label{vq6}
  \hq(W^\gf,\Psi)&=\frac{1}\rho\iint_{\gO_1^2}H\bigpar{\gf(x),\gf(y)}
  \innprod{\Psi'(x),\Psi'(y)}\dd\mu_1(x)\dd\mu_1(y).
\end{align}
Using \eqref{vq5} again, we thus obtain
\begin{align}
  \label{vq7}
  \hq(W^\gf,\Psi)&=
  \hq(W^\gf,\Psi')
=
  \hq(W^\gf,\gU^\gf)
.\end{align}
Hence \eqref{vq3} shows that
$\hq(W^\gf,\Psi)\le \hqx_\kk(W)$. Since $\Psi\in\PPsi_\kk(\gO_1)$ is arbitrary, 
this shows $\hqx_\kk(W^\gf)\le\hqx_\kk(W)$ by \eqref{mod3}.
The converse inequality follows directly from \eqref{vq3} and \eqref{mod3}.
Consequently we have
$\hqx_\kk(W^\gf)=\hqx_\kk(W)$ for every $\kk$, and thus also \eqref{lc1mix}
holds.  
\end{proof}

\begin{remark}
  Note that in the proof above, even if $\Psi$ is a pure coloured graphon, then in general
  $\Psi'$ and $\gU$ are not. 
This explains why the proof does not yield equality also in \eqref{lc1pure}.
\end{remark}

\begin{theorem}\label{TLC2}
  If\/ $W_1$ and $W_2$ are two equivalent uncoloured graphons, 
then $\hqx_\kk(W_1)=\hqx_\kk(W_2)$ for every $\kk\ge1$
and thus $\hqx(W_1)=\hqx(W_2)$.
Consequently, $\hqx_\kk(\hW)$ 
and $\hqx(\hW)$
are well-defined for graph limits $\hW$ (i.e., equivalence classes  of graphons).
\end{theorem}
\begin{proof}
This follows from \refL{LC1} by exactly the same argument as in the proof of
\refT{TMC2}. 
\end{proof}

Note that \refE{ex:C5} below shows that \refT{TLC2} does not hold 
for $\hqx_\pure(W)$. 
This is the main reason for our choice of 
mixed colourings in the definition \eqref{mod2}--\eqref{mod3} of $\hqx(W)$.
We note also the following.

\begin{theorem}\label{TLC3}
  For every uncoloured graphon $W$ and $\kk\ge1$, 
the supremum over $\Psi\in\PPsi_\kk$ in \eqref{mod2} is
  attained. (So the supremum is a maximum.) 
 However, there exists an uncoloured graphon $W$ such that the supremum over $\kappa$ in~\eqref{mod3} is not attained.
\end{theorem}

\begin{proof}
We may assume $\rho_W>0$. Then
\refLs{LM} and \ref{LA2} (or \ref{LA3}) show that 
$\Psi\mapsto\hq(W,\Psi)$ is a continuous function on the compact space
$\PPsi_\kk$, and thus it attains its maximum.

An example where the supremum over all $\kk$ is not attained
is given in \refE{ex:inf_colours_needed} below.
\end{proof}

\begin{remark}If 
$W$ is a graphon on $\oi$, so $\hqx_\kkpure(W)=\hqx_\kk(W)$ by
\refT{LX1}, we do not know whether there always is a pure colouring that
attains the supremum in \eqref{mod4} (and thus \eqref{mod2}).  
\end{remark}

\begin{figure}
   \centering

    \includegraphics[scale=0.35]{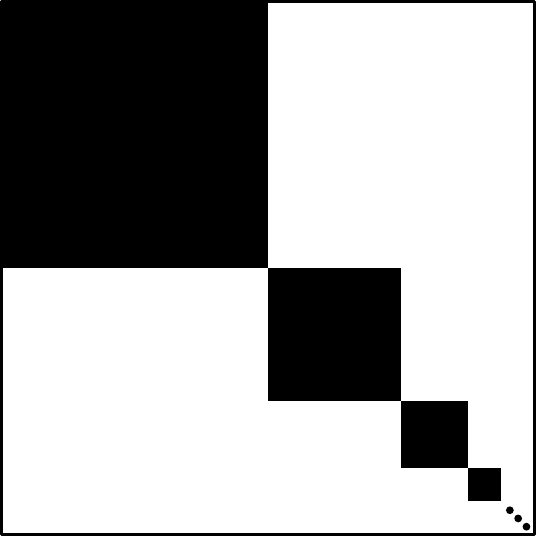}
    \caption{{\small A graphon whose modularity value is not achieved using any finite number of colours. See \refT{TLC3} and Example~\ref{ex:inf_colours_needed}.}}
    \label{fig:inf_colours_needed} 
\end{figure}

\begin{example}\label{ex:inf_colours_needed}
Consider the graphon in Figure~\ref{fig:inf_colours_needed} corresponding to the limit of disjoint unions of cliques on $n/2, n/4, \ldots$ many vertices. We may define it $W(x,y):=1$ if $\exists j\in \mathbb{N}_0$ such that $2^{-j} < x, y \leq 2^{-j+1}$, and $W(x,y):=0$ otherwise. Label these intervals by $I_j:=(2^{-j}, 2^{-j+1}]$ and note for $x\in I_j$ we have,
\begin{equation}\label{eq.inf_basics} k(x) = \int_{I_j} W(x,y) \dd y = 2^{-j} \quad \mbox{and} \quad \rho=\sum_{j=0}^\infty \int_{I_j^2} W(x,y) \dd x \dd y = 1/3 .\end{equation}
We \textbf{claim} that the modularity value of this graphon is not achieved by any finite number of colours. To prove the claim, assume for a contradiction that the maximum modularity value may be achieved by a colouring $\Psi$ using colours $\{1, \ldots, \kappa\} =[\kappa]$.
Since there are finitely many colours, but infinitely many intervals $I_j$,
there must be either (i) some interval $I_h$ such that at least two colours
have non-zero measure on that interval, or (ii) there are two intervals
$I_{h}$ and $I_{h_1}$ and a colour $\ell^*$ such that $\ell^*$ is the only colour with non-zero measure on both intervals. 

In both cases, define a colouring $\Psi'$ on $\kappa+1$ colours by re-colouring the
interval $I_h$ with the `new colour' $\kappa+1$, i.e. 
\begin{equation}\label{eq.phi_dash_col}
\Psi'(x) = 
    \begin{cases}
    (0, \ldots, 0, 1)                            & \mbox{if } x \in I_h,\\
    (\psi_1(x), \ldots, \psi_\kappa(x), 0)  & \mbox{otherwise}.
    \end{cases}
\end{equation}
(Note that if $\Psi$ is a pure colouring then so is $\Psi'$.) We shall show
$\hat{q}(W, \Psi')>\hat{q}(W, \Psi)$,
which yields the desired contradiction.
Define also, for convenience, $\psi_{\kk+1}(x):=0$. 
Let
\begin{equation}\label{eq.defn_x} 
x_\ell:=\int_{I_h}\psi_\ell(x) dx \ge0, 
\qquad \ell \in[\kk+1],
\end{equation}   
and note that
\begin{equation}\label{mar0}
\sum_\ell x_\ell = \gl(I_h)=2^{-h}.
\end{equation}
Write $\WW'=(W, \Psi')$, $\WW=(W, \Psi)$,  
\begin{align}t'_{\ell,\gl}=t(\KK_{2,\ell\gl},\WW') \quad \mbox{and}\quad t_{\ell,\gl}=t(\KK_{2,\ell\gl},\WW).\end{align}
We also write \begin{equation}\label{eq.tdash_vol}
t'_{\ell, \bullet}:=\sum_{\gl} t'_{\ell, \gl} =\intoi W(x,y) \psi'_\ell(x) \dd x \dd y = \intoi k(x) \psi'_\ell(x) \dd x, 
\end{equation}
defining $t_{\ell, \bullet}$ in the corresponding way. Then by Lemma~\ref{LM}, we may write the modularity score of $\Psi'$ on $W$ as below, scaling by $\rho^2$ for convenience,
\begin{align} 
\hq(W, \Psi')\rho^2 = \sum_{\ell=1}^{\kappa+1} 
\bigpar{\rho\, t'_{\ell, \ell} - (t'_{\ell, \bullet})^2 }   
.\end{align}
Thus,
\begin{align}\label{eq.mod_diff_after_purple_introduced}
\bigpar{\hat{q}(W, \Psi')-\hat{q}(W, \Psi)}\rho^2
& = \sum_{\ell=1}^{\kk+1} 
\bigsqpar{
\rho\, (t'_{\ell, \ell}-t_{\ell, \ell}) - (t'^{\;2}_{\ell, \bullet}-t_{\ell, \bullet}^2)  
}.\end{align}
We may now directly calculate the terms above. Note that $\psi_{\kappa+1}(x)=\mathbf{1}[x \in I_h]$ and for $x,y\in I_h$, $W(x,y)=1$, and so
\begin{equation}\label{eq.new_colour_edges} t'_{\kappa+1, \kappa+1} = \int_{[0,1]^2} W(x,y) \psi_{\kappa+1}(x) \psi_{\kappa+1}(y) \dd x \dd y = \int_{I_h^2} W(x,y) \dd x \dd y = 2^{-2h}, \end{equation}
and similarly
\begin{equation}\label{eq.new_colour_vol} t'_{\kappa+1, \bullet} = \int_{I_h\times [0,1]} W(x,y) \psi_{\kappa+1}(x) \dd x \dd y = 2^{-2h}. \end{equation}
Now fix $a\in [\kk]$, then, split the integral over $x\in [0,1]$ to $x\in [0,1]\setminus I_h$ and $x\in I_h$; then by~\eqref{eq.tdash_vol} 
\begin{align} 
t'_{a, \bullet} 
     &= \int_{[0,1]\setminus I_h} k(x) \, \psi'_{a}(x) \dd x + \int_{I_h} k(x) \, \psi'_{a}(x)  \dd x
.\end{align}
Since $\psi'_a(x)=0$ for $x\in I_h$, the second term above is zero. On the other hand for $x\notin I_h$, $\psi'_a(x)=\psi_a(x)$. Hence,  
\begin{align} \label{mar10}
t'_{a, \bullet} 
    &= \int_{[0,1]\setminus I_h} k(x) \, \psi_{a}(x) \dd x 
= t_{a, \bullet} - \int_{I_h} k(x) \, \psi_{a}(x) \dd x
,\end{align}
where for $x\in I_h$, $k(x)=2^{-h}$. Now, apply this and recall the definition of $x_\ell$ in~\eqref{eq.defn_x} to obtain,
\begin{align} \label{mar11}
t'_{a, \bullet} 
    &= t_{a, \bullet} - 2^{-h} x_a
\qquad\text{or}\qquad
t_{a, \bullet} = t'_{a, \bullet} + 2^{-h} x_a
\end{align}
and thus
\begin{align} \label{mar1}
t_{a, \bullet}^2 - t'^{\;2}_{a, \bullet} 
    &=  2^{1-h} t'_{a, \bullet} x_a +  2^{-2h} x_a^2. 
\end{align}
To relate $t'_{a,a}$ and $t_{a,a}$ is similar. Again for fixed $a\in [\kk]$,
noting that $\psi'_a(x)=0$ for all $x\in I_h$ and $\psi'_a(x)=\psi_a(x)$ for
$x\notin I_h$  we have, using the symmetry of $W$,
\begin{align} 
t'_{a, a} 
    &= \int_{[0,1]\setminus I_h}\int_{[0,1]\setminus I_h} W(x,y) \psi'_a(x) \psi'_a(y) \dd y \, \dd x \notag \\ 
    &= t_{a,a} - 2\int_{I_h}\int_{[0,1]\setminus I_h} W(x,y) \psi_a(x) \psi_a(y) \dd y \, \dd x 
\notag \\
& \hskip4em- \int_{I_h}\int_{I_h} W(x,y) \psi_a(x) \psi_a(y) \dd y \, \dd x
.\end{align}
However, $W(x,y)=0$ if $x\in I_j$ and $y \in I_{j'}$ for some $j\neq j'$, while $W(x,y)=1$ for $x,y\in I_h$. Thus,
\begin{align} 
\label{eq.t_dash_edges}
t'_{a, a} 
    =  t_{a,a} - \int_{I_h}\int_{I_h} \psi_a(x) \psi_a(y) \dd y \, \dd x 
    =  t_{a,a} - x_a^2 
.\end{align}
Hence by our expression for the difference in modularity
scores~\eqref{eq.mod_diff_after_purple_introduced}, as well
as~\eqref{eq.new_colour_edges}, \eqref{eq.new_colour_vol},
\eqref{mar1}, \eqref{eq.t_dash_edges}, and that $\rho=1/3$
from~\eqref{eq.inf_basics},
\begin{align}\label{mar2}
\bigpar{\hat{q}(W, \Psi')-\hat{q}(W, \Psi)}\rho^2
& = \rho\, t'_{\kappa+1, \kappa+1}- (t'_{\kappa+1, \bullet})^2 + 
\sum_{\ell=1}^\kk \bigsqpar{\rho\, (t'_{\ell, \ell}-t_{\ell, \ell}) 
+ (t_{\ell, \bullet}^2-(t'_{\ell, \bullet})^2)}
\notag\\
& =  2^{-2h}/3- 2^{-4h} + 
\sum_{\ell=1}^\kk\bigsqpar{-x_\ell^2/3 + 2^{1-h}t'_{\ell, \bullet}x_\ell+ 2^{-2h} x_\ell^2} 
\notag \\
& = \bigpar{1/3-2^{-2h}}\Bigpar{2^{-2h} -  \sum_{\ell=1}^\kk x_\ell^2}
 + 2^{1-h}\sumlkk t'_{\ell, \bullet} x_\ell.
\end{align}
By \eqref{mar0}, we have
\begin{align}
  \label{mar3}
2^{-2h} -  \sum_{\ell=1}^\kk x_\ell^2
=\Bigpar{\sum_{\ell=1}^\kk x_\ell}^2 - \sum_{\ell=1}^\kk x_\ell^2
=\sum_{\ell\neq\gl} x_\ell x_\gl
\ge0.
\end{align}
Hence, all terms on the \rhs{} of \eqref{mar2} are non-negative.
Moreover, in case (i), we have $x_\ell,x_\gl>0$ for two different colours
$\ell$ and $\gl$, and thus we have strict inequality in \eqref{mar3}.
And in case (ii), for some colour $\ell=\ell^*$ we have $x_\ell>0$ and also 
$\int_{I_{h_1}}\psi_\ell(x)\dd x>0$, and thus by \eqref{mar10}
\begin{align}
  t'_{\ell,\bullet} \ge \int_{I_{h_1}}k(x)\psi_\ell(x)\dd x
>0.
\end{align}
Hence, in both cases the \rhs{} of \eqref{mar2} is strictly positive, and
thus
$\hat{q}(W, \Psi')-\hat{q}(W, \Psi)>0$, which yields the desired contradiction.
\end{example}

 \begin{remark}
\refE{ex:inf_colours_needed} suggests that one might extend the
theory developed in this paper to include coloured graphons with a countably
infinite set of colours. We leave this for future research;
in the present paper, we have chosen to consider only colourings
with finitely many colours.
\end{remark}

\subsection{Pure and fractional modularity for graphs and graphons}
\label{SSMGG}

We have seen that pure colourings of a graphon correspond to
partitions of the vertex set $V(G)$ of a graph.
In particular,
for a graph $G$ and the corresponding graphon $W'_G$ (see \refE{EG1}), 
it was noted above that $\hqx(G)=\hqx_\pure(W'_G)$.
Similarly, mixed colourings of graphons correspond to fractional partitions.

    Let $\kk\ge 1$. A \emph{fractional $\kk$-partition}\label{def.frac_partition} of $V(G)$ is a collection
$\Phi=(\phi_1,\dots,\phi_\kk)$ of functions $\phi_j:V(G)\to[0,1]$ such that
\begin{align}\label{eq.def.frac_partition}
\sum_{j=1}^{\kk}\phi_j(v)=1\qquad\text{for every }v\in V(G).
\end{align}
We interpret $\phi_j(v)$ as the weight with which vertex $v$ belongs to
community $j$.

\begin{definition}[Fractional modularity of $G$]\label{def:fractional-modularity}
Let $G$ be a graph with $m\ge 1$ edges, and let
$\Phi=(\phi_1,\dots,\phi_\kk)$ be a fractional $\kk$-partition of $V(G)$. The
\emph{fractional modularity score} of $\Phi$ on $G$ is
\begin{equation}\label{eq:qPhi-def}
q(G,\Phi)
  := \frac{1}{2m}\sum_{u,v\in V(G)}
     \left(\mathbf{1}[uv\in E(G)]-\frac{d_ud_v}{2m}\right)
     \sum_{j=1}^{\kk}\phi_j(u)\phi_j(v).
\end{equation}
The \emph{fractional modularity} of $G$ is
\begin{align}\label{fm1}
q_\fra^*(G):=\sup_{\kk\ge 1}\ \sup_{\Phi} q_{\Phi}(G,\Phi),
\end{align}
where the inner supremum is over all fractional $\kk$-partitions $\Phi$ of
$V(G)$. We let $q(G,\Phi):=0$ and consequently $\qx_\fra(G):=0$ if $G$ has no edges.
Since every vertex partition $\cA$ gives rise to a corresponding
$\Phi_\cA$ satisfying $q(G,\Phi_\cA)=q(G,\cA)$, we have
$\qx_\fra(G)\ge \qx(G)$ for every graph $G$.
\end{definition}

We then have the following relations between pure and fractional/mixed
modularity for a graph $G$ and the corresponding 
atomic and continuous  graphons
defined in \refE{EG1}.
Cf.\ the similar result for maximum cut and the Cheeger constant
in \refTs{TMC1} and \ref{TCH1}, 
and note that, 
as for the Cheeger constant but unlike maximum cut, the quantities in
\eqref{tm2a} and \eqref{tm2b} may differ, 
as will be seen in \refE{ex:C5}.

\begin{theorem}\label{TM2}
  Let $G$ be a graph and let $W'_G$ and $W_G$ be the corresponding
graphons defined in \refE{EG1}.
Then
\begin{align}\label{tm2a}
  \qx(G)&=  \hqx_\pure(W'_G),
\\\label{tm2b}
\qx_\fra(G)&
=\hqx_\mix(W'_G)
=\hqx_\mix(W_G)
= \hqx_\pure(W_G)
.\end{align}
The same equalities hold also for $\kk$-modularity, for every fixed $\kk\ge1$.
\end{theorem}

\begin{proof}
As noted above, it follows directly by
comparing the definitions that
$\qx_\kk(G)=  \hqx_\kkpure(W'_G)$ for every $\kk$, and thus
$\qx(G)=  \hqx_\pure(W'_G)$.
Similarly, we have
$\qx_\kkfra(G)=  \hqx_\kkmix(W'_G)$ for every $\kk$, and thus
$\qx_\fra(G)=  \hqx_\mix(W'_G)$.

Furthermore, we have
$\hqx_\kkmix(W'_G)= \hqx_\kkmix(W_G)$ and
$\hqx_\mix(W'_G)= \hqx_\mix(W_G)$ 
by \refL{LC1} (or \refT{TLC2}),
and, finally,
$\hqx_\kkmix(W_G)= \hqx_\kkpure(W_G)$ and
$\hqx_\mix(W_G)= \hqx_\pure(W_G)$
by \refT{LX1}.
\end{proof}

We now show that for $C_5$ we have $\q(C_5)<\q_{\rm frac}(C_5)$ as referred to several times above; the path $P_5$ also exhibits this property, see~\cite{thesis}. 

\begin{example}\label{ex:C5}
Let $G=C_5$, with vertices labelled cyclically as $v_1,\ldots,v_5$. We first
compute the ordinary modularity of $C_5$. Since $C_5$ is 2-regular and
$m=e(C_5)=5$, for every partition $\cA$ of $V(C_5)$,
\begin{align}\label{c51}
q(C_5,\cA)=\frac{e_\cA(C_5)}{5}-\sum_{B\in \cA}\frac{\vol(B)^2}{100}=
\frac{e_\cA(C_5)}{5}-\sum_{B\in \cA}\frac{|B|^2}{25}.
\end{align}
The possible block-size patterns give the following upper bounds, all of which are attainable by suitable partitions of $C_5$:
\begin{align}\label{c52}
\begin{array}{c|c|c|c}
\text{block sizes of } \cA
& \max e_\cA(C_5)& \sum_{B\in \cA} (|B|/5)^2 & \max q(C_5,\cA)\\ 
\hline
5 & 5 & 1 & 0 \\
4+1 & 3 & 17/25 & -2/25\\
3+2 & 3 & 13/25 & 2/25 \\
3+1+1 & 2 & 11/25 & -1/25 \\
2+2+1 & 2 & 9/25 & 1/25\\
2+1+1+1 & 1 & 7/25 &-2/25 \\
1+1+1+1+1 & 0 & 5/25 & -1/5.
\end{array}
\end{align}
Hence, $q^*(C_5)=\frac{2}{25}$. This value is attained, for example, by the partition $\cA=\bigl\{\{v_1,v_2,v_3\},\{v_4,v_5\}\bigr\}$, for which $e_\cA(C_5)=3$. We now exhibit a fractional $2$-partition with larger score. Let
$\Phi=(\Phi_1,\Phi_2)$ be defined by
\begin{align}\label{c53}
\Phi(v_1)=\Phi(v_2)=(1,0),\qquad
\Phi(v_4)=\Phi(v_5)=(0,1),\qquad
\Phi(v_3)=\left(\tfrac12,\tfrac12\right).
\end{align}
Using Definition \ref{def:fractional-modularity},
\begin{align}\label{bb1}
q(C_5,\Phi)=\frac{1}{10}\sum_{u,v\in V(C_5)}
\left(\mathbf 1[uv\in E(C_5)]-\frac{d_ud_v}{10}\right)
\langle \Phi(u),\Phi(v)\rangle .  
\end{align}
The edge contribution in \eqref{bb1} is
\begin{align}\label{c54}
\frac1{10}\sum_{u,v}\mathbf 1[uv\in E(C_5)]\langle \Phi(u),\Phi(v)\rangle=
\frac2{10}\sum_{\{u,v\}\in E(C_5)}\langle \Phi(u),\Phi(v)\rangle .
\end{align}
For the five cyclic edges,
\begin{gather}\label{c55}
\langle \Phi(v_1),\Phi(v_2)\rangle =1,\quad \langle
  \Phi(v_2),\Phi(v_3)\rangle =\tfrac12,\quad \langle
  \Phi(v_3),\Phi(v_4)\rangle =\tfrac12,
\notag\\
\langle \Phi(v_4),\Phi(v_5)\rangle =1,\quad \langle \Phi(v_5),\Phi(v_1)\rangle =0.
\end{gather}
Thus,
\begin{align}\label{c56}
\sum_{\{u,v\}\in E(C_5)}\langle \Phi(u),\Phi(v)\rangle=3,
\end{align}
and the edge contribution to \eqref{bb1} is $3/5$. 
The degree-tax term is,
since $d_v=2$ for every $v$, 
\begin{align}\label{c57}
\sum_{\ell=1}^2\Bigpar{\frac1{10}\sum_{v\in V(C_5)}d_v\Phi_\ell(v)}^2=
\Bigpar{\frac{2(1+1+\frac12)}{10}}^2
+\Bigpar{\frac{2(\frac12+1+1)}{10}}^2
=\frac12.
\end{align}
Therefore,
\begin{align}\label{c58}
q(C_5,\Phi)=\frac35-\frac12=\frac1{10}.
\end{align}
Consequently,
\begin{align}\label{c59}
q_\fra^*(C_5)\ge \frac1{10}>\frac{2}{25}=q^*(C_5).
\end{align}

Equivalently, we can regard the partitions $\cA$ and the colouring $\Phi$ above as 
pure and mixed partitions of the atomic graphon $W'_{C_5}$.
\refT{TM2} shows that
$\hat q^*_{\mathrm{pure}}(W'_{C_5})=q^*(C_5)$
and $\qx_\fra(C_5)=\hqx_\mix(W'_{C_5})$; thus \eqref{c59} yields
\begin{align}\label{c510}
\hat q^*_{\mathrm{mix}}(W'_{C_5})
\ge \frac1{10} > \frac{2}{25}
=\hat q^*_{\mathrm{pure}}(W'_{C_5}).
\end{align}
This shows that, on an atomic probability space, optimising over mixed colourings can give a strictly larger value than optimising over pure colourings.

For the blow-up graphon $W_{C_5}$ on $[0,1]$, this discrepancy disappears:
\refT{TM2} shows  that
$\hqx_\mix(W_{C_5})= \hqx_\pure(W_{C_5})$, i.e.,  
the pure and mixed suprema agree. 
\refT{TM2} further shows that
$\hqx_\mix(W_{C_5})=\hqx_\mix(W'_{C_5})$,
but also, combined with \eqref{c510},
\begin{align}\label{c511}
\hqx_\pure(W_{C_5})=\hqx_\mix(W'_{C_5})>\hqx_\pure(W'_{C_5}),
\end{align}
in spite of the fact that 
$W_{C_5}$ is a pullback of $W'_{C_5}$, and thus  
$W_{C_5}$ and $W'_{C_5}$ are equivalent graphons.
Hence, this example also illustrates why the mixed definition is the
invariant one on general probability spaces. 
\end{example}

The next lemma shows that fractional modularity differs only slightly from
ordinary modularity when the graph is sufficiently dense. Recall that we denote the maximum degree of $G$ by $d_{\rm max}(G):=\max_v d_v$.

\begin{lemma}\label{lem:rounding}
Let $G$ be a graph on $n$ vertices with $m\ge 1$ edges, and let $\Phi$ be a
fractional partition of $V(G)$. Then there exists a partition $\cA$ of $V(G)$
such that
\begin{align}\label{lro1}
q(G,\cA)\ge q(G,\Phi)-\frac{ d_{\rm max}(G) 
}{2m}.\end{align}
Consequently, 
  \begin{align}\label{lro2}
    \qx(G) \le \qx_\fra(G) \le \qx(G)+\frac{{ d_{\rm max}(G)}}{2m}.
  \end{align}
\end{lemma}

\begin{proof}
For each vertex $v\in V(G)$, assign independently a random colour $X_v$ by
\begin{align}
\Pr(X_v=j)=\Phi_j(v),\qquad j=1,\dots,\kk.
\end{align}
Let $\cA_X$ be the resulting random partition of $V(G)$.

By taking expectations in the definition \eqref{D1a}, we obtain
\begin{align}
\E[q(G,\cA_X)]
 = \frac{1}{2m}\sum_{u,v\in V(G)}
   \left(\indic{uv\in E(G)}-\frac{d_ud_v}{2m}\right)
   \E [\indic{X_u=X_v}].
\end{align}
For $u\neq v$, we have
\begin{align}
   \E [\indic{X_u=X_v}]
=    \P(X_u=X_v)
   = \sum_{j=1}^{\kk}\Phi_j(u)\Phi_j(v),
\end{align}
while for $u=v$ we have
$   \E [\indic{X_u=X_v}]=1$.
Hence
\begin{align}
\E[q(G,\cA_X)]
&= \frac{1}{2m}\sum_{u\neq v}
   \left(\mathbf{1}[uv\in E(G)]-\frac{d_ud_v}{2m}\right)
   \sum_{j=1}^{\kk}\Phi_j(u)\Phi_j(v)
   - \frac{1}{4m^2}\sum_{u} d_u^2 \notag\\
&= q(G,\Phi)
   - \frac{1}{4m^2}\sum_{u} d_u^2
     \left(1-\sum_{j=1}^{\kk}\Phi_j(u)^2\right).
\end{align}
Therefore,
\begin{align}
\bigl| \E[q(G,\cA_X)] - q(G,\Phi) \bigr|
&\le \frac{1}{4m^2}\sum_u d_u^2
 \le \frac{1}{4m^2}{ d_{\rm max}(G)}\sum_u d_u
 = \frac{{d_{\rm max}(G)}}{2m}.
\end{align}
Hence, $\E[q(G,\cA_X)]\ge q(G,\Phi)-{d_{\rm max}(G)}/(2m)$, and so there exists a
deterministic realization $\cA$ of $\cA_X$ satisfying the same inequality.
This proves \eqref{lro1}, and \eqref{lro2} follows by taking the suprema over all fractional partitions.
\end{proof}

One consequence of \refL{lem:rounding} is the following interpretation of
the fractional modularity of a graph.
For a graph $G$ and an integer $b\ge1$, let 
$G[b]$ be the $b$-\emph{blow-up} of $G$, defined by
replacing each vertex $v$ in $G$ by $b$ copies $(v,1),\dots,(v,b)$,
and putting an edge between $(u,i)$ and $(v,j)$ whenever $uv\in E(G)$.
Thus, if $G$ has $n$ vertices and $m$ edges, then $G[b]$ has $bn$ vertices
and $b^2m$ edges. We note Theorem~\ref{TB} was first proven in~\cite{sampling}, however the theory in this paper allows a much shorter proof.

\begin{theorem}\label{TB}
For any graph $G$,
\begin{align}\label{tb1}
\qx_\fra(G) = \lim_{b\to\infty}\qx(G[b])    
=\sup_{b\ge1}\qx(G[b]).
\end{align}
Furthermore, $\qx_\fra(G[b])=\qx_\fra(G)$ for every $b\ge1$.
\end{theorem}

\begin{proof}
For every $b\ge1$ 
we have $W_{G[b]}=W_G$ for the corresponding graphons defined by the
standard construction in \refE{EG1}.  
Hence, \refT{TM2} (twice) yields
\begin{align}\label{tb2}
  \qx_\fra(G[b]) = \hqx(W_{G[b]}) = \hqx(W_{G}) =\qx_\fra(G).
\end{align}
Consequently, \eqref{lro2} applied to $G[b]$ yields
\begin{align}\label{tb3}
    \qx(G[b]) \le \qx_\fra(G) \le \qx(G[b])+\frac{bn}{2b^2m}
= \qx(G[b])+\frac{n}{2bm},
\end{align}
and \eqref{tb1} follows.
\end{proof}

\subsection{Continuity of modularity}

The main result of the present subsection is the following continuity
result, which is proved at the end of the subsection after some preliminary
results.

\begin{theorem}\label{TM}
The functional $W\mapsto\hqx(W)$ is continuous 
at every uncoloured graphon $W$ with $\rho_W>0$.
\end{theorem}

\begin{remark}\label{Rnot0}
  We do not have continuity at $W$ when $\rho_W=0$, i.e, when $W=0$ 
a.e.
To see this, take any graphon $W_1$ with $\hqx(W_1)>0$, and let
$W_n:=\frac1n W_1$, $n\ge2$. Then $W_n\to0$
(in cutnorm, and thus in $\cWo$),
but $\hqx(W_n)=\hqx(W_1)>0$ for all $n$.
\end{remark}

We begin with a version where the number of colours is fixed.

\begin{theorem}\label{TMk}
Fix $\kk\ge1$. Then
the functional $W\mapsto\hqx_\kk(W)$ is continuous 
at every uncoloured graphon $W$ with $\rho_W>0$.
\end{theorem}
\begin{proof}
  Considering the functionals
  \begin{align}\label{sw1}
    \hQ_\kk(W,\Psi)&:=\rho_W^2\hq_\kk(W,\Psi),
\\\label{sw2}
\hQx_\kk(W)&:=\sup_{\Psi\in\PPsi_\kk}\hQ_\kk(W,\Psi)=\rho_W^2\hqx_\kk(W).
  \end{align}
\refL{LM} shows that, for every $\WW=(W,\Psi)$ with $\Psi\in\PPsi_\kk$,
\begin{align}\label{lm+}
  \hQ_\kk(\WW)={\rho_W} \sumlkk t(\KK_{2,\ell\ell},\WW)
-\sumlkk\Bigpar{\sumglkk t(\KK_{2,\ell\glm},\WW)}^2
.\end{align}
Hence it follows from \refL{LA1} and \eqref{m30} that for some constant
$C_\kk$ depending only on $\kk$, 
\begin{align}\label{sw4}
  \bigabs{\hQ_\kk(W_1,\Psi)-\hQ_\kk(W_2,\Psi)}\le C_\kk\cnorm{W_1-W_2}.
\end{align}
Consequently,
taking the suprema over colourings $\Psi$,
\begin{align}\label{sw5}
  \bigabs{\hQx_\kk(W_1)-\hQx_\kk(W_2)}\le C_\kk\cnorm{W_1-W_2}.
\end{align}
Together with \refT{TLC2}, this shows that
$W\to\hQx_\kk(W)$ is continuous for 
the cut metric, and thus $\hQx_\kk$ is continuous as a functional on the space
of uncoloured graph limits. Hence, by \eqref{sw2}, $\hqx_\kk$ is continuous at
every $W$ with $\rho_W>0$.
\end{proof}

Lemma~\ref{LMkl} below shows that we may approximate the modularity of a
graphon to within a multiplicative factor of $(1-1/\kappa)$ by considering
only colourings with at most $\kappa$ colours. This 
mirrors the
behaviour of modularity for finite graphs --
see~\cite[Lemma~1]{dinh2011finding}, which proves that for any partition
$\cA$, there is a partition~$\cA'$ with at most $k$ parts such that 
\begin{align}\label{eq.dinhthai}
q_{\cA'}(G)\geq \frac{k-1}{k}\, q_{\cA}(G).
\end{align}
(We use the same method of proof.)
However, the modularity score of a graphon need not be achieved by a finite colouring
-- see Figure~\ref{fig:inf_colours_needed} and Example~\ref{ex:inf_colours_needed}.

\begin{lemma}\label{LMkl}
Let $\nu\ge\kk\ge1$. Then, for
every uncoloured graphon $W$,
\begin{align}\label{lmkl}
\frac{\kk-1}{\kk}  \hqx_\nu(W)
& \le  \hqx_\kk(W) \le  \hqx_\nu(W), \;\;\mbox{and }\\
\frac{\kk-1}{\kk}  \hqx_{\nupure}(W)
& \le  \hqx_{\kkpure}(W)  \le \hqx_{\nupure}(W)
\label{lmkl_pure}
.
\end{align}
\end{lemma}

\begin{proof}
  We may assume $\rho_W>0$. 
The second inequality 
in both \eqref{lmkl} and \eqref{lmkl_pure}
is immediate from the
  definitions, since every $\kk$-colouring may be regarded as a
  $\nu$-colouring.

Now, let $\Psi\in\PPsi_\nu$ be a $\nu$-colouring of $W$.
We define a random $\kk$-colouring $\Psix\in\PPsi_\kk$ by mapping each
colour $\ell\in[\nu]$ to a uniformly random colour $Z_\ell\in[\kk]$,
independently for each $\ell$. More formally, we define
\begin{align}\label{st1}
  \psix_k(x):=\sum_{\ell=1}^\nu \psi_\ell(x)\indic{Z_\ell=k},
\qquad x\in\gO.
\end{align}
This means that, for any $x,y\in\gO$,
\begin{align}\label{st2}
  \sum_{k=1}^\kk \psix_k(x)\psix_k(y)&
= \sum_{k=1}^\kk\sum_{\ell_1,\ell_2=1}^\nu \psi_{\ell_1}(x)\psi_{\ell_2}(y)\indic{Z_{\ell_1}=k}
\indic{Z_{\ell_2}=k}
\notag\\&
=\sum_{\ell_1,\ell_2=1}^\nu \psi_{\ell_1}(x)\psi_{\ell_2}(y)\indic{Z_{\ell_1}=Z_{\ell_2}}.
\end{align}
Note that for any $\ell_1,\ell_2\in[\nu]$,
\begin{align}\label{st3}
  \P\xpar{Z_{\ell_1}=Z_{\ell_2}}
=
  \begin{cases}
    1,& \ell_1=\ell_2
\\
\frac{1}{\kk}, & \ell_1\neq \ell_2
  \end{cases}
\Biggr\}
=\frac{1}{\kk}+\frac{\kk-1}{\kk}\indic{\ell_1=\ell_2}.
\end{align}
Consequently, 
\eqref{st2} yields by
taking the expectation over the random colouring $\Psix$
and recalling $\sumlnu \psi_\ell(x)=1$,
\begin{align}\label{st4}
\E \sum_{k=1}^\kk \psix_k(x)\psix_k(y)&
=\sum_{\ell_1,\ell_2=1}^\nu \psi_{\ell_1}(x)\psi_{\ell_2}(y)\P\xpar{Z_{\ell_1}=Z_{\ell_2}}
\notag\\&
=\sum_{\ell_1,\ell_2=1}^\nu \psi_{\ell_1}(x)\psi_{\ell_2}(y)
  \Bigpar{\frac{\kk-1}{\kk}\indic{\ell_1=\ell_2}+\frac{1}{\kk}}
\notag\\&
=\frac{\kk-1}{\kk}\sum_{\ell=1}^\nu \psi_\ell(x)\psi_\ell(y)
+\frac{1}{\kk}
.\end{align}
Substituting \eqref{st4} in \eqref{m3}, we obtain
\begin{align}\label{st5}
\E [\hq_\kk(W,\Psix)]&
=\frac{1}{\rho_W}\iint_{\gO^2}\Bigpar{W(x,y)-\frac{k(x)k(y)}{\rho_W}}
\E\biggsqpar{\sumkkk\psix_k(x)\psix_k(y)}\dd\mu(x)\dd \mu(y)
\notag\\&
=\frac{1}{\rho_W}\iint_{\gO^2}\Bigpar{W(x,y)-\frac{k(x)k(y)}{\rho_W}}
\frac{\kk-1}{\kk}\sumlgl\psi_\ell(x)\psi_\ell(y)\dd\mu(x)\dd \mu(y)
\notag\\&\qquad +
\frac{1}{\kk\rho_W}\iint_{\gO^2}\Bigpar{W(x,y)-\frac{k(x)k(y)}{\rho_W}}
\dd\mu(x)\dd \mu(y)
\notag\\&
=\frac{\kk-1}{\kk}\hq_\nu(W,\Psi)
,\end{align}
using \eqref{m3} again and noting that the final integral is 0 as a
consequence of \eqref{m1}--\eqref{m2}.
Consequently, there exists a realization $\Psix$ such that
\begin{align}\label{st6}
\frac{\kk-1}{\kk}\hq_\nu(W,\Psi)\le
\hq_\kk(W,\Psix)\le \hqx_\kk(W),
\end{align}
and the first inequality in \eqref{lmkl} follows by taking the supremum over
$\Psi\in\PPsi_\nu$. To see that the statement extends to the pure case in~\eqref{lmkl_pure}, it is enough to note that if $\Psi$ is pure, then the random colouring $\Psi^*$ is a random pure colouring.
\end{proof}

\begin{lemma}\label{LMk2}
Let $\kk\ge1$. Then, for
every uncoloured graphon $W$,
\begin{align}\label{lmk2}
\frac{\kk-1}{\kk}  \hqx(W)
&\le  \hqx_\kk(W) 
\le  \hqx(W) \qquad\text{and}
\\\label{lmk2_pure}
\frac{\kk-1}{\kk}  \hqx_{\pure}(W)
& \le  \hqx_{\kkpure}(W)  \le \hqx_{\pure}(W)
.\end{align}
\end{lemma}
\begin{proof}
  Follows from \refL{LMkl} by taking the supremum over $\nu\ge\kk$.
\end{proof}

\begin{proof}[Proof of \refT{TM}]
It follows from Lemmas \ref{LMk2} and \ref{LF1} that
\begin{align}\label{fa1}
\bigabs{\hqx(W)-\hqx_\kk(W)} 
\le  \frac{1}{\kk}.
\end{align}
Consequently, $\hqx_\kk(W)\to\hqx(W)$ uniformly in $W$ as $\kk\to\infty$, 
and the result follows by \refT{TMk}.
\end{proof}

\subsection{The graphon limit of modularity}\label{sec4:modularity-limit}  
The main result of this subsection is
Theorem~\ref{thm:limit-identification}, which identifies graphon
modularity \eqref{mod3} as the limit of finite graph modularity
along cut-convergent graph sequences with positive limiting
edge density.

\begin{theorem}\label{thm:limit-identification}
Let $(G_n)_{n \geq 1}$ be graphs with $N_n:=|G_n|\to \infty$, and suppose
$G_n\to W$,
where $W$ is a graphon with $\rho_W > 0$. Then,
\begin{align}\label{mell0}
\lim_{n\to \infty}q^*(G_n)=\hqx(W).
\end{align}
\end{theorem}

\begin{proof}
Let $m_n:=e(G_n)$. 
By Lemma~\ref{lem:rounding},
we have
\begin{align}\label{mell1}
   \bigabs{\qx(G_n) - \qx_\fra(G_n)} \le \frac{N_n}{2m_n}.
\end{align}
let $W_{G_n}$ denote the step graphon corresponding
to $G_n$. Then
the assumption $G_n\to W$ is equivalent to 
\begin{align}\label{mell3}
W_{G_n}\to W  
\end{align}
(in the usual topology, i.e., in cut distance), which implies
\begin{align}\label{mell2}
  \frac{2m_n}{N_n^2} =\rho(W_{G_n}) \to \rho(W)>0,
\end{align}
and hence $m_n=\Theta(N_n^2)$.
Thus $N_n/2m_n\to 0$, so \eqref{mell1} implies
\begin{align}\label{mell5}
\qx(G_n) - \qx_\fra(G_n)\to0.  
\end{align}
Furthermore, we have $\qx_\fra(G_n)=\hqx(W_{G_n})$ by \refT{TM2}, and
thus, by \eqref{mell5} and \refT{TM},
\begin{align}\label{mell4}
\qx_\fra(G_n)=  \hqx(W_{G_n})\to\hqx(W)
\end{align}
The result \eqref{mell0} follows from \eqref{mell5} and \eqref{mell4}.
\end{proof}
In particular, \refT{LX1} implies that, if $G_n\to W$, where 
$W:[0,1]^2\to[0,1]$, 
then
\begin{align}
\lim_{n \to \infty} q^*(G_n) \to \qhatstarpur(W),
\end{align}
where $\qhatstarpur(W)$ is the graphon modularity defined in \cite{Klimm}.

\subsection{Graphons with modularity value zero.}
The following proposition characterises which (uncoloured) graphons $W$ have modularity zero. 
Note (i)--(iii) may be considered a graphon analogue of 
\cite[Proposition~1.7]{modexpansion} -- see Remark~\ref{rem.link_to_modexpansion}.
 We again write $\KK_{2;ij}$ for the coloured graph with 
one edge and two vertices coloured $i$ and~$j$, and define (similar to the $t_{i,\bullet}$ notation defined in~\eqref{eq.tdash_vol}),
\begin{align}\label{eq.def_tbullet}
t(\KK_{2,i\bullet},\WW):=\sum_{j}  t(\KK_{2, ij},\WW).
\end{align}
\begin{proposition}\label{prop.equiv_to_mod_zero}
Let $W$ be an uncoloured graphon with density $\rho_W>0$ defined on $\gO$. Then~\ref{equiv.mod_zero} to~\ref{equiv.hom_vers_vols} below are equivalent. 
  \begin{thmenumerate}
    \item\label{equiv.mod_zero} $\hqx(W)=\qhatstarmix(W)=0$
    \item\label{equiv.hom_vers}
    For any mixed two-colouring $\Psi$ of $W$, writing $\WW=(W, \Psi)$,
    \begin{align}\label{eq.hom_cond_for_zero} t(\KK_{2,12},\WW)^2 \geq \, t(\KK_{2,11},\WW)\, t(\KK_{2,22},\WW) \end{align}  
    \item\label{equiv.hom_vers_vols}
    For any mixed two-colouring $\Psi$ of $W$, writing $\WW=(W, \Psi)$,
    \begin{align}\label{eq.hom_vols_cond_for_zero} t(K_2, W) \, t(\KK_{2,12},\WW) \geq \, t(\KK_{2,1\bullet},\WW)\, t(\KK_{2,2\bullet},\WW) \end{align} 
    \end{thmenumerate}
    \noindent Additionally,~\ref{equiv.mod_pure_zero} to~\ref{equiv.int_vers_vols} below are equivalent, and implied by any of~\ref{equiv.mod_zero} to \ref{equiv.hom_vers_vols} above.
    \begin{thmenumerate}
    \addtocounter{thmenumerate}{3}
    \item\label{equiv.mod_pure_zero}
    $\qhatstarpur(W)=0$
    \item\label{equiv.hom_vers_pure}
    For any pure two-colouring $\chi$ of $W$, writing $\WW=(W, \chi)$,~\eqref{eq.hom_cond_for_zero} holds.
    \item\label{equiv.hom_vers_vols_pure}
    For any pure two-colouring $\chi$ of $W$, writing $\WW=(W, \chi)$,~\eqref{eq.hom_vols_cond_for_zero} holds.
    \item\label{equiv.int_vers}
    For any measurable set $I\subseteq\Omega$ (writing $I\comp:=\Omega\setminus I$)
    \begin{align}\label{eq.int_cond_for_zero} 
&\biggpar{\int_I \int_{I\comp}
      W(x,y) \dd\mu(x) \dd\mu(y)}^2 
\notag\\&\qquad
\geq \biggpar{\int_I \int_{I} W(x,y) \dd\mu(x) \dd\mu(y)}
\biggpar{ \int_{I\comp} \int_{I\comp} W(x,y) \dd\mu(x) \dd\mu(y)}
.\end{align}
    \item\label{equiv.int_vers_vols}
    For any measurable set $I \subseteq \Omega$ (writing $I\comp:=\Omega\setminus I$)
    \begin{align}\label{eq.int_cond_for_zero_vols} \rho\int_I \int_{I\comp} W(x,y) \dd\mu(x) \dd\mu(y) \geq \biggpar{\int_{I} k(x) \dd\mu(x)}\biggpar{ \int_{I\comp} k(x) \dd\mu(x)}
 .  \end{align}
\end{thmenumerate}
    \noindent Further, if we assume that $W$ is defined on $\oi$, then~\ref{equiv.mod_zero} to~\ref{equiv.int_vers_vols} are all equivalent.
\end{proposition}

\begin{remark}\label{rem.link_to_modexpansion} In this remark we explain why we may consider Proposition~\ref{prop.equiv_to_mod_zero} the graphon equivalent of~\cite[Proposition~1.7]{modexpansion}. That result on graphs states that conditions (i)--(iii) are equivalent where (i) $\q(G)=0$, and the conditions (ii) and (iii) state, respectively, that for any set of vertices $A\subset V(G)$
\begin{align}\label{eq.modexpansion_equivs}
\mbox{(ii):}\;\; e(A, A^c)^2 \geq 4e(A)e(A^c)\quad \quad \mbox{(iii):}\;\; e(A, A^c)\vol(G) \geq \vol(A)\vol(A^c)
\end{align}
We will see that \ref{equiv.mod_zero}, \ref{equiv.hom_vers} and \ref{equiv.hom_vers_vols} respectively in
Proposition~\ref{prop.equiv_to_mod_zero} are the graphon equivalents of
these.
Condition \ref{equiv.mod_zero}  is immediate since  $\hqx(W)=0$ is clearly the graphon analogue of $\q(G)=0$.

Note that the partition $\{A, A^c\}$ induces a pure-colouring of $V(G)$, denote this $\chi_A$ and write $\GG=(G, \chi_A)$; then by defintion~\eqref{tinj} of coloured homomorphism densities we may calculate that 
\begin{align}\label{eq.homs_to_edges}
\tinj(\KK_{2,12},\GG) = \frac{e(A,A^c)}{n(n-1)} \quad \mbox{and} \quad 
\tinj(\KK_{2,11},\GG) = \frac{2e(A)}{n(n-1)}
\end{align}
(note also the corresponding relation between $\KK_{2,22}$ and $e(A^c)$). Then condition (ii) in~\eqref{eq.modexpansion_equivs} from~\cite{modexpansion} can be equivalently written as,
\begin{align}
t(\KK_{2,12},\GG)^2 \geq \, t(\KK_{2,11},\GG)\,t(\KK_{2,22},\GG)
\end{align}
which clearly matches the form of~\eqref{eq.hom_cond_for_zero} in condition \ref{equiv.hom_vers} of Proposition~\ref{prop.equiv_to_mod_zero}. 

To see the other equivalence write $\tinj(\KK_{1,\bullet},\GG):=\sum_{j} \tinj(\KK_{1,j},\GG)$. Then by~\eqref{eq.homs_to_edges}, $\tinj(\KK_{1,\bullet},\GG)= \vol(A)/(n(n-1))$. Note also the uncoloured density $\tinj(K_2, G)=\vol(G)/(n(n-1))$. Now condition (iii) in~\eqref{eq.modexpansion_equivs} from~\cite{modexpansion} becomes,
\begin{align} 
t(\KK_{2,12},\GG) \, t(K_2, G) \geq \, t(\KK_{2,1\bullet},\GG)\, t(\KK_{2,2\bullet},\GG) \end{align}  
which matches~\eqref{eq.hom_vols_cond_for_zero} in Proposition~\ref{prop.equiv_to_mod_zero}\ref{equiv.hom_vers_vols}.
\end{remark}

\begin{proof}[Proof of Proposition~\ref{prop.equiv_to_mod_zero}]
We first establish that \ref{equiv.mod_zero}--\ref{equiv.hom_vers_vols}
  are equivalent by showing that
  \ref{equiv.hom_vers}$\Rightarrow$\ref{equiv.mod_zero},
  \ref{equiv.mod_zero}$\Rightarrow$\ref{equiv.hom_vers}
  and~\ref{equiv.hom_vers}$\Leftrightarrow$\ref{equiv.hom_vers_vols} in Part
  1. Since any pure-colouring can be expressed as a mixed-colouring
  \ref{equiv.mod_zero} 
  implies \ref{equiv.mod_pure_zero}, hence by Part 1, 
any of \ref{equiv.mod_zero}--\ref{equiv.hom_vers_vols} 
then implies \ref{equiv.mod_pure_zero}. Also, if
  $W$ is defined on $[0,1]$, then by Theorem~\ref{LX1},
  $\qhatstarpur(W)=\qhatstarmix(W)$ and thus \ref{equiv.mod_zero}  and
  \ref{equiv.mod_pure_zero} are 
  equivalent. Hence by Part 1, to conclude the proof it then only remains to
  prove that~\ref{equiv.mod_pure_zero}--\ref{equiv.int_vers_vols} are all
  equivalent, which we do in Part 2 below.

\medskip
\noindent\emph{Part 1:~\ref{equiv.mod_zero}-\ref{equiv.hom_vers_vols} equivalent.}\\
\noindent\ref{equiv.hom_vers}$\Rightarrow$\ref{equiv.mod_zero}. Assume that
$\hqx(W)>0$, and we will show that $\eqref{eq.hom_cond_for_zero}$ does not
hold. Since $\hqx(W)>0$, 
we have also that $\hqx_2(W)>0$, see Lemma~\ref{LMk2}. 
By~\refT{TLC3}, we may fix two-colouring $\Psi$ such that $\hq(W,\Psi)=\hqx_2(W)$. 

Recall the definitions of $t_{i,j}$ 
from~\eqref{eq.tab_defn}; and that we have $\rho=t_{1,1}+2t_{1,2}+t_{2,2}$ from~\eqref{mag3}.
By \refL{LM}, 
\begin{align}\label{eq.q_as_t12}
\hq(W,\Psi)\rho^2 & = \bigpar{t_{1,1}+t_{2,2}}\bigpar{t_{1,1}+2t_{1,2}+t_{2,2}}-\bigpar{t_{1,1}+t_{1,2}}^2-\bigpar{t_{1,2}+t_{2,2}}^2 \\
\label{eq.q_as_t12_2}
 & = 2t_{1,1}t_{2,2}-2t_{1,2}^2 
\end{align}
and thus $t_{1,1}t_{2,2}>t_{1,2}^2$ which contradicts~\eqref{eq.hom_cond_for_zero} and finishes the proof of the implication~\ref{equiv.hom_vers}$\Rightarrow$\ref{equiv.mod_zero}.

\medskip
\noindent\ref{equiv.mod_zero}$\Rightarrow$\ref{equiv.hom_vers}. Assume that there exists a mixed two-colouring $\Psi$ such that $\eqref{eq.hom_cond_for_zero}$ does not
hold and we will show that
$\hqx(W)>0$. Our assumption implies that for the mixed two-colouring $\Psi$ we have $t_{1,1}t_{2,2}>t_{1,2}^2$. However, by~\eqref{eq.q_as_t12_2}, this implies that $\qhat(W, \Psi)>0$, and thus $\hqx(W)>0$. This concludes the proof of the implication~\ref{equiv.mod_zero}$\Rightarrow$\ref{equiv.hom_vers}.

\medskip
\noindent\ref{equiv.hom_vers}$\Leftrightarrow$\ref{equiv.hom_vers_vols}. Again, recall the definition of $t_{i,j}$ from~\eqref{eq.tab_defn} and also that $\rho=t(K_2,W)=t_{1,1}+2t_{1,2}+t_{2,2}$ from~\eqref{mag3}. Then~\eqref{eq.hom_vols_cond_for_zero} holds if and only if,
\begin{align}\label{eq.equiv1}
t_{1,2}\bigpar{t_{1,1}+2t_{1,2}+t_{2,2}} - \bigpar{t_{1,1}+t_{1,2}}\bigpar{t_{1,2}+t_{2,2}}\geq 0.
\end{align}
Noting cancellations on the left hand side we see that~\eqref{eq.equiv1} is equivalent to
\begin{align}
t_{1,2}^2 - t_{1,1}t_{2,2}\geq 0.
\end{align}
and thus~\eqref{eq.hom_vols_cond_for_zero} holds if and only if~\eqref{eq.hom_cond_for_zero} holds; and hence 
\ref{equiv.hom_vers}$\Leftrightarrow$\ref{equiv.hom_vers_vols}.

\bigskip
\noindent{\emph{Part 2: \ref{equiv.mod_pure_zero}--\ref{equiv.int_vers_vols} equivalent.}}\\ 
We first show that~\ref{equiv.mod_pure_zero}--\ref{equiv.hom_vers_vols_pure} are equivalent (which follows in much that same way as showing~\ref{equiv.mod_zero}--\ref{equiv.hom_vers_vols} above). Then to conclude the proof we show that~\ref{equiv.hom_vers_pure}$\Leftrightarrow$\ref{equiv.int_vers} and that~\ref{equiv.hom_vers_vols_pure}$\Leftrightarrow$\ref{equiv.int_vers_vols}.

\medskip
\noindent\ref{equiv.mod_pure_zero}$\Rightarrow$\ref{equiv.hom_vers_pure}. Assume that there exists a pure two-colouring $\chi$ such that $\eqref{eq.hom_cond_for_zero}$ does not
hold and we will show that
$\qhatstarpur(W)>0$. By assumption, for $\chi$ we have that
$t_{1,1}t_{2,2}>t_{1,2}^2$. But then noting that the
formula~\eqref{eq.q_as_t12_2} holds (as a special case) also 
for pure colourings; 
this implies
that $\qhat(W,\chi)>0$; and thus that $\qhatstarpur(W)>0$. This concludes
the proof that
\ref{equiv.mod_pure_zero}$\Rightarrow$\ref{equiv.hom_vers_pure}. 

\medskip
\noindent\ref{equiv.hom_vers_pure}$\Rightarrow$\ref{equiv.mod_pure_zero}.
Assume that $\qhatstarpur(W)>0$, and we will show that there exists a pure
colouring $\chi$, such that for $\WW=(W, \chi)$ the inequality
\eqref{eq.hom_cond_for_zero} does not hold. By Lemma~\ref{LMk2} this implies
$\qhatstarpurtwo(W)>0$ and hence by the definition of $\qhatstarpurtwo(W)$,
see~\eqref{mod4}, there exists a pure two-colouring $\chi$ with
$\qhat(W,\chi)>0$. Then as in the proof
of~\ref{equiv.hom_vers}$\Rightarrow$\ref{equiv.mod_zero}, for this
two-colouring $\chi$, set $\WW=(W,\chi)$ and we must have
$t_{1,1}t_{2,2}>t^2_{1,2}$, which implies that~\eqref{eq.hom_cond_for_zero}
does not hold. This concludes the proof
that~\ref{equiv.hom_vers_pure}$\Rightarrow$\ref{equiv.mod_pure_zero}.

\medskip
\noindent\ref{equiv.hom_vers_pure}$\Leftrightarrow$\ref{equiv.hom_vers_vols_pure}. In the proof of \ref{equiv.hom_vers}$\Leftrightarrow$\ref{equiv.hom_vers_vols} we established that for any fixed mixed two-colouring~$\Psi$, the inequality~\eqref{eq.hom_cond_for_zero} holds if and only if~\eqref{eq.hom_vols_cond_for_zero} holds. Hence since mixed two-colourings include pure two-colourings, we have that \ref{equiv.hom_vers_pure}$\Leftrightarrow$\ref{equiv.hom_vers_vols_pure}.

\medskip
\noindent(Note that we now have that~\ref{equiv.mod_pure_zero}--\ref{equiv.hom_vers_vols_pure} are equivalent, and it remains only to show that \ref{equiv.hom_vers_pure}$\Leftrightarrow$\ref{equiv.int_vers} and \ref{equiv.hom_vers_vols_pure}$\Leftrightarrow$\ref{equiv.int_vers_vols}.)

\medskip
\noindent\ref{equiv.hom_vers_pure}$\Leftrightarrow$\ref{equiv.int_vers}. Under the natural identification
between measurable $I$ and pure two-colourings $\chi$ (see, for example, after~\eqref{eq:graphon-MC}), by definition~\eqref{pa1b} we have, writing $\WW=(W,\chi)$,
\begin{align}\label{eq.t12}
t(\KK_{2,12},\WW) = \iint_{\gO^2} W(x,y) \chi_1(x)\chi_2(y) \dd\mu(x) \dd\mu(y)= \int_I \int_{I\comp} W(x,y) \dd\mu(x) \dd\mu(y)
\end{align}
and
\begin{align}\label{eq.t11}
t(\KK_{2,11},\WW) = \iint_{\gO^2} W(x,y) \chi_1(x)\chi_1(y) \dd\mu(x) \dd\mu(y)= \iint_{I^2} W(x,y) \dd\mu(x) \dd\mu(y)
\end{align}
with a similar expression for $t(\KK_{2,22},\WW)$. Thus under this identification between $I$ and $\chi$ the inequalities~\eqref{eq.hom_cond_for_zero} and~\eqref{eq.int_cond_for_zero} are equivalent, which concludes the proof that~\ref{equiv.hom_vers_pure}$\Leftrightarrow$\ref{equiv.int_vers}.

\medskip
\noindent\ref{equiv.hom_vers_vols_pure}$\Leftrightarrow$\ref{equiv.int_vers_vols}. This proof is similar to that immediately above, we show that under the natural identification between $I$ and $\chi$ the inequalities~\eqref{eq.hom_vols_cond_for_zero} and~\eqref{eq.int_cond_for_zero_vols} are equivalent. Since by definition, see~\eqref{eq.def_tbullet}, $t(\KK_{2,1\bullet},\WW):=t(\KK_{2,11},\WW)+t(\KK_{2,12},\WW)$, then by adding~\eqref{eq.t12} and~\eqref{eq.t11}, we get that
\begin{align}\label{eq.t1bullet}
t(\KK_{2,1\bullet},\WW) = \int_{I} \int_{\gO} W(x,y) \dd\mu(x) \dd\mu(y) = \int_{I} k(x) \dd\mu(x). 
\end{align}
where the second equality follows by the definition of $k(x)$, see~\eqref{m1}. By the same argument $t(\KK_{2,2\bullet},\WW)=\int_{I\comp} k(x) \dd\mu(x)$. Lastly note that for the (usual) uncoloured density of $K_2$ in $W$, we have
$t(K_2, W)=\iint_{\gO^2} W(x,y) = \rho$. 

Now by these expressions for $t(K_2, W)$, $t(\KK_{2,i\bullet}, \WW)$ and that for $t(\KK_{2,12}, \WW)$ from~\eqref{eq.t12}, the inequalites~\eqref{eq.hom_vols_cond_for_zero} and~\eqref{eq.int_cond_for_zero_vols} are indeed equivalent, as required. This concludes the proof that \ref{equiv.hom_vers_vols_pure}$\Leftrightarrow$\ref{equiv.int_vers_vols}, and hence concludes the proof of Proposition~\ref{prop.equiv_to_mod_zero}.
\end{proof}

\section{Restricted estimability}\label{sec.estim} Theorems \ref{thm:cheeger-conv} and \ref{thm:limit-identification} establish convergence of the Cheeger constant and modularity under positive lower bounds on normalised minimum degree and edge density, respectively. As discussed in Sections \ref{sec:cheeger-constant} and \ref{sec3:modularity}, without such restrictions, sequences with the same graphon limit can have different limiting parameter values; see also \cite[Section~5.1]{Khetan}. For bounded graph parameters, \cite[Theorem 6.1]{BCLSV1}
characterises estimability by convergence along every
cut-convergent graph sequence whose order tends to infinity. In this section, we develop an analogous characterisation for restricted graph
classes, connecting our convergence results to estimability.

\vspace{0.5em}

\noindent Let $\cG$ be the set of all finite simple graphs modulo isomorphisms.
For a graph $G\in\cG$ and an integer $k\le|G|$,
let $\mathbbm{G}(k,G)\in\cG$\label{def.G_vertex_sampling} be the random induced subgraph of $G$ obtained from sampling $k$ vertices uniformly without replacement in $G$.   
Furthermore, let $\cW=\cW(\oi)$ be the space of all graphons on $\oi$,
i.e., all measurable functions \(W: [0,1]^2\to [0,1]\) that are symmetric. (In this section, we consider only graphons defined on $\oi$.)

Moreover, let \(r: \cG\to [0,\infty)\) be a graph parameter with the following properties: 
\begin{enumerate}[label=(R\arabic*)]
\item\label{Cest1} There exists \(\hat{r}: \cW\to [0,\infty)\) such that  for every sequence of finite simple graphs \((G_n)_{n\ge 1}\) such that \(|G_n|\to \infty\) as \(n\toinf\), 
we have
\begin{align}\label{r1} 
r(G_n) - \hat{r}(W_{G_n})\to 0,\quad n\toinf,
\end{align}
and furthermore,
$\hr$ is upper semicontinuous for the usual topology on $\cW$ defined by the
cut distance, in other words, for any $W,W_n\in\cW$,
\begin{align}\label{usc}
\dcut(W_n,W)\to0
\implies \limsup_\ntoo\hr(W_n)\le \hr(W).
\end{align}

\item\label{Cest2} For every $\eta>0$ and $\eps>0$, there exists
$k_1(\eps,\eta)>0$ such that for all $k\ge k_1$
and \(G\in \cG\) with \(|G|\ge k\) and
\(r(G)\ge \eta\),
we have
\begin{align}
\IP\bigpar{r(\mathbbm{G}(k,G)) \le \eta - \eps} \le \eps.  
\end{align}
\end{enumerate}
Note that the upper semicontinuity \eqref{usc}
also can be written
$\limsup_{W'\to W}\hr(W')\le \hr(W)$ for all $W\in\cW$, and that it is
equivalent to the sets
$ 
  \set{W\in\cW: \hr(W)\ge\eta}
$
being closed in $\cW$ for every $\eta>0$.
Furthermore, \eqref{usc} in particular implies that
$\hat{r}(W)=\hat{r}(W')$ for any $W,W'\in \cW$ such that 
are equivalent, i.e., have
$\dels(W,W')=0$. 

We extend the definition of estimability in
\cite[Section 6]{BCLSV1} and \cite[Chapter 15.1]{Lovasz}
to restricted classes of graphs as follows.
\begin{definition}[Restricted estimability]\label{de:est1}
A graph parameter $f$ is \emph{estimable} on a given class
$\cG_1\subseteq\cG$ of graphs, if for every $\eps>0$, there is a positive integer $k = k(\eps)$
such that for every $G\in\cG_1$ with  $|G|\ge k$, 
we have
\begin{align}
\P(|f(G)-f(\mathbbm{G}(k,G))| >\eps)<\eps.
\end{align} 

Let \(r:\cG\to [0,\infty)\) be a graph parameter satifying conditions \ref{Cest1} and \ref{Cest2}. For $\eta>0$, define
\begin{align}
\mathcal G_{r,\eta}:=\{G \in \mathcal G: r(G) \geq \eta\}.
\end{align} 
We say that $f$ is \emph{$r$-estimable} if it is estimable on $\cG_{r,\eta}$ for every $\eta>0$.
\end{definition}

Two examples of a graph parameter $r$ satisfying the above are  subgraph
densities and the normalised minimum degree: recall $d_{\rm min}(G):=\min_{v\in V(G)}d_v$ and then set,
\begin{align}\label{de:norm-min-deg}
    \dnormmin (G):= \frac{ d_{\rm min}(G)}{|G|-1},
\end{align}
where we set $\dnormmin:=0$ for $|G|\leq 1$. In our applications to modularity and Cheeger constant, $r$ will be chosen
as the edge density and the normalised minimum degree, respectively.
For both choices, \ref{Cest1} is immediate;
the corresponding $\hat r(W)$ for graphons  are $\rho_W$ and 
$\essinf k_W$, see \eqref{che2}.
That subgraph densities meet condition \ref{Cest2} follows from
\cite[Theorems 2.7(a) and 2.9]{BCLSV1}, 
and for the normalised minimum degree, we verify \ref{Cest2} later in \refL{Lmin-degree-sampling} by a
concentration argument.

\subsection{Characterisations of restricted estimability}\label{app:testable}

The following proposition is an 
analogue of \cite[Theorem~6.1]{BCLSV1}. That result gives a list of
properties equivalent to the function $f$ being estimable over all graphs. 
We give the corresponding result for $f$ being $r$-estimable, using the following definition.

\begin{definition}\label{DWreta}
For $\eta\ge0$, let \(\cW_{r,\eta}\) be the subset of \(\cW\) 
consisting of all $W\in\cW$ such that there exists a sequence of finite graphs $(G_n)_{n\ge 1}$ such that
 $r(G_n)\ge \eta$ for all $n\ge 1$, and, as $n\toinf$, $|G_n|\toinf$ and
 $G_n\to W$. 

Furthermore, let $\cW_r=\bigcup_{\eta>0} \cW_{r,\eta}$.
\end{definition}

\begin{remark}\label{RWreta}
If $W\in\cW_{r,\eta}$, then, for a sequence $(G_n)_{n\ge1}$ as in the definition, 
 we have by \eqref{usc} and \eqref{r1}, since $\dcut(W_{G_n},W)\to0$,
\begin{align}\label{oct2}
  \hr(W)
\ge \limsup_\ntoo \hr(W_{G_n})
= \limsup_\ntoo r(G_n)
\ge\eta.
\end{align}
Conversely, for the examples above when $r$ is the edge density or 
normalised minimum degree, 
it is easily seen that for any graphon $W$ with $\hr(W)\ge\eta$,
there exists such a sequence of graphs; in other words,
$\cW_{r,\eta}=\set{W\in\cW:\hr(W)\ge\eta}$.
We think that this will be the case in all interesting examples,
although it does not follow from our definitions above, as can be seen by
simple counter examples.
See also \refL{Lhr} for a partial converse.\end{remark}

\begin{proposition}\label{prop:res-estimability}
    Let \(f:\cG\to [0,\infty)\) be a bounded graph parameter and let \(r:\cG\to [0,\infty)\) be a graph parameter satisfying \ref{Cest1} and \ref{Cest2}. Then the following statements are equivalent:
    \begin{enumerate}
 \renewcommand{\labelenumi}{\textup{(\alph{enumi})}}%
 \renewcommand{\theenumi}{\labelenumi}%

    \item $f$ is \(r\)-estimable.\label{esta}

    \item For every $\eps>0$ and $\eta>0$, there exists a positive integer $k:=k(\eps,\eta)$ such that for every $G\in\cG$ with $|G|\ge k$
    and $r(G)\geq \eta$, $|f(G)- \E[f(\mathbbm{G}(k,G))]|\leq \eps$.

    \item For every sequence $(G_n)_{n\ge 1}$ of simple graphs such that 
$\liminf_{n\toinf} r(G_n) >0$ and,
as $n\toinf$, 
$|G_n|\toinf$  and\/  $G_n\to W$ for some graphon $W$,
the limit  $\lim_{n\to\infty} f(G_n)$ exists.

\item There exists a functional $\check{f}:\mathcal{W}_{r}\to\IR$ with
      the following properties: 
(1) $\check{f}$ is continuous in the cut metric 
on each set $W_{r,\eta}$, $\eta>0$; and (2) if 
a sequence $(G_n)_{n\ge1}$ of graphs
satisfies $\liminf_{n\toinf} r(G_n) >0 $ and\/ $|G_n|\toinf$ as $n\toinf$, then $\check{f}(W_{G_n})-f(G_n)\to 0$ as $n\to\infty$.    
    \item For every $\eps>0$ and $\eta>0$, there are  $\eps'=\eps'(\eps,\eta)>0$ and $n_0=n_0(\eps,\eta)>0$ such that if $G_1, G_2$ are simple graphs such that (1) $|G_1|,|G_2|\geq n_0$; (2) $\dels(G_1,G_2)\leq\eps'$; and (3) $r(G_1),r(G_2)\ge \eta,$  then $|f(G_1)-f(G_2)|\leq\eps$.
 \end{enumerate}
\end{proposition}

The proof of this proposition follows in an easy way by generalising the proof in \cite{BCLSV1};
we give the proof in Section~\ref{subsec:proof_prop:estimability}.

\subsection{Applications}
We apply Proposition \ref{prop:res-estimability} above to Cheeger constant
and modularity, proving in particular that our convergence results in
Theorems \ref{thm:cheeger-conv} and \ref{thm:limit-identification} imply
that these graph parameters are $r$-estimable for suitably chosen $r$. We
further consider the application to the clustering coefficient, and
show that while these graph parameters are not estimable they are
$r$-estimable. To establish non-estimability, we show below that 
the equivalent condition (e) in \cite[Theorem 6.1]{BCLSV1} does not hold.

\subsubsection{Cheeger constant}
Recall the definition of the Cheeger constant $h(G)$ on a finite simple graph $G$ with no isolated vertices in Definition \ref{def:cheeger}, and $\dnormmin$ in \eqref{de:norm-min-deg}.

\begin{proposition}\label{prop:cheeger-estimation-alt}
For any $\eps, \eta>0$, $\exists k=k(\eps, \eta)$ such that for any graph
$G$ with minimum degree $ d_{\rm min}(G) \ge \eta(n-1)$ we have
\begin{align}\pr(|h(G)-h(\mathbbm{G}(k,G))| >\eps)<\eps,\end{align}
Moreover, the Cheeger constant is not estimable. 
\end{proposition}

The proposition follows directly from Theorem \ref{thm:cheeger-conv} and Proposition \ref{prop:res-estimability} once we verify that \ref{Cest2} holds for a suitably defined $r(G)$ which ensures the minimum degree of the graph is $\eta n$. We first show that the minimum degree of a sampled graph is also sufficiently high then prove the proposition.

\begin{lemma}\label{Lmin-degree-sampling}
    For every \(\eta>0\) and \(\eps>0\), there exists 
\(k_1=k_1(\eps)>0\)
such that for every \(G\in \cG\) and integer $k$
such that $|G|\ge k\ge k_1$ and $\dnormmin(G)\ge\eta$, 
we have
    \begin{align}\label{eh0}
\P\Bigpar{\dnormmin(\mathbbm{G}(k,G))\le \eta-\eps}
\le e^{-\eps^2 k}.
    \end{align}
\end{lemma}

\begin{proof}
Let $v\in V(G)$.
Conditionally on $v\in V(\mathbbm{G}(k,G))$,  
the distribution of $ d_{\mathbbm{G}(k,G)}(v)$ 
is hypergeometric and equals the distribution of the number of black balls in $k-1$ drawings without replacement from an urn
with initially $d_G(v)$ black balls and $|G|-1-d_G(v)$ white balls. 
Let
\begin{align}
\Ev [ d_{\mathbbm{G}(k,G)}(v)]:=
  \E [ d_{\mathbbm{G}(k,G)}(v)\mid v\in V(\mathbbm{G}(k,G)) ] 
\end{align}
denote the conditional expectation.
Then,
 by Hoeffding's concentration inequality for sampling without replacement
 \cite{Hoeff63}, 
see also \cite[Theorem 2.10]{JLR},
\begin{align}\label{eh1}
     \IP\Bigpar{d_{\mathbbm{G}(k,G)}(v) 
- {\Ev} [d_{\mathbbm{G}(k,G)}(v)] \le -t \mid v\in V(\mathbbm{G}(k,G))} 
\le e^{-2t^2/(k-1)}, 
\quad t>0.
\end{align}
We have $\Ev [d_{\mathbbm{G}(k,G)}(v)] = (k-1) d_G(v)/(|G|-1)
\ge (k-1)\eta$, and thus \eqref{eh1} yields
\begin{align}\label{Ehoeff}
    &\IP\Bigpar{d_{\mathbbm{G}(k,G)}(v) \le (k-1)
(\eta-\eps)  \mid v\in V(\mathbbm{G}(k,G)) }
\notag \\&\qquad
    \le \IP\Bigpar{d_{\mathbbm{G}(k,G)}(v) - 
\Ev [ d_{\mathbbm{G}(k,G)}(v)]
  \le  - (k-1)\eps \mid v\in V(\mathbbm{G}(k,G)) }
\notag \\&\qquad
    \le e^{-2(k-1)\eps^2}
.\end{align}
Since $\IP(v\in V(\mathbbm{G}(k,G)))=k/|G|$, it follows from a union bound and \eqref{Ehoeff} that
    \begin{align}
        &\IP\Bigpar{ \min_{v\in \mathbbm{G}(k,G) }d_{\mathbbm{G}(k,G)}(v) 
         \le (k-1)}(\eta-\eps)
\notag \\&\qquad
         \le \sum_{v\in V(G)} \IP(v\in V(\mathbbm{G}(k,G)))
\IP\Bigpar{d_{\mathbbm{G}(k,G)}(v) \le (k-1)(\eta-\eps)\mid v\in V(\mathbbm{G}(k,G))}
\notag \\ &\qquad
\le  k e^{-2(k-1)\eps^2},
    \end{align}
which implies \eqref{eh0} for $k\ge k_1$ if $k_1=k_1(\eps)$ is large enough.
\end{proof}

\begin{proof}[Proof of Proposition \ref{prop:cheeger-estimation-alt}]
    We show that the Cheeger constant is $\dnormmin$-estimable. By Lemma \ref{Lmin-degree-sampling}, $\dnormmin$ satisfies condition \ref{Cest2}, while \ref{Cest1} is immediate. 
    The proposition therefore follows from Theorem \ref{thm:cheeger-conv}, since we have shown the equivalence between (c) and (a) in Proposition \ref{prop:res-estimability}.

    The following example, which is essentially due to \cite[Section 5.1]{Khetan},
    demonstrates that the Cheeger constant is not estimable. Let $G_n$ be a complete bipartite graph, where each of the parts $V_{n,1}$ and $V_{n,2}$ has $n$ vertices, with $n$ even. Let $G_n'$ be a simple graph obtained by appending a clique $K_{\kappa_n}$ to $G_n$, where 
$\kappa_n \to \infty$ and $\kappa_n=o(n)$, and joining  $K_{\kappa_n}$ to $V_{n,1}$ by a single edge. Thus, for all sufficiently large $n$, the minimum degrees in $G_n$ and $G_n'$ are respectively $n$ and $\kappa_n-1$, while the Cheeger constants are $h(G_n)=1/2$ (see \cite{Khetan}) and
    \begin{align}
        h(G_n') &
= \frac{e_{G'_n}(V_{n,1}\cup V_{n,2}, K_{\kappa_n})}{\vol_{G'_n}(K_{\kappa_n})} = \frac{1}{\kappa_n(\kappa_n-1) +1 } \to 0,\quad n\toinf. 
    \end{align}
    Moreover, $G_n\to W$ and $G_n'\to W$ as $n\toinf$, where $W$ is the step graphon with $W(x,y)=1$ for $0\le x < 1/2 \le y\le 1$ and $0\le y < 1/2\le x \le 1$, and $W(x,y)=0$ otherwise. Thus, as $n\toinf$, $\dels(G_n,G_n')\to 0$ but $|h(G_n)-h(G_n')|\to 1/2$. 
Hence, \cite[Theorem 6.1(e)]{BCLSV1} shows that
$h$ is not estimable.
(This counterexample also shows that the condition $r(G_1),r(G_2)\ge \eta$
in \refProp{prop:res-estimability}(e) cannot be weakened to $r(G_j)\ge
\gamma(|G_j|$ for any function $\gamma(n)=o(n)$.)
\end{proof}

\subsubsection{Clustering coefficient}
\label{subsec.cc}
 We apply Proposition \ref{prop:res-estimability} to clustering coefficient, which is an important graph parameter in the study of networks; see~\cite{van2020limit} and references therein for details. For a graph $G$, with $\tinj(\cherry, G)>0$, the clustering coefficient is defined to be
 \begin{equation}\label{eq.cc_defn}
 \cc(G):= \frac{\tinj(\trian, G)}{\tinj(\cherry, G)} = \frac{\sum_{u,v,w} A_{uv}A_{vw}A_{uw}}{\sum_u d_u(d_u-1)}
 \end{equation}
where the sum is over tuples of distinct $i,j,k$.

Let \(\rho:\cG\to [0,\infty)\) be the edge density. Below we show that
clustering coefficient is $\rho$-estimable but not estimable, but before
that, we recall that it was noted above that edge density $\rho$ (as any subgraph
  density) satisfies \ref{Cest1} (immediately by the definitions),
and \ref{Cest2} by 
\cite[Theorems 2.7(a) and 2.9]{BCLSV1}.

\begin{proposition}\label{prop.cc_dense_est}
 The clustering coefficient is not estimable but \(\rho\)-estimable, i.e.\ for all $\eps, \rho>0$, there is a positive integer $k:=k(\eps, \rho)$ such that for every simple graph $G$ with $|G|\ge k$ and $\rho_G\geq \rho$,
    \begin{equation}
    \IP\big(|\cc(G)- \cc(\mathbbm{G}(k,G))|>\eps \big) < \eps.
    \end{equation}
\end{proposition}

\begin{proof}
Suppose for notational convenience that the graph $G_n$ has $n$ vertices,
and let its density be $\rho_n$. 
Then $\tinj(\cherry,G_n)\geq \rho_n^2 +O(1/n)$. To see this, note that 
\begin{align} \tinj(\cherry, G_n) = \sum_{u}\frac{d_u(d_u-1)}{n(n-1)(n-2)},\end{align}
and also note 
that by Cauchy--Schwarz,
\begin{align}\sum_u d_u^2 \geq \frac{1}{n}(\sum_u d_u)^2 = \frac{1}{n}(2e(G))^2=\rho_n^2n(n-1)^2.\end{align}
Hence, since $2e(G_n)/(n(n-1)(n-2))=\rho_n/(n-2)$, we have
\begin{align} \tinj(\cherry, G_n) \geq \rho_n^2\frac{n-1}{n-2} - \rho_n\frac{1}{n-2}
= \rho_n^2 - \rho_n(1-\rho_n)\frac{1}{n-2}
=\rho_n^2 + O(1/n).
\end{align}

For any sequence $G_n\rightarrow W$ 
(with $|G_n|=n$), where $W$ has density $\rho>0$, there exists $n>n_0$ such that $\rho_n>\rho/2$ for all $n>n_0$. Hence, it follows from the computation
above that the clustering coefficient is defined for all $G_n$ with
$n>n_1$, for some $n_1\ge n_0$. 
Also, since $G_n\rightarrow W$, we have that
$\tinj(\cherry,G_n)\rightarrow t(\cherry,W)$ and
$\tinj(\trian,G_n)\rightarrow t(\trian, W)$. Thus the limit
$\lim_{n\rightarrow\infty} \cc(G_n)$ exists. This confirms part~(c) of
Proposition~\ref{prop:res-estimability}; and hence $\cc(G)$ is
  \(\rho\)-estimable.

 To see that $\cc(G)$ is not estimable in general, note the following example. Take $G_n$ to be the graph consisting of a single triangle $\trian$ and $n-3$ isolated vertices. Then $G_n$ converges to the zero graphon and $\cc(G_n)\rightarrow 1$. On the other hand, let $G_n'$ be the graph consisting of a single cherry $\cherry$ and $n-3$ isolated vertices. Then $G_n'$ also converges to the zero graphon but $\cc(G_n)\rightarrow 0$. Thus as $n\toinf$, $\dels(G_n, G_n')\to 0$ but $|\cc(G_n)-\cc(G'_n)|=1$.
\end{proof}

Given the importance of clustering coefficient in the study of networks, we also emphasise the following sampling result, which is an immediate consequence of Propositions~\ref{prop:res-estimability} and \ref{prop.cc_dense_est}.

\begin{corollary}   
For all $\eps, \rho>0$, there is a positive integer $k:=k(\eps, { \rho})$ such that for every simple graph $G$ with $|G|\ge k$
    and $\rho_G\geq \rho$, $|\cc(G)- \IE[\cc(\mathbbm{G}(k,G))]|\leq \eps$.
\end{corollary}

\subsubsection{Modularity}

Recall the definition of modularity $q^*$ in \refD{D1}. Using Proposition
\ref{prop:res-estimability}, below we show that $\rho$-estimability of
modularity is a direct consequence of Theorem
\ref{thm:limit-identification}.

\begin{theorem}\label{thm.modest}
For all $\eps, \rho>0$, there exists a positive integer $k:=k(\eps, \rho)$ such that for every $G\in \cG$ with $|G|\ge k$ and $\rho_G\geq \rho$,
    \begin{align}
    \IP\big(|q^*(G)- q^*(\mathbbm{G}(k,G))|>\eps \big) < \eps,
    \end{align}
i.e., modularity is $\rho$-estimable.
 \end{theorem}

\begin{proof}
    Noting that modularity is a bounded graph parameter, that it is $\rho$-estimable follows from Theorem \ref{thm:limit-identification}, which shows that (c) in Proposition \ref{prop:res-estimability} holds. An example showing that modularity is not estimable is given in~\cite{sampling}. Alternatively, this follows from \refR{Rnot0}, since if modularity were estimable, the corresponding functional for graphons would be continuous everywhere on $\cWo$ as a consequence of~\cite[Theorems 6.1(d)]{BCLSV1}, but this functional would be our graphon modularity $\hqx$ by \refT{thm:limit-identification}, a contradiction. 
\end{proof}

\subsection{Proof of Proposition~\ref{prop:res-estimability}.}
\label{subsec:proof_prop:estimability}

We  begin with a lemma, mentioned in \refR{RWreta}.

\begin{lemma}\label{Lhr}
Let\/ $\hr$ be as in \ref{Cest1}. 
If\/ $G$ is a graph such that $\hr(W_G)>\eta>0$, then 
$W_G\in\cW_{r,\eta}$.
\end{lemma}
\begin{proof}
Consider the blow-ups $G[b]$ of $G$, and let $b\to\infty$.
We have 
$|G[b]|\to\infty$ as $b\to\infty$, and 
$W_{G[b]}=W_G$ for every $b\ge1$;
hence \eqref{r1} implies $r(G[b])\to\hr(W_G)>\eta$. 
Consequently, $r(G[b])\ge\eta$ for all large $b$, so by omitting a finite number
of the blow-ups, we obtain a sequence $(G[b])_{b\ge b_0}$ that (after
renumbering) is as in \refD{DWreta}, 
and demonstrates that $W_G\in\cW_{r,\eta}$. 
\end{proof}

\begin{proof}[Proof of Proposition~\ref{prop:res-estimability}.]
We follow the proof of \cite[Theorem~6.1]{BCLSV1}, with some modifications.
We start by showing that 
(a)$\implies$(b)$\implies$(c)$\implies$(e)$\implies$(a).

 (a)$\implies$(b). 
Let $Y:=f(G)-f(\mathbbm{G}(k,G))$. 
Suppose that $|f|\leq b/2$, so that $|Y|\leq b$. 
Let $ \eps'=\eps/(1+b)$ and let $k:=k(\eps',\eta)$ be as in Definition
\ref{de:est1}. 
Then, by an easy calculation using (a),
\begin{align}
|\E Y|\le
\E |Y|\le \eps' \P[|Y|\le\eps'] + b\P[|Y|> \eps'] \leq  \eps'(1+b)=\eps.  
\end{align}
 
 (b)$\implies$(c). 
Let $(G_n)_{n\ge 1}$ be as in (c). 
Then there exist $n_0>0$ and $\eta>0$
such that for $n\ge n_0$, 
$r(G_n)\ge \eta$. 
Given $\eps>0$, choose $k=k(\eps/4,\eta)$ such that for every $G$ with
$r(G)\geq \eta$ and $|G|\geq k$, 
we have $|f(G)-\IE[f(\mathbbm{G}(k,G))]|\leq \eps/4$.  
Then choose $n_1\ge n_0$ such that
for $n\ge n_1$, we have $|G_n|\ge k$, and thus
by (b)
\begin{align}\label{oct1}
|f(G_n)-\E[f(\mathbbm{G}(k,G_n))]|\leq \eps/4.  
\end{align}
Let $F$ be a simple graph on $k$ vertices. 
Since $|G_n|\to\infty$,
the assumption $G_n\to W$ as \ntoo{} means that 
$\tind(F,G_n)\to\tind(F,W)$, 
see \cite{BCLSV1} and \cite{Lovasz}
(and our generalisations to coloured graphs in \refL{L1} and \eqref{tw3}).
Thus as $n\toinf$, recalling \eqref{tind} (for the uncoloured case $\kk=1$),
$\IP[\mathbbm{G}(k,G_n)=F]=t_{\mathrm{ind}}(F,G_n)\to t_{\mathrm{ind}}(F,W)$
and, summing over all graphs $F\in\cG$ with $k$ vertices, 
 \begin{align}\label{eq:limf}
    \IE [f(\mathbbm{G}(k,G_n))]
=\sum_{|F|=k} \P[\mathbbm{G}(k,G_n)=F]f(F) 
\to \sum_{|F|=k} t_{\mathrm{ind}}(F,W)f(F)=:a_k. 
\end{align} 
Thus, we may
choose $n_2\ge n_1$ large enough so that  $| \IE
[f(\mathbbm{G}(k,G_n))]-a_k|\leq \eps/4$ for $n\ge n_2$. Then by 
\eqref{oct1} and 
the triangle inequality, $|f(G_n)-a_k|\le \eps/2$ for $n\ge n_2$,
and so $|f(G_n)-f(G_{n'})|\le \eps$ for $n,n'\ge n_2$. 

This shows that the sequence $(f(G_n))_{n\ge1}$, is a Cauchy sequence, and
therefore it converges.

(c)$\implies$(e). 
Suppose that (c) is valid but (e) does not hold. 
Then there exist $\eps>0$ and $\eta>0$ such that for every $n\ge 1$, there are
$G_n,G'_{n}$ such that $|G_n|, |G'_{n}|\ge n$, 
$\dels(G_n,G'_n)\le 1/n$, and
$r(G_n), r(G'_{n})\ge\eta$, 
 but $|f(G_n)-f(G'_n)|>\eps$. By compactness,
there exists a subsequence $(n_i)_{i\ge 1}$ such that   as $i\toinf$,
$n_i\toinf$ and $G_{n_i}\to W$ for some graphon $W$.
So by the triangle inequality, $\dels(G'_{n_i},W)\le n_i^{-1} +
\dels(G_{n_i},W)\to 0$ as $i\toinf$. The sequence $(G_{m_1}, G'_{m_1},
G_{m_2}, G'_{m_2},\dots)$ thus converges to $W$, with  $\liminf_{i\to\infty}
r(G_{n_i})\ge\eta >0$. 
By (c), $|f(G_{n_i})-f(G'_{n_i})|\to 0$ as $i\toinf$, giving
the desired contradiction. 

(e)$\implies$(a). 
Given $\eps>0$ and $\eta>0$, let
$\eps'=\eps'(\eps,\eta/2)$ and $n_0=n_0(\eps,\eta/2)$ be as in (e).  Choose
$k\ge n_0$ large enough so that  
$10/\sqrt{\log_2 k}\le \eps'$ and $\exp\{-k^2/2\log_2 k\}\le \eps/2$. 
Then by
\cite[Theorem 2.9]{BCLSV1}, for $G$ such that $|G|\ge k$, we have
\\
$\dels(G,\mathbbm{G}(k,G)))\le \eps'$ with probability at least
$1-\eps/2$. 

Furthermore, let $k_1=k_1(\min\set{\eps/2,\eta/2},\eta)$ be as in
\ref{Cest2},
and assume, by increasing $k$ if necessary, that $k\ge k_1$.
Then,
if $|G|\ge k$ and $r(G)\ge \eta$,  
by \ref{Cest2} we have
$r(\mathbbm{G}(k,G))\ge\eta/2$ with probability at least $1-\eps/2$.

Let $G$ be a graph with $r(G)\ge\eta$ and $|G|\ge k$. Then, taking $G_1=G$ and $G_2=\mathbbm{G}(k,G)$ in (e), 
we see that with probability at least $1-\eps$,
$G_1$ and $G_2$ satisfy conditions (1), (2) and (3) in (e) (with $\eta/2$), and thus
$|f(G_1)-f(G_2)|\le \eps$. In other words, $f$ is $r$-estimable. 
 
\smallskip
The full equivalence of the conditions now follows from showing that  
(e)$\implies$(d)$\implies$(c).

(e)$\implies$(d). 
For any graphon $W\in \mathcal{W}_{r}$,
let $\eta>0$ be such that $W\in\cW_{r,\eta}$ and choose 
a sequence of graphs $(G_n)_{n\ge1}$
as in \refD{DWreta}.
Then, define $\check{f}(W):=\lim_{n\to\infty}f(G_n)$; the limit exists by (c) (which follows from (e) by the first part of the proof),
and it follows from (e) that if $(G'_n)_{n\ge1}$ is another such sequence of
graphs, then $f(G_n)-f(G_n')\to0$; 
hence the limit does not depend on the choice of $\eta$ and $G_n$. 

To prove the continuity of $\check{f}(W)$
on $\cW_{r,\eta}$, where $\eta>0$,
let $\eps>0$ and let $\eps'=\eps'(\eps,\eta)$ 
be as in (e).
Let $W,W'\in\cW_{r,\eta}$ and assume that $\dcut(W,W')<\eps'$.
Choose sequences $(G_n)_{n\ge1}$ and $(G'_n)_{n\ge1}$ of graphs 
as in \refD{DWreta} with
$G_n\to W$ and $G_n'\to W'$ as \ntoo.
Then $f(G_n)\to f(W)$ and $f(G_n')\to f(W')$ by the construction just made.
Moreover, 
\begin{align}
  \dcut(G_n,G_n')
\le
  \dcut(G_n,W)+  \dcut(W,W') +   \dcut(W',G_n')
\to \dcut(W,W')<\eps',
\end{align}
and thus $\dcut(G_n,G_n')<\eps'$ for all large $n$. 
Hence, (e) implies that 
$|f(G_n)-f(G_n')|<\eps$ for all large $n$, which implies
\begin{align}
|f(W)-f(W')|=
  \lim_\ntoo|f(G_n)-f(G_n')|\le\eps,
\end{align}
which proves the continuity of $\check{f}$ in the cut metric on $\cW_{r,\eta}$.

For the second part in (d), 
suppose that there exists a sequence of simple graphs $(G_n)_{n \geq 1}$
with $|G_n| \to \infty$, $\liminf_{n\toinf} r(G_n) >0$ and $|\check
f(W_{G_n}) - f(G_n)| \not \to 0$. Then there exist $\eta>0$, $\eps >0$ and a
subsequence of $(G_n)_{n\ge 1}$ such that (after relabelling) $r(G_{n})>
\eta$ and $|f(G_{n})- \check f(W_{G_{n}})|\ge \eps$ for every $n$. By
compactness of
$\cW$, we may by selecting a further subsequence assume that
$G_{n}\to W$ for some graphon $W$, 
and it follows from \refD{DWreta} that $W\in\cW_{r,\eta}$;
furthermore, by the definition of $\check f$,
\begin{align}\label{fn2}
\check f(W)=\lim_\ntoo f(G_n).
\end{align}
Moreover, 
by \eqref{r1},  for all large $n$ we have
$\hr(W_{G_n}) > r(G_n) -\eta/2>\eta/2$,
and thus \refL{Lhr} shows that $W_{G_n}\in\cW_{r,\eta/2}$.
Since also $W_{G_n}\to W$, the continuity of $\check f$ on $\cW_{r,\eta/2}$, which was shown above, implies
\begin{align}\label{fn3}
  \check{f}(W_{G_n})\to\check{f}(W).
\end{align}
By comparing \eqref{fn2} and \eqref{fn3}, we find
$\check f(W_{G_{n}})- f(G_{n})\to 0$ as $n\toinf$ -- a
contradiction 
to our assumption
$|f(G_{n})- \check f(W_{G_{n}})|\ge \eps$.

(d)$\implies$(c). Suppose that the sequence of simple graphs $(G_n)_{n\ge1}$ 
is such that $r(G_n)\ge \eta$ for some $\eta>0$ and as $n\toinf$,
$|G_n|\toinf$ and $G_n\to W$ for some 
$W\in\cW$.
Then $W\in \mathcal{W}_{r,\eta}$ by the definition of the latter.
Moreover, it follows from \eqref{r1} that, for all large $n$,
$\hr(W_{G_n}) > \eta/2$, and thus by \refL{Lhr} we have
$W_{G_n}\in\cW_{r,\eta/2}$. 
Since $\check{f}$ is continuous on $\cW_{r,\eta/2}$ by (d), and
we have $\dels(G_n,W)\to 0$, it follows that
$|\check{f}(W_{G_n})-\check{f}(W)|\to 0$. 
Furthermore, by (d),
$\check{f}(W_{G_n})-f(G_n)\to 0$ as
$n\to\infty$. By the triangle inequality, we have $f(G_n)-\check{f}(W)\to 0$
as $n\to\infty$, showing that $\lim_{n\to\infty}f(G_n)$ exists.  
\end{proof}

\needspace{6\baselineskip}


\section*{Acknowledgements}
SJ was supported by the Knut and Alice Wallenberg Foundation and the Swedish
Research Council. 
TL was supported by the Sverker Lerheden Foundation. 
AS and FS were supported by the Wallenberg AI, Autonomous Systems and Software Program WASP. This research was partly done at
the Simons Laufer Mathematical Sciences Institute in Berkeley, California,
where two authors participated in the program 
Probability and Statistics of Discrete Structures, 2025, supported by the
National Science Foundation Grant No. DMS-1928930.

\newpage
\appendix

\section{Uncoloured graphons} 
\label{app:uncoloured-pullback}
We recall some standard facts for uncoloured graphons, see e.g.\
\cite{Lovasz}, \cite{SJ249}.
Let $(\gO_1,\cF_1,\mu_1)$ and $(\gO_2,\cF_2,\mu_2)$ be two probability
spaces. We will often abuse notation and denote these by just $\gO_1$ and
$\gO_2$. 

A mapping $\gf:\gO_1\to\gO_2$ is  \emph{measure-preserving} if it is
measurable and 
$\mu_1(\gf\qw(A))=\mu_2(A)$ for every $A\in\cF_2$ (i.e., for every measurable
$A\subseteq\gO_2$). 

We define, for any functions $f$ on $\gO_2$ and $W$ on $\gO_2^2$,
the \emph{pullbacks}
\begin{align}
  f^\gf(x)&:=f(\gf(x)), \\
 W^\gf(x,y)&:= W(\gf(x),\gf(y)); \label{Wpull}
\end{align}
these are functions on $\gO_1$ and $\gO_1^2$, respectively.

 Consider next graphons on a fixed probability space $\OFmu$, and 
let $\cWo=\cWo\OFmu$ be the set of all uncoloured graphons on $\OFmu$.
Recall that the cut norm $\cnorm{\cdot}$ is defined for integrable functions 
on $\gO^2$ by
\begin{equation}\label{cutnorm1}
 \cnorm{F} := \sup_{S,T}  \Bigabs{\int_{S \times T} F(x,y) \dd\mu(x)\dd\mu(y)},
\end{equation}
where the supremum is taken over all pairs of measurable subsets of $\gO$.
As is well known, we also have
\begin{equation}\label{cutnorm2}
 \cnorm{ F} := \sup_{0\le f,g\le1}
  \Bigabs{\int_{\gO^2} F(x,y)f(x)g(y)\dd\mu(x)\dd\mu(y)},
\end{equation}
taking the supremum over all measurable functions $f,g:\gO\to\oi$.
(There are also other versions of the cut norm that are not identical but
equivalent within constant factors, see \cite[Appendix E]{SJ249}.)
Then $\cnorm{W_1-W_2}$ defines a metric, and thus a topology, on $\cWo$ 
(modulo a.e.\  equivalence).

For two uncoloured
graphons $W_1$ and $W_2$ defined on arbitrary, possibly different,
probability spaces $\gO_1$ and $\gO_2$,  the \emph{cut distance} 
is defined by
\begin{equation}\label{dcut}
 \dcut(W_1,W_2) = \inf_{\gf_1,\gf_2} \cnorm{W_1^{\gf_1} - W_2^{\gf_2}}, 
\end{equation}
where the infimum is taken over all 
probability spaces $\OFmu$ and \mpp{} maps $\gf_1:\gO\to\gO_1$
and $\gf_2:\gO\to\gO_2$,
and $W_i^{\gf_i}$ is the pullback defined in 
\eqref{Wpull};
see 
\cite[Lemma 3.5]{BCLSV1},
\cite[Theorem 8.13]{Lovasz} (both for $\gO=\oi$), and \cite[Theorem 6.9]{SJ249}
for several equivalent versions.
In particular, we use that if $W_1$ and $W_2$ are graphons defined on $\oi$,
then
\begin{align}\label{dcut1}
 \dcut(W_1,W_2) = \inf_{\gf} \cnorm{W_1^{\gf} - W_2},   
\end{align}
taking the infimum over all \mpp{} bijections $\gf:\oi\to\oi$.
It is well known that 
\begin{align}\label{dcut0}
\dcut(W_1,W_2)=0\iff W_1\cong W_2.
\end{align}
It follows that $\dcut$ is well defined also for graph limits; furthermore
$\dcut$ is a metric on the space of graph limits which defines the same 
compact topology as defined in Section \ref{sec:coloured-graphons}
(see e.g.\ \eqref{tw9})
for the uncoloured case $\kk=1$. 

It follows immediately from \eqref{dcut} that for any pullback $W^\gf$ of a
graphon $W$ we have $\dcut(W,W^\gf)=0$ and thus $W\cong W^\gf$.
The converse does not hold, but we have the following deep substitutes by 
\cite{BCL:unique} and \cite{BRmetrics};
see also \cite[Corollary 10.35 and Theorem 13.10]{Lovasz} 
and \cite[Theorems 8.3 and 8.6]{SJ249} for further related results.
\begin{theorem}\label{Tunique0}
  Let\/ $W_1$ and $W_2$ be uncoloured graphons defined on probability spaces
  $\gO_1$ and $\gO_2$ such that $W_1\cong W_2$. 
  \begin{romenumerate}
    
  \item \label{Tunique0BCL}(\cite{BCL:unique})
Then there exists a graphon $W$ defined on a probability space $\OFmu$ and
\mpp{} maps $\gf_j:\gO_j\to\gO$ such that $W_j=W^{\gf_j}$ a.e., $j=1,2$.

\item \label{Tunique0BR}(\cite{BRmetrics})
Suppose also that $\gO_1=\gO_2=\oi$. 
Then there exist \mpp{} maps $\gf_1,\gf_2:\oi\to\oi$ such that
$W_1^{\gf_1}=W_2^{\gf_2}$ a.e.
  \end{romenumerate}
\end{theorem}

\bigskip
\bigskip

\section{Notation Table}\label{app:notation-table}
\begin{center}
\renewcommand{\arraystretch}{1.15}
\small
\begin{longtable}{>{\raggedright\arraybackslash}p{0.3\textwidth} >{\raggedright\arraybackslash}p{0.80\textwidth}}      
\midrule
\bottomrule
\endlastfoot
\multicolumn{2}{l}{\textbf{General}}\\ 

$[n]$ & $\set{1,\dots,n}$ if $n\in\bbN$; $[\infty]:=\bbN$.\\
$[n]\pairs$ & The set of unordered pairs $\set{i,j}$ with $i\neq j\in[n]$.\\
$(\Omega,\mathcal F,\mu)$ & Underlying probability space.\\
$\kappa$ & Number of colours.\\
$\gl$ & Lebesgue measure.\\
$\innprod{v,w}$ & inner product of vectors in $\bbR^\kk$.\\
$\indic{A}$ or $\sindic{A}$ & The indicator of an event (or a set) $A$.\\

\addlinespace
\multicolumn{2}{l}{\textbf{Basic graph notation}}\\

$G,G_n$ & Finite simple graphs; $(G_n)$  denotes a graph sequence.\\
$V(G), E(G)$ & Vertex set, edge set of $G$.\\
$d_u$, $d_u(G)$ & Degree of vertex $u$ in a finite graph $G$.\\
$d_{\text{min}}, \ d_{\text{max}}$ & Minimum and maximum degrees.\\
$A_{u,v}$ & Adjacency matrix.\\
$e_G(S,T)$ & Number of ordered pairs $(i,j)\in S\times T$ such that $ij\in E(G)$.\\
$e_{\cA}(G)$ & Edges inside partition $\cA$ in $G$.\\
$\vol(A)$ & Volume of $A$, i.e.\ $\sum_{u\in A} d_u$.\\
$\rho_G$ & Edge density of a graph:
$\rho_G=e(G)\binom{|G|}{2}^{-1}$.\\
$G[b]$ & $b$-blow-up of $G$.\\
$\Gnp$ & Erd\H{o}s-R\'enyi random graph on $n$ vertices, with edge probability $p$. \\

\addlinespace
\multicolumn{2}{l}{\textbf{Coloured graphs and spaces of coloured graphs}}\\ 

$\cL_n$ & The finite set of all $\kk$-coloured graphs with vertex set
$[n]$.\\

$\cL$ & The countable set of all finite labelled $\kk$-coloured graphs.\\

$\cL^{\fra}_n$ & The set of all fractionally $\kk$-coloured graphs with vertex set
$[n]$. \\

$\cU_n$ & The set of all unlabelled $\kk$-coloured graphs with 
$n$ vertices.\\

$\cU$ & The countable set of all
finite unlabelled $\kk$-coloured graphs.\\

$\KK_{2;\ell\glm}$ & The coloured graph with 
one edge and two vertices coloured $\ell$ and $\glm$.\\
$\GGx(n,\WW)$, $\GG^{\fra}(n,\WW)$ & Pure coloured $\WW$-random graph and fractional version -- see page~\pageref{def.WW-random-graph}. \\
\addlinespace
\multicolumn{2}{l}{\textbf{Graph distances and norms}}\\
$d_\square(G,G')$ & Cut distance between two labelled graphs on the same vertex set.\\
 $\delta_\square(W_1,W_2)$\  & Cut metric / cut distance notation used for graphons. \\
$\|W\|_\square$ & Cut norm of a graphon or integrable kernel.\\
$\|\cdot\|_{L^1}$ & $L^1$ norm.\\
$\|\cdot\|_\infty$ & $L^\infty$ norm / sup norm.\\

\addlinespace

\multicolumn{2}{l}{\textbf{Graphons and graphon quantities}}\\
$W$ & A graphon, i.e.\ a symmetric measurable function taking values in $[0,1]$.\\
$W_G$, $W'_G$ & Graphons associated with a finite graph $G$, see Ex.~\ref{EG1}.\\
$W^\varphi$ & Pullback of a graphon $W$ under a measure-preserving map $\varphi$.\\
$\mathcal W$ & Space of graphons.\\
$k(x)$ & Degree function of a graphon:
$k(x)=\int W(x,y)\,dy$ -- see~\eqref{m1}.\\
$\rho_W$ & Edge density of a graphon: $\rho_W
=\iint_{\gO^2}W(x,y)\dd\mu(x)\dd\mu(y)$
-- see~\eqref{m2}\\
$\chi$ &Pure colouring: a measurable mapping $\gO\to[\kk]$ -- see Defn~\ref{def:pure-mixed}. \\
$\Psi$ & Mixed colouring: a measurable mapping $\gO\to\gD_\kk$ -- see Defn~\ref{def:pure-mixed}.\\
$\gD_\kappa$ & Set of probability measures on $[\kappa]$ -- see Defn~\ref{def:pure-mixed}.\\

\addlinespace

\multicolumn{2}{l}{\textbf{Finite-graph   maxcut, min $k$-section, Cheeger constant and modularity}}\\
$\mathcal A$ & A partition of the vertex set of a graph.\\
$\Phi=(\Phi_1,\dots,\Phi_\kappa)$ & Fractional $\kappa$-partition of $V(G)$.\\
$C_k(G, \chi)$, $C_k(G, \Phi)$ & Colouring score or cut-value of $\chi$, $\Phi$ -- see~\eqref{MC2G}, \eqref{MCGx}. \\
$\MC(G), \MC_\fra(G)$ &  Maxcut, and fractional maxcut of a graph -- see~\eqref{MC1pG},\eqref{MCGmx}.\\
$\MS_k(G)$ & Min $k$-section of a graph -- see~\eqref{mink}.\\
$h(G)$ & Cheeger constant of $G$ -- see~\eqref{che-G}\\
$q(G,{\mathcal A})$
& Modularity score of partition $\mathcal A$ on graph $G$ -- see Def~\ref{D1a}.\\
$q^*(G)$ & Modularity of $G$; max of $q(G,\cA)$ over all $\cA$ -- see Def~\ref{D1a}.\\
$q(G,\Phi)$ & Fractional modularity score of $\Phi$ on $G$ -- see~\eqref{def:fractional-modularity}.\\
$q_\fra^*(G)$ & Fractional modularity of $G$; max of $q(G,\Phi)$ over all $\Phi$
-- see Def~\ref{def:fractional-modularity}.\\

\addlinespace

\multicolumn{2}{l}{\textbf{Graphon   maxcut, min $k$-section, Cheeger constant and modularity}}\\
$C_\kappa(W, \chi)$, $C_\kappa(W, \Phi)$ & Colouring score or cut-value of $\chi$, $\Phi$ -- see~\eqref{MC2W}, \eqref{MCWx2}. \\
$\MC_\pure(W)$ & Pure maxcut of a graphon $W$; supremum over $\C_2(W,\chi)$ -- see~\eqref{MC1pW}.\\ 
$\MC(W)$, $\MC_\mix(W)$ & Maxcut of a graphon $W$; supremum over $\C_2(W,\Phi)$ -- see~\eqref{MCW0},~\eqref{MCWmx}.\\ 
$\MS_k(W)$, $\MS_{k,\mathrm{mix}}(W)$ & Mixed min $k$-section of a graphon  -- see~\eqref{msk22},~\eqref{msk21}\\
$\MS_{k,\text{pure}}(W)$ & Pure min $k$-section of a graphon -- see~\eqref{MS2}.\\
$h(W)$, $h_{\mathrm{mix}}(W)$
& Mixed graphon Cheeger constant; infimum over admissible
mixed two-colourings -- see~\eqref{che3}.\\
$h_{\mathrm{pure}}(W)$
& Pure graphon Cheeger constant; infimum over admissible
pure two-colourings -- see just below~\eqref{che3}.\\
$\hat q(W,\chi)$ & Pure graphon modularity score under a pure colouring $\chi$ -- see~\eqref{m3pure}\\
$\hat q(W,\Psi)$ & Mixed graphon modularity score under a mixed colouring $\Psi$ -- see~\eqref{m3}.\\
$\hat q^*_{\mathrm{pure}}(W)$ & Optimal pure graphon modularity -- see~\eqref{mod4}.\\
$\hqx(W)$, $\qhatstarmix(W)$ & Optimal mixed graphon modularity -- see~\eqref{mod3}.\\
$\hat q_\kappa^*(W)$, $\hqx_{\kappa,\mathrm{pure}}(W)$ & Optimal mixed, pure graphon modularities with at most $\kappa$ colours/communities -- see \eqref{mod2},\eqref{mod4}.\\

\addlinespace
\multicolumn{2}{l}{\textbf{Restricted estimability}}\\  
$\mathbbm{G}(k,G)$ & Random induced subgraph of $G$ on $k$ vertices -- see page~\pageref{def.G_vertex_sampling}.\\
$r$, $\hat r$
& Restricting graph parameter and its graphon counterpart.\\
$\cG_{r,\eta}$
& Graphs satisfying $r(G)\ge\eta$.\\
$\cW_{r,\eta}$, $\cW_r$
& Graphon limits of sequences in $\cG_{r,\eta}$, 
$\cW_r=\bigcup_{\eta>0}\cW_{r,\eta}$.\\
\end{longtable}
\end{center}

\newcommand\AMS{Amer. Math. Soc.}
\newcommand\Springer{Springer-Verlag}
\newcommand\Wiley{Wiley}

\newcommand\bvol{\textbf}
\newcommand\jour{\emph}
\newcommand\book{\emph}
\newcommand\inbook{\emph}
\def\no#1#2,{\unskip#2, no. #1,} 

\newcommand\arXiv{\arxiv}
\newcommand\arxiv[1]{Preprint: \url{https://arxiv.org/pdf/#1}}


\begin{thebibliography}{99}

\bibitem[ADW25]
{ADW25}
Romain Abraham, Jean-François Delmas and  Julien Weibel.
Probability-graphons: limits of large dense weighted graphs.
\emph{Innov. Graph Theory} \textbf2 (2025), pp.\ 25--117.

\bibitem[AHR26]
{AHR26} Siva Athreya, Frank den Hollander and Adrian R\"ollin. Co-evolving vertex and edge dynamics in dense graphs.  Preprint: \url{https://arxiv.org/pdf/2504.06493v2}, 2026.

\bibitem[AKK99]{AKK99} Sanjeev Arora, David Karger and Marek Karpinski. Polynomial time approximation schemes for dense instances of NP-hard problems. \emph{J.\ Comput.\ System Sci.} {\bf 58} (1999), no.\ 1, pp.~193--210.

\bibitem[Aldo81]{Aldous}
David J.  Aldous.
Representations for partially exchangeable arrays of random variables. 
\jour{J. Multivariate Anal.} {\bf 11} (1981), no.\ 4, pp.\ 581--598.  

\bibitem[AM85]{alon1985lambda1}
Noga Alon and Vitali ~D.~Milman.
$\lambda_1$, isoperimetric inequalities for graphs, and superconcentrators,
\emph{J.\ Combin.\ Theory Ser.\ B} \textbf{38} (1985), no.~1, pp.~73--88.


\bibitem[AS05]{AS05} Noga Alon and Asaf Shapira. Every monotone graph property is
testable. 
\emph{STOC'05: Proceedings of the thirty-seventh annual ACM symposium on Theory of computing}, 2005, pp.~128--137.




\bibitem[BCL10]{BCL:unique} 
Christian~Borgs, Jennifer T.~Chayes \& L\'aszl\'o Lov\'asz,
Moments of two-variable functions and the uniqueness of graph limits.  
\emph{Geom. Funct. Anal.}  {\bf 19} (2010),  no. 6, pp.\ 1597--1619.

\bibitem[BCLSSV06]{BCLSSV} Christian Borgs, Jennifer T. Chayes, L{\'a}szl{\'o} Lov{\'a}sz,
Vera T. S\'os, Bal\'azs Szegedy and Katalin Vesztergombi. 
Graph limits and parameter testing. 
In \emph{STOC '06: Proceedings of the thirty-eighth annual ACM symposium on Theory of Computing}, 2006, pp.~261--270.

\bibitem[BCLSV08]
{BCLSV1}
Christian Borgs, Jennifer T. Chayes, L{\'a}szl{\'o} Lov\'asz, Vera T. S\'os
and Katalin Vesztergombi. 
Convergent sequences of dense graphs I: Subgraph
frequencies, metric properties and testing.
\emph{Adv.\ Math.} {\bf 219} (2008), pp.\ 1801--1851.

\bibitem[BCLSV12]{BCLSV2}
Christian Borgs, Jennifer T. Chayes, L{\'a}szl{\'o} Lov{\'a}sz, Vera T. S\'os
and Katalin Vesztergombi. 
Convergent sequences of dense graphs II. Multiway cuts and statistical
physics. 
\emph{Ann.\ Math.} \textbf{176} (2012), no. 1, pp.\ 151--219.

\bibitem[Berk66]{Berk}
Robert H. Berk,
Limiting behavior of posterior distributions when the model is incorrect.
\emph{Ann.\ Math.\ Statist.} \textbf{37} (1966), no.\ 1, pp.\ 51--58.


\bibitem[BKK12]{BKK12} Marianna Bolla, Tam\'as K\'oi and Andr\'as Kr\'amli. Testability of minimum balanced multiway cut densities. \emph{Discrete Appl.\ Math.} {\bf 160} (2012), pp.\ 1019--1027.

\bibitem[BR09]{BRmetrics}
 B\'ela Bollob\'as and Oliver Riordan.
 Metrics for sparse graphs. 
\emph{Surveys in Combinatorics 2009}, LMS Lecture Notes Series 365, 
Cambridge Univ. Press, 2009, pp.~211--287.
 
\bibitem[Co80]{Cohn} 
Donald L. Cohn.
\emph{Measure Theory},
Birkh\"auser, Boston, 1980.

\bibitem[DGKR16]{Diao}
Peter Diao, Dominique Guillot, Apoorva Khare and Bala Rajaratnam.
Model-free consistency of graph partitioning.
\arxiv{1608.03860}, 2016.

\bibitem[DJ08]{SJ209}
Persi Diaconis and Svante Janson.
Graph limits and exchangeable random graphs.
\emph{Rendiconti di Matematica (7)},
\textbf{28} (2008), pp.\ 33--61.

\bibitem[DS58]{DS}
Nelson Dunford and  Jacob T. Schwartz.
\emph{Linear Operators. I. General Theory}.
Interscience Publishers, New York, 1958.


\bibitem[DT11]{dinh2011finding}
	Thang N.~Dinh and My T.~Thai, {Finding community structure with performance guarantees in scale-free networks},
	\emph{Privacy, Security, Risk and Trust (PASSAT) and 2011 IEEE Third International Conference on Social Computing (SocialCom), 2011 IEEE Third International Conference on},
	{2011}.

\bibitem[Eng89]{Engelking}
Ryszard Engelking.
\emph{General Topology}. 
2nd ed.,
Heldermann Verlag, Berlin, 1989. 

\bibitem[GGR98]{GGR98} Oded Goldreich, Shafi Goldwasser and Dana Ron. Property testing and its connection to learning and approximation. \emph{Journal of ACM} {\bf 45} (1998), no.\ 4, pp.~653--750.

\bibitem[GP26]{GrebikP}
Jan Greb{\'i}k and Oleg Pikhurko.
Large deviation principles for graphon sampling.
\emph{Electron.\ J.\ of Probab.} \textbf{31} (2026), no.\ 59, 37 pp.

\bibitem[Gut13]{Gut}
Allan Gut.
\emph{Probability: A Graduate Course}.
2nd ed., Springer, New York, 2013. 

\bibitem[HHLS20]{van2020limit}
Remco van der Hofstad, Pim van der Hoorn, Nelly Litvak and Clara Stegehuis.
{Limit theorems for assortativity and clustering in null models for scale-free networks.}
\emph{Adv.\ Appl.\ Probab.},
\textbf{52} (2020), no.\ 4, pp.\ 1035--1084.

\bibitem[Hoeff63]{Hoeff63} Wassily Hoeffding. Probability inequalities for
  sums of bounded random variables. \emph{J.\ Amer.\
    Statist.\ Assoc.}
\textbf{58} (1963), no.\ 301, pp.\ {13--30}.


\bibitem[Hoov79]{Hoover}
Douglas N. Hoover.
Relations on Probability Spaces and Arrays of Random Variables.
Preprint, Institute for Advanced Study, Princeton, NJ, 1979.


\bibitem[Jans08]
{SJ213}
Svante Janson,
Connectedness in graph limits. 
\arXiv{0802.3795}

\bibitem[Jans09]{SJ210}
Svante Janson.
Standard representation of multivariate functions on a general probability
space. 
\emph{Electron.\ Comm.\ Probab.} \textbf{14} (2009), no.\ 34,  343--346. 

\bibitem[Jans13]{SJ249}
Svante Janson.
Graphons, cut norm and distance, couplings and rearrangements.
\emph{NYJM Monographs}, \textbf4, 2013.

\bibitem[JLR00]{JLR}  
Svante Janson, Tomasz \L uczak and Andrzej Ruci\'nski.
\emph{Random Graphs}.
Wiley-Interscience, New York, 2000.


\bibitem[Kall02]{Kallenberg}
Olav Kallenberg:
\book{Foundations of Modern Probability.}
2nd ed., Springer, New York, 2002. 


\bibitem[Kall05]{Kallenberg:symmetries}
Olav Kallenberg.
\book{Probabilistic Symmetries and Invariance Principles}. 
Springer, New York, 2005.

\bibitem[KJS22]{Klimm}
Florian Klimm,  Nick S. Jones and  Michael T. Schaub.
Modularity maximization for graphons.
\emph{SIAM J. Appl. Math.} {\bf 82} (2022), no.\ 6, pp.\ 1930--1952.

\bibitem[KM24]{Khetan}
Abhishek Khetan and Mahan Mj.
Cheeger inequalities for graph limits.
\emph{Ann. Inst. Fourier (Grenoble)} {\bf 74} (2024), no.\ 1, pp.\ 257--305.


\bibitem[LMS24]{modexpansion}
Baptiste Louf, Colin McDiarmid and Fiona Skerman.
Modularity and graph expansion.
In \emph{15th Innovations in Theoretical Computer Science Conference
(ITCS 2024)}, Leibniz Int. Proc. Inform., vol.~287,
Schloss Dagstuhl--Leibniz-Zentrum f\"ur Informatik, 2024, pp.~78:1--78:21.


\bibitem[Lov\'a12]{Lovasz}
L\'aszl\'o Lov\'asz.
\emph{Large Networks and Graph Limits}.
American Mathematical Society, Providence, RI, 2012.


\bibitem[MS24]{sampling} 
Colin McDiarmid and Fiona Skerman. Modularity and partially observed graphs. 
Preprint: \url{https://arxiv.org/abs/2112.13190}, 2024.

\bibitem[MS25]{modchapter} 
Colin McDiarmid and Fiona Skerman. Modularity and random graphs. 
Preprint: \url{https://arxiv.org/abs/2509.22066}, 2025.

\bibitem[NG04]{NG04}
Mark E.~J.~Newman and Michelle Girvan.
Finding and evaluating community structure in networks.
\emph{Phys. Rev. E} \textbf{69} (2004), 026113.

\bibitem[POM09]{POM09}
Mason A.~Porter, Jukka-Pekka Onnela and Peter J.~Mucha.
Communities in networks.
\emph{Notices Amer. Math. Soc.} \textbf{56} (2009), no.~9,
pp.~1082--1097, 1164--1166.

\bibitem[Rudi91]{Rudin}
Walter Rudin.
\emph{Functional Analysis}, 2nd ed.
McGraw-Hill, Inc., New York, 1991. 

\bibitem[Ryan02]{Ryan}
Raymond A. Ryan.
\emph{Introduction to Tensor Products of Banach Spaces}.
Springer-Verlag London, Ltd., London, 2002. 

\bibitem[Sker16]{thesis} 
Fiona Skerman. 
Modularity of Networks. Doctoral Thesis, Oxford University, 2016. 



\end{thebibliography}
\end{document}